\documentclass[10pt,a4paper]{article}
\usepackage[utf8]{inputenc}
\usepackage[english]{babel}
\usepackage{multicol,graphicx,color}
\usepackage{pslatex}
\usepackage{authblk}
\usepackage{amsthm}
\usepackage{amsmath}
\usepackage{amssymb}
\usepackage{latexsym}
\usepackage{lscape}
\usepackage{epsfig}
\usepackage{amsfonts}
\usepackage{mathrsfs,bbm}
\usepackage[
hmarginratio={1:1},     
vmarginratio={1:1},     
textwidth=15cm,        
textheight=21cm,
heightrounded,          
]{geometry}

\usepackage{tikz}

\makeatletter

\newtheorem{theorem}{Theorem}[section]

\newtheorem{definition}[theorem]{Definition}

\newtheorem{lemma}[theorem]{Lemma}
\newtheorem{proposition}[theorem]{Proposition}
\newtheorem{remark}[theorem]{Remark}

\theoremstyle{definition} \theoremstyle{remark}
\numberwithin{equation}{section}

\newcommand{\R}{\mathbb{R}}
\newcommand{\N}{\mathbb{N}}
\newcommand{\B}{\mathcal{B}}
\newcommand{\U}{\mathcal{U}}
\newcommand{\K}{\mathcal{K}}

\newcommand{\p}{\mathcal{P}}
\newcommand{\s}{\mathbb{S}}

\newcommand{\la}{\lambda}
\newcommand{\La}{\Lambda}
\newcommand{\Ga}{\Gamma}
\newcommand{\ga}{\gamma}
\newcommand{\ep}{\epsilon}

\begin{document}

\title{Exact number of positive solutions and existence of sign-changing solutions with prescribed mass for NLS on bounded domains
}

\author{Linjie Song\footnote{songlinjie18@mails.ucas.edu.cn, linjie.song@univ-fcomte.fr}} \affil{{\small Universit\'e Marie et Louis Pasteur, CNRS, LmB (UMR 6623), Besan\c{c}on F-25000, France}} 

\author{Wenming Zou\footnote{zou-wm@mail.tsinghua.edu.cn}}
\affil{{\small Tsinghua University, Department of Mathematical Sciences, Beijing 100084, China}}

\date{}
	
\maketitle
\begin{abstract}
	\noindent Given $\mu > 0$, we study the elliptic problem:
	\begin{align*}
		\text{ find } (u,\la) \in H_0^1(\Omega) \times \R \text{ such that } -\Delta u + \la u = |u|^{p-2}u \text{ in } \Omega \text{ and } \int_\Omega|u|^2dx = \mu,
	\end{align*}
    where $\Omega \subset \R^N$ is a bounded domain and $p > 2$ is Sobolev-subcritical. When $p$ is $L^2$-subcritical, i.e. $2 < p < 2 + 4/N$, we show that the problem admits infinitely many sign-changing solutions whose energies are unbounded for every fixed $\mu > 0$. Moreover, we give the limit behavior for both the parameter $\la$ and the energy of the solutions  as  $\mu \to 0^+$ and $\mu \to +\infty$ respectively. Such a multiplicity result also holds when $p$ is $L^2$-critical, i.e. $p = 2 + 4/N$, for each small $\mu > 0$, and we describe precisely what happen when $\mu \to 0^+$.  In the $L^2$-supercritical case, i.e. $2+4/N < p < 2^*$, we find as many sign-changing solutions as we want at the expense of possibly reducing the mass $\mu$. As $\mu$ tends to $0$, the energy of these solutions goes to $0$ and the limit of the parameter $\la$ is a Dirichlet eigenvalue of $-\Delta$ on $\Omega$ multiplying $-1$. When $\Omega = B_1$, the unitary ball, and the nonlinear term is $\tau |u|^{p-2}u$ with $\tau \in [1/2,1]$ fixed, in the $L^2$-supercritical regime, we prove that the problem admits exactly two positive solutions for small $\mu > 0$ and how small $\mu > 0$ must be does not depend on the value of $\tau$. Moreover, sending $\mu$ to $0$ we get that the energy of one positive solution tends to $0$ and the parameter tends to $-\la_1(B_1)$, where $\la_1(B_1)$ is the first Dirichlet eigenvalue of $-\Delta $ on the unit ball $B_1$, while both the energy of the other positive solution and the parameter $\la$ go to infinity uniformly with respect to $\tau$. The result on the exact number of positive solutions not only  has its own interest but it  also is a crucial step in our framework for  searching  sign-changing solutions with large energy and large parameter $\la$. In fact, this result enables us to show that the problem admits a sign-changing solution for small $\mu$ whose energy and the parameter $\la$ tend to infinity as $\mu \to 0^+$.

\medskip

{\small \noindent \textit{\bf Key Words:} Positive solutions; sign-changing solutions; $L^2$ constraint; bounded domain.}

\end{abstract}

\newpage
\tableofcontents

\vskip0.66in

\section{Introduction and main results}

To introduce our questions, let us begin with the NLS with a pure-power nonlinearity set on $\R^N$. For any fixed $\la > 0$, it is well known that the following equation
\begin{align} \label{eqonrn}
	-\Delta u + \la u = |u|^{p-2}u, \quad u \in H^1(\R^N)
\end{align}
has a unique positive solution $u_\la$ (up to translations) whenever $2 < p < 2^*:= 2N/(N-2)^+$, see \cite{Kwong}. Moreover, by using Pohozaev identity, see e.g. \cite[Appendix B]{Willem}, it can be checked that \eqref{eqonrn} admits no solution when $\la \leq 0$. Due to scaling invariance, direct computation yields
\begin{align*}
	u_\la = \la^{\frac1{p-2}}u_1(\sqrt{\la}x) \quad \text{and} \quad \int_{\R^N}|u_\la|^2dx = \la^{\frac2{p-2}-\frac N2}\int_{\R^N}|u_1|^2dx.
\end{align*}
This indicates that whenever $p \in (2,2^*)$ and $p \neq 2 + 4/N$, for any $\mu > 0$, there exists a \emph{unique} positive function $u_\mu$ satisfying $\int_{\R^N}|u_\mu|^2dx = \mu$, with a parameter $\la_\mu$, solving \eqref{eqonrn} with $\la = \la_\mu$. Define the energy functional
\begin{align*}
	E(u,\R^N) = \frac12\int_{\R^N}|\nabla u|^2dx - \frac1p\int_{\R^N}|u|^pdx.
\end{align*}
In the \emph{mass subcritical} case, that is $$2 < p < p_c := 2 + 4/N,$$ we see that $E(u_\mu,\R^N) \to 0$ and $\la_\mu \to 0^+$ as $\mu \to 0^+$ while $E(u_\mu,\R^N) \to -\infty$ and $\la_\mu \to +\infty$ as $\mu \to +\infty$. On the contrary, in the \emph{mass supercritical} case, that is, $p_c < p <2^*$, it holds that $E(u_\mu,\R^N) \to +\infty$ and $\la_\mu \to +\infty$ as $\mu \to 0^+$ while $E(u_\mu,\R^N) \to 0$ and $\la_\mu \to 0^+$ as $\mu \to +\infty$. 

\vskip0.1in

The situation on a {\bf bounded domain}  is quite different. We focus on the case of the unitary ball. For any fixed $\la > -\la_1(B_1)$ where $B_1$ is the unitary ball centralized at $0$ in $\R^N$ and $\la_1(B_1)$ is the first eigenvalue of $-\Delta$ with Dirichlet boundary condition, as well known, the following equation
\begin{align} \label{eqonB1}
	-\Delta u + \la u = |u|^{p-2}u \quad \text{in } B_1, \quad u \in H_0^1(B_1)
\end{align}
has a unique positive solution $u_\la$ whenever $2 < p < 2^*$, see \cite{GNN,Korman,Kwong,KL,Zhang}. Moreover, \eqref{eqonB1} admits no positive solution when $\la \leq -\la_1(B_1)$, see e.g. \cite[Theorem 1.19]{Willem}. Though scaling invariance is lost in this case, the authors in \cite{NTV} showed that when $2 < p < p_c$, \eqref{eqonB1} has a unique positive solution with $L^2$-norm $\mu$ for any fixed $\mu > 0$; when $p = p_c$, there exists $\bar \mu > 0$ such that \eqref{eqonB1} admits a unique positive solution with $L^2$-norm $\mu$ for $0 < \mu < \bar \mu$ and no positive solution with $L^2$-norm $\mu \geq \bar \mu$. In striking contrast, in the supercritical regime, readers can see \cite{Song} that $u_\la$ forms a $C^1$ curve in $H_0^1(B_1)$ parameterized by $\la$ and
$$\lim_{\la \to -\la_1(B_1)}\int_{B_1}|u_\la|^2dx = 0 = \lim_{\la \to +\infty}\int_{B_1}|u_\la|^2dx.$$
Thus, for small $\mu$, \eqref{eqonB1} has at least two positive solutions $u_{\mu,1}, u_{\mu,2}$ having prescribed $L^2$-norm $\mu$ with different parameters $\la_{\mu,1} < \la_{\mu,2}$ and for sufficiently large $\mu$, there is no positive solution whose $L^2$-norm is $\mu$, see also \cite{NTV} for such results. (The multiplicity result also holds for more general bounded domains, see \cite{CLY,PVY,SZ2024}.) In addition, as $\mu \to 0^+$ we have $E(u_{\mu,1}) \to 0$ and $\la_{\mu,1} \to - \la_1(B_1)$ while $E(u_{\mu,2}) \to +\infty$ and $\la_{\mu,2} \to +\infty$, where
\begin{align*}
	E(u) = \frac12 \int_{B_1}|\nabla u|^2dx - \frac{1}p\int_{B_1} |u|^pdx.
\end{align*}	
Inspired by the  above results regarding positive solutions, we guess that for small $\mu > 0$ one may find \emph{two  sign-changing} solutions $u_{\text{sc},\mu}, \tilde u_{\text{sc},\mu}$ for \eqref{eqonB1} with $\la = \la_{\text{sc},\mu}, \tilde \la_{\text{sc},\mu}$ satisfying $$\int_{B_1}|u_{\text{sc},\mu}|^2dx = \int_{B_1}|\tilde u_{\text{sc},\mu}|^2dx = \mu.$$
Moreover,
$$ \lim_{\mu \to 0^+} E(u_{\text{sc},\mu})= 0,\quad \lim_{\mu \to 0^+} \Big(-\la_{\text{sc},\mu}\Big) =  \text{some Dirichlet eigenvalue of}  -\Delta;$$
$$ \lim_{\mu \to 0^+} E(\tilde u_{\text{sc},\mu}) = +\infty,\quad \lim_{\mu \to 0^+} \tilde \la_{\text{sc},\mu} = +\infty.$$


\vskip0.1in
 
	In this paper, we are concerned with the exact number of positive solutions and existence of sign-changing solutions with prescribed mass for NLS on \textbf{bounded domains}. More precisely, we consider the following equation
	\begin{align} \label{eqonbounddomain}
		-\Delta u + \la u = |u|^{p-2}u \quad \text{in } \Omega, \quad u \in H_0^1(\Omega), \quad \int_{\Omega}|u|^2dx = \mu.
	\end{align}
    Here $\Omega \subset \R^N$ is bounded, $\mu > 0$ is a prescribed constant, and $\la \in \R$ is an undetermined Lagrange multiplier. Let $E(\cdot,\Omega)$ be the corresponding energy functional defined by
    \begin{align*}
    	E(u,\Omega) = \frac12\int_\Omega |\nabla u|^2dx - \frac1p \int_\Omega |u|^pdx.
    \end{align*}
    Note that $E(u,B_1)$ is denoted by $E(u)$ as defined before. Based on previous observations in the first two paragraphs, we aim to solve the following problems.

    \vskip0.1in

    Firstly, for a general $C^1$ bounded domain, we study the existence and multiplicity of sign-changing solutions of \eqref{eqonbounddomain} with small energies and investigate the limit behaviors of their energies and the Lagrange multipliers. Our research includes  mass-subcritical, mass-critical  and mass-supercritical cases. Precisely, we will give a positive answer to the following question.
    \begin{itemize}
    	\item[{\bf(Q1)}:]
    	\begin{itemize}
    		\item Find infinitely many sign-changing solutions of \eqref{eqonbounddomain} for all $\mu > 0$ when $2 < p < p_c$, whose energies tend to $0$ and the Lagrange multipliers tend to Dirichlet eigenvalues of $-\Delta$ multiplying $-1$ as $\mu \to 0^+$. The limit behaviors as $\mu \to +\infty$ will  be investigated.
    		\item Find infinitely many sign-changing solutions of \eqref{eqonbounddomain} for small $\mu > 0$ when $p = p_c$, whose energies tend to $0$ and Lagrange multipliers tend to Dirichlet eigenvalues of $-\Delta$ by multiplying $-1$ as $\mu \to 0^+$.
    		\item Find multiple sign-changing solutions of \eqref{eqonbounddomain} for small $\mu > 0$ when $p_c < p < 2^*$, whose energies tend to $0$ and Lagrange multipliers tend to Dirichlet eigenvalues of $-\Delta$ by multiplying $-1$ as $\mu \to 0^+$.
    	\end{itemize}
    \end{itemize}

Secondly, we focus on the mass-supercritical case. When $\Omega = B_1$ and $\mu$ is small, we study the exact number of positive solutions and the existence of a sign-changing solution with large energy and with large Lagrange multiplier of \eqref{eqonbounddomain}. In particular, we will provide affirmative answers to the following questions.
\begin{itemize}
	\item[{\bf(Q2)}:] Characterize the number of positive solutions of \eqref{eqonbounddomain} for small $\mu > 0$ when $\Omega = B_1$ and $p_c < p < 2^*$.
	\item[{\bf(Q3)}:] Find a sign-changing solution of \eqref{eqonbounddomain} for small $\mu > 0$ when $\Omega = B_1$ and $p_c < p < 2^*$, whose energy and Lagrange multiplier tend to infinity as $\mu \to 0^+$.
\end{itemize}
We point out that (Q3) can be proposed for a general bounded domain. However, our framework to solve it depends on the positive answer to (Q2), and so we just consider the case of the unitary ball in this paper.

\vskip0.1in
The following result answers (Q1). 


\begin{theorem} \label{thmsmall}
	Let $\Omega \subset \R^N$ be a bounded $C^1$ domain.
	\begin{itemize}
		\item[(i)] When $2 < p < p_c$, for any $\mu > 0$, \eqref{eqonbounddomain} has infinitely many sign-changing solutions $u_{\mu,1}$, $u_{\mu,2}$, $\cdots$, $u_{\mu,j}$, $\cdots$, with $\la = \la_{\mu,1}$, $\la_{\mu,2}$, $\cdots$, $\la_{\mu,j}$, $\cdots$.
		For fixed $\mu > 0$, $E(u_{\mu,j},\Omega) \to +\infty$ as $j \to +\infty$. For fixed $j$, $E(u_{\mu,j},\Omega) \to 0$ as $\mu \to 0^+$ and $E(u_{\mu,j},\Omega) \to -\infty$ as $\mu \to +\infty$. Moreover, the parameter $\displaystyle \la_{\mu,j} \to +\infty$ as $\mu \to +\infty$; the value $\displaystyle -\lim_{\mu \to 0^+}\la_{\mu,j}$ is an eigenvalue of $-\Delta$ on $\Omega$ with Dirichlet boundary condition, and particularly, $\displaystyle -\lim_{\mu \to 0^+}\la_{\mu,1} = \la_2(\Omega)$, the second Dirichlet eigenvalue.
		\item[(ii)] When $p = p_c$, there exists $\bar \mu > 0$ such that for any $0 < \mu < \bar \mu$, \eqref{eqonbounddomain} has infinitely many sign-changing solutions $u_{\mu,1}$, $u_{\mu,2}$, $\cdots$, $u_{\mu,j}$, $\cdots$, with $\la = \la_{\mu,1}$, $\la_{\mu,2}$, $\cdots$, $\la_{\mu,j}$, $\cdots$. For fixed $\mu > 0$, $E(u_{\mu,j},\Omega) \to +\infty$ as $j \to +\infty$. For fixed $j$, $E(u_{\mu,j},\Omega) \to 0$ as $\mu \to 0^+$. Moreover, the value $-\lim_{\mu \to 0^+}\la_{\mu,j}$ is an eigenvalue of $-\Delta$ on $\Omega$ with Dirichlet boundary condition, and particularly, $\displaystyle -\lim_{\mu \to 0^+}\la_{\mu,1} = \la_2(\Omega)$, the second Dirichlet eigenvalue.
		\item[(iii)] When $p_c < p < 2^*$, given a positive integer $j$, there exists $\mu^*_{p,j} > 0$ such that for any $0 < \mu < \mu^*_{p,j}$, \eqref{eqonbounddomain} has $j$ different sign-changing solutions $u_{\mu,1}$, $u_{\mu,2}$, $\cdots$, $u_{\mu,j}$, with $\la = \la_{\mu,1}$, $\la_{\mu,2}$, $\cdots$, $\la_{\mu,j}$. For fixed $i \in \{1,2,\cdots,j\}$, $E(u_{\mu,i},\Omega) \to 0$ as $\mu \to 0^+$. Moreover, the value $\displaystyle -\lim_{\mu \to 0^+}\la_{\mu,i}$ is an eigenvalue of $-\Delta$ on $\Omega$ with Dirichlet boundary condition, and particularly, $\displaystyle -\lim_{\mu \to 0^+}\la_{\mu,1} = \la_2(\Omega)$, the second Dirichlet eigenvalue.
	\end{itemize}
\end{theorem}

In the Sobolev critical case ($p = 2^*$), a similar result to Theorem \ref{thmsmall} (iii) was proved in \cite{SZ} by  introducing a new kind of linking. In this paper we use a different approach based on genus theory to address all supercritical, critical and subcritical cases.  We note that in \cite{PV}, the authors studied the existence of solutions with bounded Morse index. For any positive integer $k$, they found a solution having prescribed mass $\mu$ whose Morse index is no more than $k$. In particular,  when $2 < p < p_c$ 
 then $\mu$ may be any positive real number;  but when $p_c \le p < 2^*$, then $\mu$ must small enough. They also showed that the requirement of $\mu > 0$ being small enough  is necessary in dealing with the critical and supercritical cases by a detailed blow-up analysis of the sequences of solutions with bounded Morse index. However, one does not know whether the solution obtained in \cite{PV} is sign-changing or not in general, unless in the regime of $\mu$ such that no positive solution exists. Compared to the results in \cite{PV}, the current  Theorem \ref{thmsmall} has improvements in three aspects:

\begin{itemize}
	\item The solutions that we obtain are sign-changing.
	\item  {Only one solution was found in \cite{PV} while we establish the multiplicity of solutions. Particularly, in the mass subcritical and critical cases, we find infinitely many sign-changing solutions.} It seems that this is the first result on the existence of infinitely many sign-changing solutions  {with prescribed mass} for NLS set on bounded domains.
	\item The asymptotic behaviors of the energy and of Lagrange multiplier with respect to the parameter $\mu$ are provided.
\end{itemize}

As we know, searching for a sign-changing solution is much more difficult than just finding weak solutions without any sign information. It is worth mentioning that there are many papers and results on free sign-changing solutions (without the mass constraint), but few results of \eqref{eqonbounddomain} and of other equations with $L^2$ constraint. Recently the authors of \cite{CDDS} established  the existence of one sign-changing solution with prescribed mass by developing an action approach. In addition, descending flow technique was developed on a manifold setting to search for multiple sign-changing solutions, see \cite{JS,SZ}. Our work will further enrich the methods and results in this direction.

\vskip0.1in

The Morse index information will be used in the proof of Theorem \ref{thmsmall}. As a foresight,  the boundedness of the  Morse index  is useful for studying the asymptotic behaviors of the energy and of the Lagrange multiplier via blow-up analyses. In this respect, the $C^1$ regularity assumption on $\partial \Omega$ simplifies the treatment of possible concentration phenomena towards the boundary. In \cite{PV}, the estimate of the Morse index is a direct consequence of the  genus (see e.g. \cite{BL} and \cite{Sol}). The situation is much more complex when we need a sign-changing  information.  {To overcome these difficulties,   we provide an abstract tool to obtain Palais-Smale sequences with \emph{location} information and with \emph{approximating Morse index}, and the location information ensures that the solution obtained is sign-changing. Such an abstract tool and framework  have their own interest and we believe they will have more applications.}



\vskip0.1in
Next, we give an answer to (Q2). In order to search for a sign-changing solution with large energy and large Lagrange multiplier, we prepare  a more general result having a parameter $\tau$ in the nonlinear term.

\begin{theorem} \label{thmunique}
	Let $p_c < p < 2^*$. Given $\tau \in [1/2,1]$, there exists $\mu^*_{p,0} > 0$ independent of $\tau$ such that, for any $0 < \mu < \mu^*_{p,0}$, the following equation
	\begin{align} \label{eqequationtau}
		-\Delta u + \la u = \tau|u|^{p-2}u \quad \text{in } B_1, \quad u \in H_0^1(B_1), \quad \int_{B_1}|u|^2dx = \mu,
	\end{align}
	has exactly two positive solutions $u_{\mu,\tau}, \tilde u_{\mu,\tau}$ with $\la=\la_{\mu,\tau},\tilde\la_{ \mu,\tau}$ respectively. Moreover, let
	\begin{align*}
		E_\tau(u) = \frac12 \int_{B_1}|\nabla u|^2dx - \frac{\tau}p\int_{B_1} |u|^pdx,
	\end{align*}
then the following conclusions hold true:
	\begin{itemize}
		\item[(i)] $\displaystyle \lim_{\mu \to 0^+} E_\tau(u_{\mu,\tau}) = 0$ and $\displaystyle \lim_{\mu \to 0^+} \la_{\mu,\tau} = -\la_1(B_1)$ uniformly with respect to $\tau \in [1/2,1]$;
		\item[(ii)] $\displaystyle \lim_{\mu \to 0^+}  E_\tau(\tilde u_{\mu,\tau})= +\infty$ and $\displaystyle \lim_{\mu \to 0^+}  \tilde\la_{\mu,\tau} = +\infty$ uniformly with respect to $\tau \in [1/2,1]$.
	\end{itemize}	
\end{theorem}

\begin{remark}
 {By a scaling argument, we can see that for any $\mu > 0$, there exists $\tau^* > 0$ such that \eqref{eqequationtau} admits exactly two positive solutions when $0 < \tau < \tau^*$.}
\end{remark}

Theorem \ref{thmunique} not only has its own interest but is also a crucial step in our framework to search for a sign-changing  normalized  solution (prescribed mass) with large energy and large Lagrange multiplier.




\vskip0.1in
Finally, the following result give an affirmative answer to (Q3).

\begin{theorem} \label{thmlarge}
	Let $p_c < p < 2^*$ and $\Omega = B_1$. There exists $\mu^{**}_p > 0$ such that, for any $0 < \mu < \mu^{**}_p$, \eqref{eqonbounddomain} has a sign-changing solution $u_{\text{sc},\mu}$ with $\la=\la_{\text{sc},\mu}$. Moreover,   we have 
$$\lim_{\mu \to 0^+}E(u_{\text{sc},\mu}) = +\infty, \quad  \lim_{\mu \to 0^+}\la_{\text{sc},\mu} = +\infty.$$
\end{theorem}

In order to prove Theorem \ref{thmlarge}, we use a variational method and study the sign-changing critical points for the energy functional constrained to the $L^2$-sphere
\begin{equation*}
	S_\mu(\Omega)
	:= \bigl\{u\in H_0^1(\Omega):
	\int_{\Omega}|u|^2 dx=\mu \bigr\}.
\end{equation*}
When $\Omega = B_1$, we use notation $S_\mu$ for simplicity. A key idea is to construct the min-max level by generalizing the mountain pass structure via a saddle-point theorem. To address the bounded issue of Palais-Smale sequences,  inspired by \cite{BCJN} and \cite{Louis}, we consider approximating problems firstly and then come back to the origin problem. More precisely, we aim to find two odd functions $\gamma_{i,k}: \s^{k-1} \to S_\mu(\Omega)$ where $i = 0,1$, such that the set
\begin{align*}
	\Gamma_k := \big\{\gamma \in C([0,1]\times \s^{k-1}, S_\mu(\Omega)): \forall t \in [0,1], \gamma(t,\cdot) \text{ is odd}, \gamma(0,\cdot) = \gamma_{0,k}, \text{ and } \gamma(1,\cdot) = \gamma_{1,k}\big\}
\end{align*}
is not empty and is independent of $\tau$, and
\begin{align*}
	c_\tau^k := \inf_{\gamma \in \Gamma_k}\sup_{(t,s) \in [0,1]\times\s^{k-1}}E_\tau(\gamma(t,s),\Omega) > \max_{s\in \s^{k-1}}\max\big\{ E_\tau(\gamma_{0,k}(s),\Omega), E_\tau(\gamma_{1,k}(s),\Omega)\big\},
\end{align*}
where
\begin{align*}
	E_\tau(u,\Omega) = \frac12 \int_{\Omega}|\nabla u|^2dx - \frac{\tau}p\int_{\Omega} |u|^pdx, \quad \tau \in [\frac12,1].
\end{align*}
When $k = 1$, $c_\tau^1$ is exactly the mountain pass level. This min-max structure was pointed out in \cite[Remark 4.5]{PV}, and later in \cite{CGJT}, it was used to search for infinitely many solutions with prescribed mass of $L^2$-supercritical NLS equations on noncompact metric graphs with localized nonlinearities.

\vskip0.1in

For nonlinear elliptic equations or Schr\"odinger equations set on bounded domains or on compact metric graphs, it is unknown how to use this min-max structure to find a solution different from the one obtained by mountain pass geometry. In fact, if $c_1^1 < c_1^2$, we obtain a new solution immediately but we do not know whether $c_1^1 < c_1^2$ holds true or not. Observing that the solution at mountain pass level can be taken as a positive function, our strategy is to search for a sign-changing one: Using \cite[Theorem 1.12]{BCJN}, which is stated as Theorem \ref{Objective1} in this paper, we can prove that for almost every $\tau \in [1/2,1]$, $c_\tau^k$ is indeed a critical value and there is a constrained critical point $u_\tau^k$ whose Morse index is no more than $k+2$. Moreover, when $k = 1$, then $u_\tau^1$ is positive and its Morse index is no more than $2$. Using some new arguments, we prove that if $c_\tau^1 = c_\tau^2$, then $u_\tau^2$ is sign-changing and its Morse index is no more than $2$, see Proposition \ref{propc1=c2} below. If $c_\tau^1 < c_\tau^2$, we know $u_\tau^2$ is different from $u_\tau^1$. When $\Omega = B_1$, using the fact that $c_\tau^k$, depending on $\mu$, tends to $+\infty$ uniformly with respect to $\tau$ as $\mu \to 0^+$, and by Theorem \ref{thmunique}, for sufficiently small $\mu > 0$, we get that $u_\tau^2$ is sign-changing. Thus, regardless of $c_\tau^1 < c_\tau^2$ or not, $u_\tau^2$ is sign-changing and its Morse index is no more than $4$. Since $\tau$ is almost everywhere, we can take a sequence $\{\tau_n\}$ with $\tau_n \to 1^-$. Then, like \cite{CJS}, using blow-up analyses we can obtain that $\{u_{\tau_n}^2\}$ is bounded in $H_0^1(B_1)$. Next it is not difficult to get a constrained critical point $u_1^2$ of $E_1$ at level $c_1^2$. We cannot conclude $u_1^2 \neq u_1^1$ immediately since sign-changing functions may converge to a positive function. To complete the proof, we develop descending flow techniques. In particular, we shall introduce the positive and negative cones
\begin{align}
	\pm \p = \bigl\{u \in H_0^1(\Omega): \pm u \geq 0\bigr\}.
\end{align}
Moreover, for $\delta > 0$ we set
\begin{align}
	& (\pm \p)_\delta := \bigl\{u \in H_0^1(\Omega): \text{dist}(u,\pm \p) < \delta\bigr\}, \\
	& D^*_\delta = \overline{\p_\delta} \cup \overline{(-\p)_\delta}, \quad S^*(\delta) = H_0^1(\Omega) \backslash D^*_\delta,
\end{align}
where $\text{dist}$ denotes the distance associated with the $H_0^1$-norm. For a small $\delta > 0$ independent of $\tau_n$, we will use descending flow techniques developed in Section \ref{secinv} to show that $u_{\tau_n}^2 \in S^*(\delta)$. Thus $u_1^2 \in \overline{S^*(\delta)}$ is sign-changing and different from the positive function $u_1^1$. Descending flow techniques developed in Section \ref{secinv} are also crucial in proving Theorem \ref{thmsmall}.

\vskip0.1in

Flow invariance method was recently developed in \cite{JS} in a \emph{manifold} setting in which the constrained problem for \eqref{eqonbounddomain} is included. Particularly, Proposition 4.5, Lemma 4.7, and Corollary 4.6 in \cite{JS} constitute a general framework to prove that some subset is invariant with respect to a flow on a manifold. In practice, the gradient was used in \cite{JS} to construct the descending flow. The gradient $\nabla E|_{S_\mu(\Omega)}(u,\Omega)$ of $E$ as a constrained functional to $S_\mu(\Omega)$, at the point $u \in S_\mu(\Omega)$, belongs to $T_uS_\mu(\Omega)$, the tangent space of $S_\mu(\Omega)$ at $u$, and is given by $$\nabla E|_{S_\mu(\Omega)}(u,\Omega) = u - (-\Delta)^{-1}(|u|^{p-2}u + \la_u u).$$ Here $\la_u$ is determined by the condition $\nabla E|_{S_\mu(\Omega)}(u,\Omega) \in T_uS_\mu(\Omega)$. Different from the unconstrained case where $\la_u$ is independent of $u$, to determine the sign of $\la_u$ is important in the $L^2$-constrained problem. In \cite{JS} we used conditions on $\mu$ to check that $\la_u \geq 0$ for all $u$ in a suitable bounded set, which guarantees that the assumptions in \cite[Lemma 4.7]{JS} hold true for the gradient $\nabla E|_{S_\mu(\Omega)}$. Let us remark that, if $\nabla E|_{S_\mu(\Omega)}(u,\Omega) = 0$, it is clear that $u$ solves \eqref{eqonbounddomain} with $\la = -\la_u$. When searching for solutions with large Lagrange multiplier, the case $\la_u < 0$ for some $u$ cannot be avoided. To solve this difficult issue, we succeed to find a pseudogradient operator with well properties and define the descending flow with respect to this operator, rather than the gradient. The natural idea is to consider the operator $V = Id - G$ where $G(u)$ has a desired form as $(-\Delta + \bar\la)^{-1}(|u|^{p-2}u + \omega u)$ where $\omega \in \R$ is a parameter depending on $u$ while $\bar\la$ is independent of $u$. Further, we notice that $G(u)$ is the unique solution of \eqref{eqequofw1} studied in Proposition \ref{propuniquesolutionw} appeared in the following text, and $\omega$ is the corresponding Lagrange multiplier. For all $u$ in a bounded set, we can prove that $\omega \geq 0$ when $\bar \la$ is sufficiently large.
See Section \ref{secinv} for more details.

	\begin{remark}
{All the techniques to prove Theorem \ref{thmlarge}, except for Theorem \ref{thmunique}, are prepared for a general bounded domain $\Omega$. This means that as long as the result in Theorem \ref{thmunique} is proved for $\Omega$, one can get the conclusion in Theorem \ref{thmlarge} for the same domain immediately. Moreover, we think that the result in Theorem \ref{thmunique}  can indeed be extended to more general bounded domains, such as a convex one with certain conditions. However, this will involve complicated techniques and is far beyond our initial motivation, that is finding an approach to search for a sign-changing prescribed mass solution with large energy. Therefore,  we keep  it as a future research topic.}
\end{remark}

\vskip0.3in

The rest part of this paper is organized as follows. In Section \ref{secpre} we give some preliminary results. In section \ref{secLagrange multiplier}, we study the relationship between the boundedness of Lagrange multipliers and energies of solutions with bounded Morse index, which is useful in investigating the limit behaviors with respect to $\mu$. In Section \ref{secpositionsolu} we provide the proof of Theorem \ref{thmunique}. In Section \ref{secinv}, we develop descending flow techniques and flow invariance arguments, which are used for both the proof of Theorem \ref{thmsmall} and the proof of Theorem \ref{thmlarge}. In Section \ref{sectool} we establish a powerful tool to obtain bounded Palais-Smale sequences with location information and with Morse type information for some constrained functional having minimax geometry, which will be used to search for sign-changing critical points with estimates of Morse index for $L^2$-constrained functionals. In Section \ref{secgenus}, we search for sign-changing prescribed mass solutions by using genus and complete the proof of Theorem \ref{thmsmall}. In Section \ref{secreview}, we   make further development to prepare the proof of Theorem \ref{thmlarge}. Finally in Section \ref{secappropro} we succeed in finding a sign-changing solution with large energy via generalizing the mountain-pass geometry and prove Theorem \ref{thmlarge}.

\bigskip

\noindent \textbf{Notations}

\begin{itemize}
	\item $\Omega \subset \R^N$ is a bounded $C^1$ domain. $|\Omega|$ means the Lebesgue measure of $\Omega$.
	\item For $y \in \R^N$ and $R > 0$, $B_R(y) := \{x \in \R^N: |x - y| < R\}$, we denote $B_R = B_R(0)$ for simplicity.
	\item For $q \ge 1$, $L^q = L^q(\Omega)$ denotes the space of real-valued measurable functions $u$ on $\Omega$ satisfying $\int_\Omega|u|^qdx < \infty$ whose norm is defined by $\|u\|_{L^q} := (\int_\Omega|u|^qdx)^{1/q}$.
	\item $L^\infty = L^\infty(\Omega)$ denotes the space of real-valued essentially bounded measurable functions
	$u$ on $\Omega$ whose norm is defined by $\|u\|_{L^\infty} = \text{esssup}_{x \in \Omega}|u(x)|$.
	\item We regard $L^2$ as a Hilbert space over $\R$ by the inner product $\langle u,v \rangle_{L^2} := \int_\Omega uv dx$.
	\item The set $H_0^1 = H_0^1(\Omega)$ stands for the space of real-valued measurable functions $u$ of the Sobolev space of order $1$ whose norm is defined by $\|u\| := (\int_\Omega |\nabla u|^2dx)^{\frac12}$. We regard $H_0^1$ as a Hilbert space over $\R$ by the inner product $\langle u,v \rangle := \int_\Omega \nabla u \nabla v dx$.
	\item $0 < \la_1(\Omega) < \la_2(\Omega) \leq \cdots \leq \la_k(\Omega) \le \cdots$ are the eigenvalues of $-\Delta$ on $\Omega$ with Dirichlet boundary condition, and the corresponding unitized eigenfunctions  are denoted by $\phi_1, \phi_2, \cdots, \phi_k, \cdots$.
 	\item $E(u,\Omega) = \frac12\int_\Omega |\nabla u|^2dx - \frac1p \int_\Omega |u|^pdx.$ When $\Omega = B_1$, we denote $E(u) = E(u,B_1)$.
 	\item Given $\tau \in [\frac12,1]$, $E_\tau(u,\Omega) = \frac12 \int_{\Omega}|\nabla u|^2dx - \frac{\tau}p\int_{\Omega} |u|^pdx$. Particularly, $E_1(u,\Omega) = E(u,\Omega)$. When $\Omega = B_1$, we denote $E_\tau(u) = E_\tau(u,B_1)$.
 	\item For $\mu > 0$, $S_\mu(\Omega) = \bigl\{u \in H_0^1(\Omega): \int_\Omega |u|^2dx = \mu\bigr\}$. When $\Omega = B_1$, we denote $S_\mu = S_\mu(B_1)$.
 	\item $\pm \p = \bigl\{u \in H_0^1(\Omega): \pm u \geq 0\bigr\}$.
 	\item For $\delta > 0$, $(\pm \p)_\delta := \bigl\{u \in H_0^1(\Omega): \text{dist}(u,\pm\p) < \delta\bigr\}$, $D^*_\delta = \overline{\p_\delta} \cup \overline{(-\p)_\delta}$, $S^*(\delta) = H_0^1(\Omega) \backslash D^*_\delta$.
 	\item $C_{p,N}$ is the best constant in Gagliardo-Nirenberg inequality \eqref{eqgninequality}, and $\ga_p = N(\frac12 -\frac1p)$.
 	\item $\mathcal{S}_p$ is the best constant of Sobolev inequality $\|u\|_{L^p} \le \mathcal{S}_p\|u\|$.
 	\item $2^* := 2N/(N-2)^+$ is the Sobolev critical exponent, that is, $2^* = \infty$ if $N =1,2$ and $2^* = 2N/(N-2)$ if $N \ge 3$.
 	\item $p_c := 2 + 4/N$ is the $L^2$-critical exponent.
 	\item For $\rho > 0$, $\B_\rho^\mu = \bigl\{u \in S_\mu(\Omega): \int_\Omega|\nabla u|^2dx < \rho\bigr\}$, $\U_\rho^\mu = \bigl\{u \in S_\mu(\Omega): \int_\Omega|\nabla u|^2dx = \rho\bigr\}$.
 	\item For a real-valued function $u$, $u^+ := \max\{u,0\}$ and $u^- := \min\{u,0\}$.
\end{itemize}

\section{Preliminary} \label{secpre}

Firstly, we recall the celebrated Gagliardo-Nirenberg inequality. For any $u \in H^1(\R^N)$, it holds that
\begin{align} \label{eqgninequality}
	\|u\|_{L^p(\R^N)} \le C_{p,N}\|\nabla u\|_{L^2(\R^N)}^{\ga_p} \|u\|_{L^2(\R^N)}^{1-\ga_p},
\end{align}
where $\ga_p = N(\frac12 -\frac1p)$. Observe that $\ga_p \in (0,1)$ for $2 < p < 2^*$ and that
\begin{align*}
	p\ga_p
		\begin{cases}
			\in (0,2)  & \text{if } 2 < p < p_c, \\
			= 2  & \text{if } p = p_c, \\
			\in (2,2^*) & \text{if } p_c < p < 2^*.
		\end{cases}
\end{align*}
For any $u \in H_0^1(\Omega)$, it is natural to extend $u$ by defining $u(x) = 0$ for $x \notin \Omega$. Then we get
\begin{align} \label{eqgndomain}
	\|u\|_{L^p(\Omega)} \le C_{p,N}\|\nabla u\|_{L^2(\Omega)}^{\ga_p} \|u\|_{L^2(\Omega)}^{1-\ga_p}.
\end{align}

Next, we recall the uniqueness result for positive solution of \eqref{eqonB1}. The parameter $\tau$ is mentioned for convenience to use in Section \ref{secpositionsolu}.

\begin{proposition} \label{propuniquefixlaB1}
	Let $2 < p < 2^*$. Given $\tau \in [1/2,1]$, the following equation
	\begin{align} \label{eqwithfixlaB1}
		-\Delta u + \la u = \tau |u|^{p-1}u \quad \text{in } B_1, \quad u > 0 \quad \text{in } B_1, \quad u \in H_0^1(B_1),
	\end{align}
	has a unique, radial solution, denoted by $u_{\tau,\la}$, for any fixed $\la > -\la_1(\Omega)$.
\end{proposition}

\begin{proof}
	The existence easily follows from the mountain pass lemma. The radial symmetry of positive solutions is a direct consequence of \cite{GNN}. The uniqueness in the case $\la > 0$ was proved by Kwong \cite{Kwong} for $N \geq 2$. For $\la \in (-\la_1(B_1),0)$ the uniqueness in dimension $N \ge 3$ was proved by Kwong and Li \cite{KL} (see also \cite{Zhang}) whereas in dimension $N = 2$ it was proved by Korman \cite{Korman}. The case $\la = 0$ is treated in Section 2.8 of \cite{GNN}.
\end{proof}

In the following we make some preparations in search for sign-changing solutions.

\begin{lemma} \label{lemcapnonemp}
	Let
	\begin{align} \label{eqconondelta0}
		0 < \delta < \sqrt{\frac{\la_1(\Omega)}{2}\mu}.
	\end{align}
	Then $\overline{\p_\delta} \cap \overline{(-\p)_\delta} \cap S_\mu(\Omega) = \emptyset$.
\end{lemma}

\begin{proof}
	If $u \in \overline{\p_\delta}$, then
	$$
	\|u^-\|_{L^2}^2 = \inf_{w \in \p}\|u - w\|_{L^2}^2 \leq \la_1(\Omega)^{-1}\inf_{w \in \p}\|u - w\|^2 \leq \la_1(\Omega)^{-1}\delta^2.
	$$
	Similarly, if $u \in \overline{-\p_\delta}$, then $\|u^+\|_{L^2}^2 \leq \la_1(\Omega)^{-1}\delta^2$. Thus, for any $u \in \overline{\p_\delta} \cap \overline{(-\p)_\delta}$, using \eqref{eqconondelta0} we have
	$$
	\|u\|_{L^2}^2 \leq 2\la_1(\Omega)^{-1}\delta^2 < \mu
	$$
	yielding that $\overline{\p_\delta} \cap \overline{(-\p)_\delta} \cap S_\mu(\Omega) = \emptyset$.
\end{proof}

Next, we introduce the definition of Morse index.

\begin{definition} \label{defmorseind}
	For any solution $U \in H^1_0(\Omega)$ of
	\begin{equation}
		\begin{cases}
			-\Delta U + \lambda U = \tau |U|^{p-2}U
			& \text{in } \Omega, \\[1.5\jot]
			\displaystyle
			U(x) = 0
			&\text{on } \partial \Omega,
		\end{cases}
	\end{equation}
	with $\la, \tau \in \R$, we consider
	\begin{equation*}
		Q(\varphi; U, \Omega)
		:= \int_{\Omega} \bigl(|\nabla \varphi|^2
		+ (\lambda- (p-1)\tau|U|^{p-2}) \varphi^2\bigr) dx, \quad \forall \varphi \in H_0^1(\Omega).
	\end{equation*}
	The Morse index of $U$, denoted by $m(U)$, is the maximal dimension of a subspace $W \subset H_0^1(\Omega)$ such that $Q(\varphi; U, \Omega) <0$ for all $\varphi \in W \backslash \{0\}$.
\end{definition}

\begin{definition}\label{def D}
	Let $\phi$ be a $C^2$-functional on $H_0^1(\Omega)$.  For any $u\in H_0^1(\Omega)$, we
	define the continuous bilinear map:
	\begin{equation*}
		D^2\phi(u)
		:= \phi''(u) - \frac{\phi'(u)[u]}{|u|^2} (\cdot,\cdot).
	\end{equation*}
\end{definition}

\begin{remark}
	If $u$ is a critical point of $\phi$ restricted to
	the sphere $S_\mu(\Omega)$, then $D^2 \phi(u)$, seen as a bilinear form on
	$T_u S_\mu(\Omega)$, is the second derivative of $\phi|_{S_\mu(\Omega)}$ at $u$.
\end{remark}

\begin{definition}\label{def: app morse}
	Let $\phi$ be a $C^2$-functional on $H_0^1(\Omega)$. For any $u\in S_\mu(\Omega)$ and
	$\theta > 0$, we define the \emph{approximate Morse index} by
	\begin{align*}
		\tilde m_\theta(u)
		= \sup \bigl\{\dim L: L \text{ is a subspace of }
		T_u S_\mu(\Omega) \text{ such that } D^2\phi(u)[\varphi, \varphi]<-\theta \|\varphi\|^2, \ \forall \varphi \in L \setminus \{0\}
		\bigr\}.
	\end{align*}
	If $u$ is a critical point for the constrained functional
	$\phi|_{S_\mu(\Omega)}$ and $\theta=0$, we say that this is the \emph{Morse
		index of $u$ as constrained critical point}.
\end{definition}

\begin{lemma} \label{lemmorse}
	Let $\{u_n\} \subset S_\mu(\Omega)$ be a Palais-Smale sequence of $E(\cdot,\Omega)$ constrained to $S_\mu(\Omega)$ such that $u_n \to u$ strongly in $H_0^1(\Omega)$. Assume that there exists $\zeta_n \to 0^+$ such that $\tilde m_{\zeta_n}(u_n) \le d$ for some positive integer $d$. Then $\tilde m_0(u) \leq d$. Moreover, for some $\la_u \in \R$, $u$ is a solution of equation \eqref{eqonbounddomain} with $\la = \la_u$ and $m(u) \leq d+1$.
\end{lemma}

\begin{proof}
	Let us first show that the Morse index of $u$, as a constrained critical point, satisfies $\tilde m_0(u) \leq d$.  Assume by contradiction that this is not the case. Then, in view of Definition \ref{def: app morse}, we may assume that there exists a $W_0 \subset T_uS_\mu(\Omega)$ with $\dim W_0 = d+1$ such that
	\begin{equation*}
		D^2 E(u,\Omega)[w,w]<0
		\quad \text{for all } w \in W_0 \setminus \{0\}.
	\end{equation*}
	Since $W_0$ is of finite dimension, its unit sphere is
	compact and there exists $\theta > 0$ such that
	\begin{equation*}
		D^2 E(u,\Omega)[w,w] < -\theta \|w\|^2
		\quad \text{for all } w \in W_0 \setminus\{0\}.
	\end{equation*}
	Now, observing that $E(\cdot,\Omega) \in C^{2,\alpha}$ where $\alpha \in (0,1]$ depending on $p$, it follows that there exists $\delta>0$ small enough such that, for any $v \in S_\mu(\Omega)$ satisfying $\|v-u\| \leq \delta$,
	\begin{equation*}
		D^2 E(v,\Omega)[w,w] < -\frac{\theta}{2} \|w\|^2
		\quad \text{for all } w \in W_0 \setminus \{0\}.
	\end{equation*}
	In particular, for $n$ large enough, $\|u_n - u\| \le
	\delta$ and $\zeta_n < \theta/2$ (as $\zeta_n \to 0^+$),
	so the previous inequality implies
	\begin{equation*}
		D^2 E(u_n,\Omega)[w,w] < -\frac{\theta}{2}\|w\|^2 < -\zeta_n\|w\|^2
		\quad \text{for all } w \in W_0 \setminus \{0\} .
	\end{equation*}
	Remembering that $\dim W_0>d$ and observing that the fact that $\tilde m_{\zeta_n}(u_n) \leq d$ directly implies that if there
	exists a subspace $W_n \subset T_{u_n}S_\mu(\Omega)$ such that
	\begin{equation*}
		D^2 E(u_n,\Omega)[w,w]  < - \zeta_n \|w\|^2,
		\quad \text{for all } w \in W_n \setminus \{0\},
	\end{equation*}
	then necessarily $\dim W_n \leq d$, we have reached a contradiction. Note that
	\begin{align*}
		\la_u = -\frac1\mu E'(u,\Omega)[u] = -\lim_{n\to\infty}\frac1\mu E'(u_n,\Omega)[u_n].
	\end{align*}
	Then, recalling that $S_\mu(\Omega)$ is of codimension $1$ in $H_0^1(\Omega)$ and observing that, for any $w \in H_0^1(\Omega)$,
	\begin{align*}
		D^2 E(u,\Omega)[w,w] & = E''(u,\Omega)[w,w] + \la_u\int_\Omega|w|^2dx \\
		& = \int_\Omega\left( |\nabla w|^2 + \left(\la_u-(p-1)|u|^{p-2}\right)|w|^2 \right)dx,
	\end{align*}
	we obtain that $m(u) \leq d+1$.
\end{proof}

For the completeness, at the end of this section, we recall the abstract results in \cite{JS} on the invariance of some subset with respect to a flow on a manifold.

\vskip0.1in

Let $(E,\langle \cdot,\cdot\rangle)$ and $(H,( \cdot,\cdot))$ be two \emph{infinite-dimensional} Hilbert spaces and assume that
\begin{align*}
	E \hookrightarrow H \hookrightarrow E',
\end{align*}
with continuous injections. For simplicity, we assume that the continuous injection $E \hookrightarrow H$ has norm at most $1$ and identify $E$ with its image in $H$. We also introduce
\begin{align*}
	\left\{
	\begin{aligned}
		\|u\|^2 & = \langle u,u \rangle, \\
		|u|^2 \ & = (u,u), \\
	\end{aligned}
	\right.
	\quad \quad u \in E,
\end{align*}
and, for $\mu \in (0,\infty)$, we define
\begin{align*}
	S_\mu := \bigl\{u \in E: |u|^2 = \mu \bigr\}.
\end{align*}
For our application, it is plain that $E = H_0^1(\Omega)$ and $H = L^2(\Omega)$.

\vskip0.1in

The following result, Proposition 4.5 in \cite{JS}, extends the classical Br\'ezis-Martin type results (see e.g. \cite[Theorem 1]{Bre}, \cite[Theorem 3]{Mar1}, and \cite[Theorem 4.1]{Dei}) to a manifold setting.

\begin{proposition}[Proposition 4.5 in \cite{JS}] \label{propbmmanifold}
	Let $u \in B_\mu \subset S_\mu$, $U$ be a neighborhood of $u$ in $S_\mu$ and $B_\mu \cap \overline{U}$ closed in $S_\mu$. If $V$ is a Lipschitz mapping on $\overline{U}$ with $V(w) \in T_wS_\mu$ for any $w \in \overline{U}$, where $T_wS_\mu$ is the tangent space of $S_\mu$ at $w$ defined as
	\begin{align*}
		T_wS_\mu := \bigl\{\phi \in E: (\phi,w) = 0\bigr\},
	\end{align*}
	and
	\begin{align} \label{eqd=0}
		\lim_{s \searrow 0}s^{-1}\text{dist}(w + sV(w),B_\mu) = 0, \quad \forall w \in \overline{U} \cap B_\mu.
	\end{align}
	Then  there exist $r = r(u) > 0$ and $\eta(t,u)$ satisfying
	\begin{align} \label{eqeta}
		\left\{
		\begin{aligned}
			& \frac{\partial}{\partial t}\eta(t,u) = V(\eta(t,u)), \quad \forall t \in [0,r), \\
			& \eta(0,u) = u, \quad \eta(t,u) \in B_\mu.
		\end{aligned}
		\right.
	\end{align}
\end{proposition}

In the next lemma, we give a condition to check that hypothesis \eqref{eqd=0} holds.

\begin{lemma} \label{lemconvex}
	Let $B_\mu = \tilde B \cap S_\mu$ where $\tilde B \subset E$ is close, convex in $E$, and satisfying
	\begin{align} \label{eqcone}
		kw \in \tilde B, \quad \forall k \in (0,1), w \in \tilde B.
	\end{align}
	For $u \in B_\mu$, if $V(u)$ has the form of $u - G(u)$ and $G(u) \in \tilde B$, then
	\begin{align} \label{eqd=0atapoint}
		\lim_{s \searrow 0}s^{-1}\text{dist}(u - sV(u),B_\mu) = 0.
	\end{align}
\end{lemma}

\begin{proof}
	The proof is modified from the one of \cite[Lemma 4.7]{JS}.
	Let $u_s = u - sV(u)$ where $u \in B_\mu \subset S_\mu$. Since $V(u) \in T_uS_u$,
	\begin{align*}
		|u_s|^2 & = (u-sV(u),u-sV(u)) \\
		& = (u,u) - 2s(u,V(u)) + s^2(V(u),V(u)) \\
		& = \mu + s^2|V(u)|^2.
	\end{align*}
	Note that $u_s = sG(u) + (1-s)u \in \tilde B$ by the convexity of $\tilde B$. Moreover, $\sqrt{\mu} < |u_s|$ for $s > 0$. Then, using \eqref{eqcone} we know $\sqrt \mu u_s/|u_s| \in \tilde B \cap S_\mu = B_\mu$. Hence,
	\begin{align*}
		\lim_{s \searrow 0}s^{-1}\text{dist}(u - sV(u),B_\mu) & \leq \lim_{s \searrow 0}s^{-1}\bigl\|u_s -  \frac{\sqrt\mu u_s}{|u_s|}\bigr\| \\
		& = \lim_{s \searrow 0}\frac {s|V(u)|^2}{2\mu}\|u_s\| = 0,
	\end{align*}
	which completes the proof.
\end{proof}

\section{Lagrange multipliers and energies of solutions with bounded Morse index} \label{secLagrange multiplier}

Throughout this section we will deal with a sequence $\{u_n,\la_n,\tau_n\} \subset H_0^1(\Omega) \times \R \times [1/2,1]$ satisfying
\begin{align} \label{equn}
	\begin{cases}
		-\Delta u_n + \la_n u_n = \tau_n|u_n|^{p-2}u_n & \quad \text{in } \Omega, \\
		u_n = 0 & \quad \text{on } \partial \Omega.
	\end{cases}	
\end{align}
We always assume that there exists a positive integer $\bar k$ such that $m(u_n) \le \bar k$ for all $n$ and that $\tau_n \to \bar \tau$ for $\bar\tau \in [1/2,1]$ as $n \to \infty$ in this section. The core study is the relationship between the parameter $\la_n$ and the energy $E_{\tau_n}(u_n,\Omega)$. Roughly speaking, when assuming that the Morse index of $u_n$ is uniformly bounded, we will prove that the boundedness of $\{\la_n\}$ is equivalent to the boundedness of $\{E_{\tau_n}(u_n,\Omega)\}$ if $p \neq p_c$. Main tools in the proof are some results obtained by the Morse index information and by blow-up analyses.

\vskip0.1in
Observe that $\{\la_n\}$ is bounded from below since the Morse index of $u_n$ is uniformly bounded, that is 

\begin{lemma} \label{lemlanbelow}
	It holds that
	\begin{align*}
		\liminf_{n \to \infty}\la_n > -\infty.
	\end{align*}
\end{lemma}

\begin{proof}
	Arguing by contradiction, up to a subsequence we assume that $\la_n \to -\infty$. Let $W \subset H_0^1(\Omega)$ satisfy $\dim W = \bar k + 1$. For any $\varphi \in W \backslash \{0\}$, we have
	\begin{align*}
		Q(\varphi; u_n, \Omega) & = \int_{\Omega} \bigl(|\nabla \varphi|^2
		+ (\la_n- (p-1)\tau|u_n|^{p-2}) \varphi^2\bigr) dx \\
		& \le \int_{\Omega}|\nabla \varphi|^2dx + \la_n\int_{\Omega}| \varphi|^2dx < 0
	\end{align*}
    for sufficiently large $n$ independent of the choice of $\varphi$. This shows that $m(u_n) \geq \bar k +1$ if $n$ is large enough, contradicting that $m(u_n) \le \bar k$ for all $n$.
\end{proof}

In the following lemma we will show that the boundedness of $\{\la_n\}$ is equivalent to the boundedness of $\{u_n\}$ in $L^\infty$.

\begin{lemma} \label{lemlanbounded}
	We have that $\{\la_n\}$ is bounded if and only if $\{u_n\}$ is a $L^\infty$-bounded sequence.
\end{lemma}

\begin{proof}
	\emph{Sufficiency:}
	It is plain that $u_n \in H^2(\Omega) \cap H_0^1(\Omega) \cap C(\overline{\Omega})$ by regularity. Without loss of generality, we take $x_n \in \Omega$ such that $u_n(x_n) = \|u_n\|_{L^\infty} > 0$ and so $u_n(x_n) = \max_{x \in \Omega}u_n(x)$. Then, by \eqref{equn} we get
	\begin{align*}
		\la_n u_n(x_n) \le \tau_n |u_n(x_n)|^{p-2}u_n(x_n),
	\end{align*}
    yielding that
    \begin{align*}
    	\limsup_{n \to \infty}\la_n \le \limsup_{n \to \infty}\tau_n\|u_n\|_{L^\infty}^{p-2} = \bar\tau \limsup_{n \to \infty}\|u_n\|_{L^\infty}^{p-2}.
    \end{align*}
    In view of Lemma \ref{lemlanbelow}, we get that $\{\la_n\}$ is bounded if $\{u_n\}$ is a $L^\infty$-bounded sequence.

    \vskip0.1in
    \emph{Necessity:} Assuming that $\{\la_n\}$ is bounded and $\{u_n\}$ is unbounded in $L^\infty$, we aim to find a contradiction by a classical blow-up argument as \cite{BL2}. In fact, up to a subsequence we assume that $\lim_{n \to \infty}\frac{\la_n}{\|u_n\|_{L^\infty}^{p-2}} = 0$. We then consider
    \begin{align*}
    	\tilde u_n(x) = \|u_n\|_{L^\infty}^{-1}u_n(x_n + \|u_n\|_{L^\infty}^{-(p-2)/2}x), \quad x \in \|u_n\|_{L^\infty}^{(p-2)/2}(\overline{\Omega} - x_n),
    \end{align*}
    which satisfies in $\Omega_n = \|u_n\|_{L^\infty}^{(p-2)/2}(\Omega - x_n)$ the following equation
    \begin{align*}
    	- \Delta \tilde u_n + \frac{\la_n}{\|u_n\|_{L^\infty}^{p-2}}\tilde u_n = |\tilde u_n|^{p-2}\tilde u_n, \quad \tilde u_n = 0 \quad \text{on } \partial \Omega_n.
    \end{align*}
    Then, following the arguments in proving \cite[Theorem 2]{BL2} step by step, we can find a contradiction with the Liouville Theorem, \cite[Theorem 3]{BL2}, which was stated for $N \ge 2$ in \cite{BL2} and also holds true when $N =1$. The proof is complete.
\end{proof}

\begin{remark} \label{rmkpositive}
	If $u_n$ is positive for all $n$, as mentioned in the introduction, it is necessary that $\la_n > -\la_1(B_1)$, see e.g. \cite[Theorem 1.19]{Willem}. Moreover, we remark that the result in Lemma \ref{lemlanbounded} still holds true if we remove the assumption that the Morse index of $u_n$ is uniformly bounded, which can be proved by a blow-up argument as \cite{GS}.
\end{remark}

Next, as $\la_n \to +\infty$ we recall some known results in \cite{PV}. The following result is obtained by combining Proposition 2.10 and Proposition 2.11 in \cite{PV}.

\begin{proposition} \label{proplan}
	Let $\la_n \to +\infty$. Then, up to a subsequence, there exist $x_n^1, \cdots, x_n^k \in \Omega$ with $k \le \bar k$, such that:
	\begin{itemize}
		\item[(1)] As $n \to \infty$,
		\begin{align}
			\sqrt{\la_n}\text{dist}(x_n^i,\partial \Omega) \to +\infty \quad \text{and} \quad \sqrt{\la_n}|x_n^i - x_n^j| \to +\infty, \quad i,j = 1,\cdots,k, \ i \neq j,
		\end{align}
	    and
	    \begin{align}
	    	|u_n(x_n^i)| = \max_{\Omega \cap B_{R_n\la_n^{-1/2}}(x_n^i)}|u_n|, \quad i = 1,\cdots,k,
	    \end{align}
        for some $R_n \to +\infty$. Additionally,
        \begin{align}
        	\lim_{R \to \infty} \limsup_{n \to \infty}\biggl( \la_n^{-\frac{1}{p-2}} \max_{d_{n,k}(x) \ge R \la_n^{-1/2}}|u_n(x)| \biggr) = 0,
        \end{align}
        where
        \begin{align*}
        	d_{n,k}(x) := \min\bigl\{ |x - x_n^i|: i = 1,\cdots,k\bigr\}.
        \end{align*}
        \item[(2)] For some positive constants $C, \ga$,
        \begin{align}
        	|u_n(x)| \le C \la_n^{\frac{1}{p-2}} \sum_{i =1}^k e^{-\ga \sqrt{\la_n}|x - x_n^i|}, \quad \forall x \in \Omega, \ \forall n.       	
        \end{align}
	\end{itemize}
\end{proposition}

Along the line of the proof of \cite[Proposition 2.14]{PV}, using Proposition \ref{proplan} we conclude that

\begin{proposition} \label{propasm}
	As $n \to \infty$ we have
	\begin{align}
		& \tau_n^{\frac{2}{p-2}}\la_n^{\frac{N}{2} - \frac{2}{p-2}} \int_\Omega |u_n|^2dx \to \sum_{i = 1}^{k}\int_{\R^N} |V_i|^2dx, \label{eqntoinftyun2} \\
		& \tau_n^{\frac{p}{p-2}}\la_n^{\frac{N}{2} - \frac{p}{p-2}} \int_\Omega |u_n|^pdx \to \sum_{i = 1}^{k}\int_{\R^N} |V_i|^pdx, \\
		& \tau_n^{\frac{2}{p-2}}\la_n^{\frac{N}{2} - \frac{p}{p-2}}\int_\Omega |\nabla u_n|^2dx \to \sum_{i = 1}^{k}\int_{\R^N} |\nabla V_i|^2dx,
	\end{align}
    where $V_i$ is a bounded solution of
    \begin{align} \label{equationV}
    	-\Delta V + V = |V|^{p-2}V \quad \text{in } \R^N,
    \end{align}
    with $m(V) \leq \bar k$.
\end{proposition}

\begin{remark}
	As proved in \cite[Lemma 2.13]{PV}, there exists a constant C, only depending on the full sequence $\{u_n\}$ and not on $V_i$, such that
	\begin{align*}
		\|V_i\|_{H^1(\R^N)}^2 = \|V_i\|_{L^p(\R^N)}^p \le C,
	\end{align*}
    where $V_i$ is the bounded solution of \eqref{equationV} with $m(V) \leq \bar k$, given in Proposition \ref{propasm}.
\end{remark}

Now we can state our main result in this section.

\begin{proposition} \label{proprelationship}
	Let $2 < p < 2^*$. The following hold:
	\begin{itemize}
		\item[(1)] If $p \neq p_c$, we have that $\{\la_n\}$ is bounded if and only if $\{E_{\tau_n}(u_n,\Omega)\}$ is bounded. Moreover,
		\begin{itemize}
			\item when $2 < p < p_c$, up to a subsequence, $\la_n \to +\infty$ if $\{\la_n\}$ is not bounded, and we have that $E_{\tau_n}(u_n,\Omega) \to -\infty$;
			\item when $p_c < p < 2^*$, up to a subsequence, $\la_n \to +\infty$ if $\{\la_n\}$ is not bounded, and we have that $E_{\tau_n}(u_n,\Omega) \to +\infty$.
		\end{itemize}
		\item[(2)] If $p = p_c$, the fact that $\{\la_n\}$ is bounded indicates that $\{E_{\tau_n}(u_n,\Omega)\}$ is bounded.
	\end{itemize}
	
\end{proposition}

\begin{proof}
	\emph{Necessity:} Let $2 < p < 2^*$. Assuming that $\{\la_n\}$ is bounded and $\{E_{\tau_n}(u_n,\Omega)\}$ is not bounded, we aim to find a contradiction. Up to a subsequence, we can assume that $\int_\Omega |\nabla u_n|^2dx \to +\infty$ as $n \to \infty$. Since
	\begin{align*}
		\int_\Omega |\nabla u_n|^2dx + \la_n \int_\Omega |u_n|^2dx = \tau_n \int_\Omega |u_n|^pdx,
	\end{align*}
	one of the following holds
	\begin{itemize}
		\item[(a)] $\la_n \int_\Omega |u_n|^2dx \to -\infty$,
		\item[(b)] $\tau_n \int_\Omega |u_n|^pdx \to +\infty$.
	\end{itemize}
	Observe that both $\{\la_n\}$ and $\{\tau_n\}$ are bounded. If the case (a) occurs, then $\int_\Omega |u_n|^2dx \to +\infty$ and so $\{u_n\}$ is unbounded in $L^\infty$, contradicting Lemma \ref{lemlanbounded}. If (b) holds, we can find a similar contradiction.
	
	\vskip0.1in
	\emph{Sufficiency:} Let $p \neq p_c$ and $\{E_{\tau_n}(u_n,\Omega)\}$ be bounded. We claim that there exists $\bar \la < +\infty$ such that $\la_n \le \bar \la$ for all $n$. By negation, we assume that $\la_n \to +\infty$ up to a subsequence. By Proposition \ref{propasm} we have
	\begin{align} \label{eqlimitenergy}
		\tau_n^{\frac{2}{p-2}}\la_n^{\frac{N}{2} - \frac{p}{p-2}}E_{\tau_n}(u_n,\Omega) \to \frac12\sum_{i = 1}^{k}\int_{\R^N} |\nabla V_i|^2dx - \frac1p \sum_{i = 1}^{k}\int_{\R^N} |V_i|^pdx.
	\end{align}
    On the one hand, since $\{\tau_n\} \subset [1/2,1]$ and $N/2 - p/(p-2) < 0$, we have
    \begin{align*}
    	\tau_n^{\frac{2}{p-2}}\la_n^{\frac{N}{2} - \frac{p}{p-2}}E_{\tau_n}(u_n,\Omega) \to 0.
    \end{align*}
    On the other hand, $V_i \in H^1(\R^N)$ satisfies the following Pohozaev identity
    \begin{align*}
    	\frac{N-2}{2}\int_{\R^N}|\nabla V_i|^2dx + \frac{N}{2}\int_{\R^N}| V_i|^2dx = \frac{N}{p}\int_{\R^N}| V_i|^pdx, \quad i =1,\cdots,k,
    \end{align*}
    see e.g. \cite[Theorem B.3]{Willem}, which implies that
    \begin{align*}
    	\int_{\R^N}|\nabla V_i|^2dx = \left( \frac12 -\frac1p\right) N\int_{\R^N}| V_i|^pdx, \quad i =1,\cdots,k,
    \end{align*}
    and so
    \begin{align} \label{eqlpV}
    	\frac12\sum_{i = 1}^{k}\int_{\R^N} |\nabla V_i|^2dx - \frac1p \sum_{i = 1}^{k}\int_{\R^N} |V_i|^pdx = \left( \frac{p-2}{4p}N - \frac1p\right)\sum_{i = 1}^{k}\int_{\R^N} |V_i|^pdx \neq 0,
    \end{align}
    since $p \neq p_c$. Thus we find a contradiction and prove the claim. In view of Lemma \ref{lemlanbelow} we conclude that $\{\la_n\}$ is bounded.
	
	Moreover, if $\{\la_n\}$ is not bounded, by Lemma \ref{lemlanbelow}, up to a subsequence we have $\la_n \to +\infty$. Then using \eqref{eqlimitenergy} and \eqref{eqlpV} we can complete the proof.
\end{proof}

\section{Exact number of positive solutions} \label{secpositionsolu}

This section is devoted to prove Theorem \ref{thmunique}. In the following we will give two useful lemmas and use the reductio ad absurdum to complete the proof.

\begin{lemma} \label{lemMmu}
	Let $p_c < p < 2^*$ and $u_{\tau,\la}$ be the unique solution of \eqref{eqwithfixlaB1}, given by Proposition \ref{propuniquefixlaB1}; let $M_\tau(\la) = \int_{B_1} |u_{\tau,\la}|^2dx$. Then, there exists $\La_1^*, \La_2^* \in (-\la_1(B_1),\infty)$ independent of $\tau \in [1/2,1]$, such that $M_\tau(\la)$ is strictly increasing in $\la \in (-\la_1(\Omega), \La_1^*)$ and is strictly decreasing in $\la > \La_2^*$.
\end{lemma}

\begin{proof}
	Note that the map $\la \in (-\la_1(B_1),\infty) \mapsto u_{\tau,\la} \in H_0^1$ is of class $C^1$ (see \cite[Corollary 2.4]{Song} for the proof). Thus $M_\tau'(\la)$ exists in $\la \in (-\la_1(B_1),\infty)$, and $M_\tau'(\la) = 2\int_{B_1}u_{\tau,\la}v_{\tau,\la}dx$, where $v_{\tau,\la} = \partial_\la u_{\tau,\la}$.
	
	Firstly, we prove that there exists $\La_2^* \in (-\la_1(B_1),\infty)$ independent of $\tau \in [1/2,1]$ such that $M_\tau'(\la) < 0$ in $\la > \La_2^*$. We see that $v_{\tau,\la} \in H_0^1(B_1)$ is also radial and satisfies the following equation
	\begin{align*}
		\begin{cases}
			-\Delta v + \la v - \tau(p-1)|u_{\tau,\la}|^{p-2}v = -u_{\tau,\la}, & \quad \text{in } B_1, \\
			v(x) = 0, & \quad \text{on } \partial B_1.
		\end{cases}
	\end{align*}
    For $\la > 0$ and $\tau \in [1/2,1]$ we rescale the solutions as follows:
    \begin{align*}
    	\tilde{u}_{\tau,\la}(x) = \la^{-\frac1{p-2}}\tau^{\frac{1}{p-2}}u_{\tau,\la}(\la^{-\frac12}x), \quad \tilde{v}_{\tau,\la}(x) = \la^{-\frac{3-p}{p-2}}\tau^{\frac{1}{p-2}}v_{\tau,\la}(\la^{-\frac12}x).
    \end{align*}
    Observe that $\tilde{u}_{\tau,\la}(x) \in H_0^1(B_{\sqrt{\la}})$ satisfies
    \begin{align*}
    	-\Delta \tilde{u}_{\tau,\la} + \tilde{u}_{\tau,\la} = |\tilde{u}_{\tau,\la}|^{p-2}\tilde{u}_{\tau,\la} \quad \text{in } B_{\sqrt{\la}},
    \end{align*}
    and that $\tilde{v}_{\tau,\la}(x) \in H_0^1(B_{\sqrt{\la}})$ satisfies
    \begin{align*}
    	-\Delta \tilde{v}_{\tau,\la} + \tilde{v}_{\tau,\la} - (p-1)|\tilde{u}_{\tau,\la}|^{p-2}\tilde{v}_{\tau,\la} = -\tilde{u}_{\tau,\la} \quad \text{in } B_{\sqrt{\la}}.
    \end{align*}
	For any $\{\la_n\}$ and $\{\tau_n\} \subset [1/2,1]$ satisfying $\la_n \to +\infty$, after a small modification of the proof for \cite[Proposition 11]{FSK}, we get that as $n \to \infty$,
	\begin{align*}
		\tilde{u}_{\tau_n,\la_n} \to Q_1  \quad \text{and} \quad \tilde{v}_{\tau_n,\la_n} \to \frac{1}{p-2}Q_1 + \frac12 x \cdot \nabla Q_1, \quad \text{strongly in } H^1(\R^N),
	\end{align*}
    where $Q_1 \in H^1(\R^N)$ is the unique, positive, radial solution of \eqref{eqonrn} with $\la =1$. Moreover, since $p > p_c$, as $n \to \infty$,
    \begin{align*}
    	M_{\tau_n}(\la_n) = \int_{B_1}|u_{\tau_n,\la_n}|^2dx = \la_n^{\frac{2}{p-2}-\frac N2}\tau_n^{-\frac{2}{p-2}}\int_{B_{\sqrt{\la}}}|\tilde{u}_{\tau_n,\la_n}|^2dx \to 0.
    \end{align*}
    Now arguing by contradiction, we suppose that for any $\La > 0$,  there exist $\la > \La$ and $\tau = \tau(\la)$ such that $M_\tau'(\la) = 0$. Then we find $\{\la_n\}$ and $\{\tau_n\} \subset [1/2,1]$ such that $\la_n \to +\infty$ and $M_{\tau_n}'(\la_n) = 0$, and get the following self-contradictory inequality
    \begin{align*}
    	0 = \frac12\la_n^{\frac N2 - \frac{4-p}{p-2}}\tau_n^{\frac{2}{p-2}}M_{\tau_n}'(\la_n) & = \la_n^{\frac N2 - \frac{4-p}{p-2}}\tau_n^{\frac{2}{p-2}}\int_{B_1} u_{\tau_n,\la_n} v_{\tau_n,\la_n} dx \\
    	& = \int_{B_{\sqrt{\la}}} \tilde{u}_{\tau_n,\la_n} \tilde{v}_{\tau_n,\la_n} dx \\
    	& \to \int_{\R^N}Q_1\left( \frac{1}{p-2}Q_1 + \frac12 x \cdot \nabla Q_1\right) dx \\
    	& = \left( \frac{1}{p-2} - \frac N4\right) \int_{\R^N}|Q_1|^2dx \\
    	& < 0.
    \end{align*}
Next, we prove that there exists $\La_1^* \in (-\la_1(B_1),\infty)$ independent of $\tau \in [1/2,1]$ such that $M_\tau'(\la) > 0$ in $\la \in (-\la_1(B_1),\La_1^*)$. For any $\{\la_n\}$ and $\{\tau_n\} \subset [1/2,1]$ satisfying $\la_n \to +\infty$, after a small modification of the proof for \cite[Proposition 12]{FSK}, we get that as $n \to \infty$,
    \begin{align*}
    	\frac{u_{\tau_n,\la_n}}{\|u_{\tau_n,\la_n}\|_{L^2}} \to \phi_1, \quad \text{and} \quad \frac{v_{\tau_n,\la_n}}{\|v_{\tau_n,\la_n}\|_{L^2}} \to \phi_1 \quad \text{strongly in } H_0^1(B_1).
    \end{align*}
    We argue by contradiction and suppose that for any $\La > -\la_1(B_1)$,  there exist $\la \in (-\la_1(B_1), \La)$ and $\tau = \tau(\la)$ such that $M_\tau'(\la) \le 0$. Then we find $\{\la_n\}$ and $\{\tau_n\} \subset [1/2,1]$ such that $\la_n \to +\infty$ and $M_{\tau_n}'(\la_n) \le 0$, and get the following self-contradictory inequality
    \begin{align*}
    0 \ge \frac{1}{2\|u_{\tau_n,\la_n}\|_{L^2}\|v_{\tau_n,\la_n}\|_{L^2}}M_{\tau_n}'(\la_n) = \int_{B_1}\frac{u_{\tau_n,\la_n}}{\|u_{\tau_n,\la_n}\|_{L^2}}\frac{v_{\tau_n,\la_n}}{\|v_{\tau_n,\la_n}\|_{L^2}} dx \to \int_{B_1}|\phi_1|^2dx = 1.
    \end{align*}
    The proof is complete.
\end{proof}

\begin{lemma} \label{lemboundedla}
	Let $u_n \in H_0^1(B_1)$ solve
	\begin{align} \label{eq2.1}
		-\Delta u_n + \la_n u_n = \tau_n|u_n|^{p-2}u_n \quad \text{in } B_1, \quad u_n > 0 \quad \text{in } B_1, \quad \int_{B_1}|u_n|^2dx = \mu_n,
	\end{align}
    with $\{\tau_n\} \subset [1/2,1]$ and $\mu_n \to 0$ as $n \to \infty$. Up to a subsequence, we have
    \begin{align*}
    	\text{either} \quad \la_n \to +\infty \quad \text{or} \quad \la_n \to - \la_1(B_1).
    \end{align*}
\end{lemma}

\begin{proof}
	As mentioned in the introduction, it is necessary that $\la_n > -\la_1(B_1)$, see e.g. \cite[Theorem 1.19]{Willem}. Passing to a subsequence, we suppose that $\{\la_n\}$ is bounded in $\R$, $\la_n \to \bar \la \geq -\la_1(B_1)$, $\tau_n \to \bar \tau \in [1/2,1]$, and we aim to prove $\bar \la = -\la_1(B_1)$, which will complete our proof.
	
	Since $\{\la_n\}$ is bounded, Remark \ref{rmkpositive} shows that $\{u_n\}$ is bounded in $L^\infty$. Thus $\{u_n\}$ is a $L^p$-bounded sequence. By \eqref{eq2.1}, it follows that
	\begin{align*}
		\int_{B_1} |\nabla u_n|^2dx + \la_n \mu_n = \tau_n\int_{B_1}|u_n|^pdx,
	\end{align*}
    yielding the boundedness of $\{u_n\}$ in $H_0^1(B_1)$. Moreover, by \eqref{eqgndomain}, the boundedness of $\{u_n\}$ in $H_0^1(B_1)$ and the fact that $\mu_n \to 0$, we conclude that $u_n \to 0$ strongly in $L^p$. Then we get
    \begin{align*}
    	\int_{B_1} |\nabla u_n|^2dx = \tau_n\int_{B_1}|u_n|^pdx - \la_n \mu_n \to 0.
    \end{align*}
    This means that $-\bar \la$ is a bifurcation point, and by classical bifurcation theory, $-\bar \la$ is an eigenvalue of $-\Delta$ on $B_1$ with Dirichlet condition. Recalling that $\bar \la \geq -\la_1(B_1)$, we conclude that $\bar \la = -\la_1(B_1)$ immediately. Then the proof is completed.
\end{proof}

\begin{proof}[Proof of Theorem \ref{thmunique}]
	We argue by contradiction and suppose that, for a sequence $\{\mu_n\}$ with $\mu_n \to 0^+$, there exist three different sequences $\{u_{n,1}\}, \{u_{n,2}\}, \{u_{n,3}\} \subset H_0^1(B_1)$ such that $u_{n,k} \neq u_{n,l}$ for $k \neq l$, and $u_{n,i}$ solves \eqref{eqequationtau} with $\tau = \tau_n$ and $\la = \la_{n,i}$, $i = 1,2,3$. By Lemma \ref{lemboundedla}, either $\la_{n,i} \to -\la_1(B_1)$ or $\la_{n,i} \to \infty$ for all $i \in \{1,2,3\}$. Thus, at least two parameters tends to $-\la_1(B_1)$ or to the infinity. If the former holds, without loss of generality, we assume that $\la_{n,1}, \la_{n,2} \to -\la_1(B_1)$. Observe that $M_{\tau_n}(\la_{n,1}) = M_{\tau_n}(\la_{n,2})$ where $M_{\tau}(\la)$ is given in Lemma \ref{lemMmu}. By Lemma \ref{lemMmu} it follows that $\la_{n,1} = \la_{n,2}$ for sufficiently large $n$. However, we get a contradiction since Proposition \ref{propuniquefixlaB1} yields that $u_{n,1} = u_{n,2}$ for $n$ large enough. We obtain a similar contradiction if the latter holds, and complete the proof.
\end{proof}

\section{Descending flow techniques and flow invariance} \label{secinv}

In this section we will construct a descending flow with a pseudogradient operator. Introducing the following bounded subsets of the mass constraint for $\rho > 0$:
\begin{align*}
	\B_\rho^\mu = \bigl\{u \in S_\mu(\Omega): \int_\Omega|\nabla u|^2dx < \rho\bigr\}, \quad \U_\rho^\mu = \bigl\{u \in S_\mu(\Omega): \int_\Omega|\nabla u|^2dx = \rho\bigr\},
\end{align*}
we give a result preparing the introduction a pseudogradient operator.

\begin{proposition} \label{propuniquesolutionw}
	Let $\tau \in [1/2,1]$ and $2 < p < 2^*$. Given $u \in S_\mu(\Omega)$, for each $\bar\la > -\la_1(\Omega)$ the linear equation with a constraint
	\begin{align} \label{eqequofw1}
		\left\{
		\begin{aligned}
			& - \Delta w + \bar \la w = \tau|u|^{p-2}u + \omega u, \quad w \in H_0^1(\Omega), \\
			& \int_\Omega wu dx = \mu,
		\end{aligned}	
		\right.
	\end{align}
    has a unique solution where $\omega \in \R$ is an unknown Lagrange multiplier. Moreover, if $\mu, \rho > 0$ and $\bar \la > -\la_1(\Omega)$ satisfy
    \begin{align} \label{eqcononbarla}
    	C_{p,N}^p\rho^{\frac{(p-1)\ga_p}{2}}\mu^{\frac{p-(p-1)\ga_p - 2}{2}} \leq \left( \frac{\la_1(\Omega) + \min\{0,\bar\la\}}{\sqrt{\la_1(\Omega)}}\right) ^{\ga_p}\left( \la_1(\Omega) + \bar\la\right) ^{1-\ga_p},
    \end{align}
    then $\omega \geq 0$ for all $u \in \B_\rho^\mu$.
\end{proposition}

\begin{proof}
	\emph{Existence:} Let $\bar\la > -\la_1(\Omega)$. We consider the following minimization problem
    \begin{align} \label{eqminimization problem1}
    	m := \big\{\inf \int_\Omega \left(|\nabla w|^2 + \bar \la w^2 - \tau|u|^{p-2}uw \right)dx: w \in H_0^1(\Omega), \int_\Omega wu dx = \mu \big\}.
    \end{align}
	By H\"older inequality and Sobolev inequality we get
	\begin{align*}
		\tau\int_\Omega|u|^{p-2}uw dx & \leq \left(\int_\Omega|u|^pdx\right)^{\frac{p-1}p}\left(\int_\Omega|w|^pdx\right)^{\frac1p} \leq \mathcal{S}_p\left(\int_\Omega|u|^pdx\right)^{\frac{p-1}p}\left(\int_\Omega |\nabla w|^2dx\right)^{\frac12},
	\end{align*}
	and thus
	\begin{align} \label{eqboungbelow}
		&\int_\Omega \left(|\nabla w|^2 + \bar \la w^2 - \tau|u|^{p-2}uw \right)dx \nonumber \\
& \geq \frac{\la_1(\Omega)+\min\{0,\bar \la\}}{\la_1(\Omega)}\int_\Omega |\nabla w|^2dx - \mathcal{S}_p\left(\int_\Omega|u|^pdx\right)^{\frac{p-1}p}\left(\int_\Omega |\nabla w|^2dx\right)^{\frac12},
	\end{align}
	which indicates that $m$, defined in \eqref{eqminimization problem1}, is bounded from below. Let $\{w_n\}$ be a minimizing sequence for $m$. By \eqref{eqboungbelow} we obtain that the sequence $\{w_n\}$ is bounded in $H_0^1(\Omega)$. Up to a subsequence, we assume that $w_n$ converges to $w$ weakly in $H_0^1(\Omega)$ and strongly in $L^2(\Omega)$ and $L^p(\Omega)$ for some $w \in H_0^1(\Omega)$. It is clear that $\int_\Omega wu dx = \mu$ and $\lim_{n \to \infty}\int_\Omega |u|^{p-2}uw_ndx \to \int_\Omega |u|^{p-2}uw dx$. Then we have
	\begin{align*}
		m \leq \int_\Omega \left(|\nabla w|^2 + \bar \la w^2 - \tau|u|^{p-2}uw \right)dx \leq \liminf_{n \to \infty}\int_\Omega \left(|\nabla w_n|^2 + \bar \la w_n^2 - \tau|u|^{p-2}uw_n \right)dx = m,
	\end{align*}
	which shows that $m$ is attained by $w$. By Lagrange multiplier principle, $w$ solves \eqref{eqequofw1} for some $\omega \in \R$.
	
	\vskip0.1in
	\emph{Uniqueness:} If $w_1, w_2$ are two solutions of \eqref{eqequofw1} with $\omega = \omega_1,\omega_2$ respectively, by direct computation we get
    \begin{align}
    	\int_\Omega\left( |\nabla (w_2-w_1)|^2 + \bar \la (w_2-w_1)^2\right) dx = \int_\Omega \omega_2 u(w_2-w_1)dx - \int_\Omega \omega_1 u(w_2-w_1)dx = 0,
    \end{align}
    implying $w_2 \equiv w_1$, and so $\omega_2 = \omega_1$. Thus \eqref{eqequofw1} has a unique solution.

    \vskip0.1in
	\emph{Sign of $\omega$:} Given $u \in \B_\rho^\mu$, let $\chi = (-\Delta + \bar \la)^{-1}u \neq 0$. It holds that
	\begin{align*}
	   (\la_1(\Omega) + \bar\la)\|\chi\|_{L^2}^2 \le \sqrt{\la_1(\Omega)}\|\chi\|_{L^2(\Omega)}\|\nabla \chi\|_{L^2} + \bar\la \|\chi\|_{L^2}^2 \le \int_\Omega \left( |\nabla \chi|^2 + \bar \la \chi^2\right) dx = \int_\Omega u\chi dx \le \|u\|_{L^2}\|\chi\|_{L^2},
	\end{align*}
    by which it follows that
    \begin{align*}
    	\left(\la_1(\Omega) + \bar\la\right) \|\chi\|_{L^2} \le \|u\|_{L^2} \quad \text{and} \quad \frac{\la_1(\Omega)+\min\{0,\bar \la\}}{\sqrt{\la_1(\Omega)}} \|\nabla \chi\|_{L^2} \leq \|u\|_{L^2}.
    \end{align*}
    Then, by H\"older inequality, Gagliardo-Nirenberg inequality \eqref{eqgndomain}, and \eqref{eqcononbarla},
    \begin{align*}
    	\tau\int_\Omega |u|^{p-2}u\chi dx & \leq \|u\|_{L^p}^{p-1}\|\chi\|_{L^p} \\
    	& \le C_{p,N}^p \|\nabla u\|_{L^2}^{(p-1)\ga_p}\|u\|_{L^2}^{(p-1)(1-\ga_p)}\|\nabla \chi\|_{L^2}^{\ga_p}\|\chi\|_{L^2}^{1-\ga_p} \\
    	& \le C_{p,N}^p \left( \frac{\sqrt{\la_1(\Omega)}}{\la_1(\Omega)+\min\{0,\bar \la\}} \right)^{\ga_p} \left( \frac{1}{\la_1(\Omega) + \bar\la} \right) ^{1-\ga_p}\|\nabla u\|_{L^2}^{(p-1)\ga_p}\|u\|_{L^2}^{p-(p-1)\ga_p} \\
    	& \le \mu.
    \end{align*}
    Multiplying \eqref{eqequofw1} by $\chi$ and integrating by parts yields that
    \begin{align*}
    	\mu = \int_\Omega uw dx = \tau\int_\Omega |u|^{p-2}u\chi dx + \omega\int_\Omega u\chi dx \le \mu + \omega\int_\Omega u\chi dx.
    \end{align*}
    Thus we conclude that $\omega \geq 0$ and complete the proof.
\end{proof}

For $\mu,\rho > 0$, we fix $\bar\la = \bar\la(\rho,\mu) \geq 0$ satisfying \eqref{eqcononbarla}. We can now define the operator
\begin{align} \label{def:G}
	G_\tau : S_\mu(\Omega) \to H_0^1(\Omega), \quad u \mapsto G_\tau(u) = w,
\end{align}
that is, for each $u$, $G_\tau(u)$ is the unique solution of the equation \eqref{eqequofw1}.

\vskip0.1in

Next we state and prove some properties of the operator $G_\tau$.

\begin{lemma} \label{lemproG1}
	Given $u \in S_\mu(\Omega)$, we have $u - G_\tau(u) \in T_uS_\mu(\Omega)$.
\end{lemma}

\begin{proof}
	The result is clear by  noticing
	\begin{align*}
		\int_\Omega (u - G_\tau(u))u dx = \int_\Omega u^2 dx - \int_\Omega wu dx = \mu - \mu = 0,
	\end{align*}
    where $w$ is the unique solution of the  equation \eqref{eqequofw1} given in Proposition \ref{propuniquesolutionw}.
\end{proof}

\begin{lemma} \label{lemproG2}
	The operator $G_\tau \in C^1(S_\mu(\Omega),H_0^1(\Omega))$.
\end{lemma}

\begin{proof}
	It suffices to apply the Implicit Theorem to the $C^1$ map
	\begin{align*}
		& \Psi : S_\mu(\Omega) \times H_0^1(\Omega) \times \R \to H_0^1(\Omega) \times \R, \quad \text{where} \\
		& \Psi(u,w,\omega) = \left( w - (-\Delta + \bar\la)^{-1}(\tau|u|^{p-2}u + \omega u), \int_\Omega wu dx - \mu \right).
	\end{align*}
    Note that \eqref{eqequofw1} holds if and only if $\Psi(u,w,\omega) = 0$. By computing the derivative of $\Psi$ with respect to $w,\omega$ at the point $(u,w,\omega)$ in the direction $(\bar w, \bar \omega)$, we obtain a map $\Phi : H_0^1(\Omega) \times \R \to H_0^1(\Omega) \times \R$ given by
    \begin{align*}
    	\Phi(\bar w, \bar \omega) & := D_{w,\omega}\Psi(u,w,\omega)(\bar w, \bar \omega) \\
    	& = \left(\bar w - \bar\omega (-\Delta + \bar\la)^{-1}u, \int_\Omega \bar wu dx \right)
    \end{align*}
    If $\Phi(\bar w, \bar \omega)= (0,0)$, then we multiply the equation
    \begin{align*}
    	-\Delta \bar w + \bar\la \bar w = \bar \omega u
    \end{align*}
    by $\bar w$ and obtain
    \begin{align*}
    	\int_\Omega \left( |\nabla \bar w|^2 + \bar\la \bar w ^2\right) dx = \bar \omega\int_\Omega \bar wu dx = 0.
    \end{align*}
    By $\bar \la > -\la_1(\Omega)$ we have $\bar w \equiv 0$ and so $\bar \omega = 0$. Hence, $\Phi$ is injective.

    On the other hand, for any $(g,\nu) \in H_0^1(\Omega) \times \R$, let $$\kappa = \frac{1}{\int_\Omega(-\Delta + \bar \la)^{-1}u u dx}\left( \nu -\int_\Omega gu dx\right) .$$
    Computing directly we obtain
    \begin{align*}
    	\Phi(g+\kappa (-\Delta + \bar \la)^{-1}u,\kappa) = (g,\nu).
    \end{align*}
    Hence $\Phi$ is surjective, that is, $\Phi$ is a bijective map. This completes the proof.
\end{proof}

\begin{remark}
	When $G_\tau$ is viewed as a map from $[1/2,1] \times S_\mu(\Omega)$ to $H_0^1(\Omega)$, defined by $(\tau,u) \mapsto G_\tau(u)$, the proof of Lemma \ref{lemproG2} shows that it is also of class $C^1$.
\end{remark}

\begin{lemma} \label{lemproG3}
	Let $\{u_n\} \subset S_\mu(\Omega)$ be a sequence such that $u_n \rightharpoonup u$ weakly in $H_0^1(\Omega)$ and let $w_n = G_{\tau_n}(u_n)$ for some $\{\tau_n\} \subset [1/2,1]$ satisfying $\tau_n \to \bar\tau \in [1/2,1]$. Then $w_n \to w$ strongly in $H_0^1(\Omega)$ for some $w \in H_0^1(\Omega)$.
\end{lemma}

\begin{proof}
	By \eqref{eqboungbelow} it is readily seen that the sequence $\{w_n\}$ is bounded in $H_0^1(\Omega)$. Let $\omega_n$ be the Lagrange multiplier corresponding to $w_n$ given in Proposition \ref{propuniquesolutionw}. It is clear that $\omega_n$ is also bounded. Up to a subsequence, we assume that $w_n$ converges to some $w \in H_0^1(\Omega)$ weakly in $H_0^1(\Omega)$ and strongly in $L^2(\Omega)$ and $L^p(\Omega)$. Multiplying \eqref{eqequofw1} by $w_n - w$ we see that
	\begin{align*}
		\int_\Omega \left( \nabla w_n \nabla (w_n-w) + \bar\la w_n(w_n-w)\right) dx = \int_\Omega\left( \tau_n|u_n|^{p-2}u_n(w_n-w) + \omega_n u_n(w_n-w)\right) dx \to 0
	\end{align*}
    as $n \to \infty$ and therefore $w_n \to w$ strongly in $H_0^1(\Omega)$.
\end{proof}

\begin{lemma} \label{lempre1}
	Define
	\begin{align}
		M_{\pm}(\delta) := \sup\bigl\{\Big|\{x \in \Omega: \pm G_\tau(u)(x) < 0\}\Big| : u \in \overline{\pm\p_\delta} \cap \B_\rho^\mu \bigr\}.
	\end{align}
    Then $\lim_{\delta \to 0}M_{\pm}(\delta) = 0$ uniformly with respect to $\tau \in [1/2,1]$.
\end{lemma}

\begin{proof}
	We provide the proof of $\lim_{\delta \to 0}M_{+}(\delta) = 0$ uniformly with respect to $\tau \in [1/2,1]$ and similar arguments yield that $\lim_{\delta \to 0}M_{-}(\delta) = 0$ uniformly with respect to $\tau \in [1/2,1]$. Arguing by contradiction, we suppose that there exist $\delta_n \to 0$, $\{\tau_n\} \subset [1/2,1]$ and a sequence $\{u_n\} \subset \overline{\p_{\delta_n}} \cap \B_\rho^\mu$ such that
	\begin{align} \label{eqgun>0}
		\lim_{n \to \infty}|\{x \in \Omega: G_{\tau_n}(u_n)(x) < 0\}| > 0.
	\end{align}
	Up to a subsequence, we assume that $\tau_n \to \bar\tau \in [1/2,1]$ and that $u_n$ converges to a nonnegative function $u \in S_\mu(\Omega)$ weakly in $H_0^1(\Omega)$ and strongly in $L^2(\Omega)$ and $L^p(\Omega)$. By Lemma \ref{lemproG3} we have $w_n = G_{\tau_n}(u_n) \to w$ strongly in $H_0^1(\Omega)$ and it is readily seen that $w$ solves \eqref{eqequofw1} for $\tau = \bar \tau$ with Lagrange multiplier $\omega$. By Proposition \ref{propuniquesolutionw}, $\omega \ge 0$ and so
	\begin{align*}
		-\Delta w + \bar \la w = \bar\tau |u|^{p-2}u + \omega u \geq 0.
	\end{align*}
	Hence by the strong maximum principle we have $w > 0$ in $\Omega$, and therefore we conclude that $\lim_{n \to \infty}|\{x \in \Omega: w_n(x) < 0\}| = 0$, contradicting \eqref{eqgun>0}. The proof is complete.
\end{proof}

The following lemma not only plays an important role in proving Proposition \ref{propdescendingflow} in this section later, but also is a key step to complete the proof of Theorem \ref{thmlarge} in Section \ref{secappropro}.

\begin{lemma} \label{lemproG4}
    There exists $\hat \delta = \hat\delta(\rho,\mu) > 0$, independent of $\tau \in [1/2,1]$, such that whenever $0 < \delta < \hat\delta$, $G_\tau(u) \in (\pm\p)_{\delta/2}$ for all $u \in \overline{(\pm\p)_\delta} \cap \B_\rho^\mu$.
\end{lemma}

\begin{proof}
	Let $u \in \overline{\p_\delta} \cap \B_\rho^\mu$ where $\delta$ will be determined later. We aim to prove that $w = G_\tau(u)$ belongs to $\p_{\delta/2}$. The Lagrange multiplier $\omega \geq 0$ is given in Proposition \ref{propuniquesolutionw}. As a consequence of Lemma \ref{lemproG2} and Lemma \ref{lemproG3}, we obtain the existence of $\omega^* = \omega^*(\rho,\mu) < +\infty$ independent of $u$ and of $\tau$, such that $\omega < \omega^*$ for all $u \in \B_\rho^\mu$. Setting $w^- = \min\{w,0\}$ and $u^- = \min\{u,0\}$, recalling $\bar\la \geq 0$ in defining the operator $G_\tau$, we have
	\begin{align*}
		\text{dist}(w,\p)\|w^-\| & \leq \|w^-\|^2 \le \int_\Omega\left( |\nabla w^-|^2 + \bar\la(w^-)^2\right) dx \\
		& = \int_\Omega \left( \nabla w \nabla w^- + \bar \la ww^-\right) dx \\
		& = \int_\Omega \left( \tau|u|^{p-2}u w^- + \omega u w^-\right) dx \\
		& \leq \int_\Omega \left( |u^-|^{p-2}u^- w^- + \omega u^- w^-\right) dx \\
		& \leq \left( \|u^-\|_{L^p}^{p-1}\|w^-\|_{L^p} + \omega M_+(\delta)^{p-2}\|u^-\|_{L^p}\|w^-\|_{L^p}\right) \\
		& = \left( \inf_{v \in \p}\|u-v\|_{L^p}^{p-1} + \omega M_+(\delta)^{p-2}\inf_{v \in \p}\|u-v\|_{L^p} \right) \|w^-\|_{L^p} \\
		& \le \left( \mathcal{S}_p^p \text{dist}(u,\p)^{p-1} + \mathcal{S}_p^2\omega M_+(\delta)^{p-2}\text{dist}(u,\p) \right) \|w^-\|
	\end{align*}
    which indicates that
    \begin{align} \label{eqdistwp}
    	\text{dist}(w,\p) \leq \left( \mathcal{S}_p^p\delta^{p-2} + \mathcal{S}_p^2\omega^* M_+(\delta)^{p-2}\right) \delta,
    \end{align}
    where $\mathcal{S}_p$ is the best constant of Sobolev embedding $H_0^1(\Omega) \hookrightarrow L^p(\Omega)$. By Lemma \ref{lempre1} there exists $\hat \delta = \hat \delta (\rho,\mu)$ such that for $0 < \delta < \hat\delta$,
    \begin{align} \label{eqconondelta}
    	\mathcal{S}_p^p\delta^{p-2} + \mathcal{S}_p^2\omega^* M_+(\delta)^{p-2} < \frac12,
    \end{align}
    and we conclude that $w = G_\tau(u) \in \p_{\delta/2}$. A similar argument works for the "$-$" sign and we complete the proof.
\end{proof}

Now let us define a map
\begin{align} \label{def:V}
	V_\tau: S_\mu(\Omega) \to H_0^1(\Omega), \quad u \mapsto V_\tau(u) = u - G_\tau(u),
\end{align}
where $G_\tau$ is defined in \eqref{def:G}. By Lemma \ref{lemproG1} we have $V_\tau(u) \in T_uS_\mu(\Omega)$ for each $u \in S_\mu(\Omega)$. Moreover, we observe that for $u \in S_\mu(\Omega)$, $V_\tau(u) = 0$ is equivalent to that $u$ solves
	 \begin{align} \label{eqequationomegatau}
	 	\begin{cases}
	 		-\Delta u + \la u = \tau|u|^{p-2}u & \quad \text{in } \Omega, \\
	 		u(x) = 0 & \quad \text{on } \partial \Omega,
	 	\end{cases}
	 \end{align}
with $\la = \bar\la - \omega$.

\vskip0.1in

As a consequence of Lemma \ref{lemproG3} we obtain that $V_\tau$ satisfies the Palais-Smale type condition in $\B_\rho^\mu$.

\begin{lemma}[Palais-Smale type condition] \label{lempscondition}
	Let $\tau \in [1/2,1]$ be fixed and let $\{u_n\} \subset \B_\rho^\mu$ be such that $V_\tau(u_n) \to 0$ strongly in $H_0^1(\Omega)$. Then $\{u_n\}$ has a strongly convergent subsequence in $S_\mu(\Omega)$.
\end{lemma}

\begin{proof}
	It is clear that $\{u_n\}$ is bounded in $H_0^1(\Omega)$. Up to a subsequence, we assume that $u_n \to u$ weakly in $H_0^1(\Omega)$. Note that $u \in S_\mu(\Omega)$. By Lemma \ref{lemproG3} we further assume that $G_\tau(u_n) \to G_\tau(u)$ strongly in $H_0^1(\Omega)$. Then we have
	\begin{align*}
		o(1) = \langle V_\tau(u_n),u_n-u \rangle = \langle u_n,u_n-u \rangle - \langle G_\tau(u_n),u_n-u \rangle = \langle u_n,u_n-u \rangle + o(1),
	\end{align*}
	indicating that $\langle u_n,u_n-u \rangle \to 0$, and thus $u_n \to u$ strongly in $H_0^1(\Omega)$.
\end{proof}

\begin{remark}
	When $\bar\la$ is taken as $0$, the operator $V_\tau$ is exactly the gradient of $E_\tau(\cdot,\Omega)$ as a constrained functional to $S_\mu(\Omega)$. In this case, Lemma \ref{lempscondition} shows that $E_\tau(\cdot,\Omega)$, constrained to $S_\mu(\Omega)$, satisfies the Palais-Smale condition corresponding to the gradient operator on any bounded closed subset of $S_\mu(\Omega)$.
\end{remark}

Next we show that $V_\tau$ is a pseudogradient for $E_\tau(\cdot,\Omega)$ constrained to $S_\mu(\Omega)$.

\begin{lemma} \label{lempseudogradient}
	Recall that $\bar \la \geq 0$. We have
	\begin{align*}
		\|V_\tau(u)\|^2 \leq \langle \nabla E_\tau|_{S_\mu(\Omega)}(u,\Omega), V_\tau(u) \rangle \leq \frac{\la_1(\Omega) + \bar\la}{\la_1(\Omega)} \|V_\tau(u)\|^2, \quad \forall u \in S_\mu(\Omega).
	\end{align*}
\end{lemma}

\begin{proof}
	Let $w = G_\tau(u)$ satisfy \eqref{eqequofw1}. Observe that $\int_\Omega u(u-w)dx = \mu - \mu = 0$. Moreover, multiplying \eqref{eqequofw1} by $u-w$ we see that
	\begin{align*}
		\int_\Omega \left( \nabla w \nabla (u-w) + \bar \la w (u-w)\right) dx = \int_\Omega \left( \tau|u|^{p-2}u(u-w) + \omega u(u-w) \right) dx = \tau\int_\Omega |u|^{p-2}u(u-w)dx,
	\end{align*}
	and so
	\begin{align*}
		\int_\Omega \left( \nabla w \nabla (u-w) - \tau|u|^{p-2}u(u-w)\right) dx = - \bar \la \int_\Omega w(u-w)dx = \bar \la \int_\Omega (u-w)^2dx.
	\end{align*}
	Then we have
	\begin{align*}
		\langle \nabla E_\tau|_{S_\mu(\Omega)}(u,\Omega), V_\tau(u) \rangle & = \int_\Omega\left( \nabla u \nabla (u-w) - \tau|u|^{p-2}u(u-w) \right)dx \\
		& = \int_\Omega|\nabla (u-w)|^2dx + \int_\Omega\left( \nabla w \nabla (u-w) - \tau|u|^{p-2}u(u-w) \right)dx \\
		& = \int_\Omega\left( |\nabla (u-w)|^2 + \bar \la (u-w)^2 \right) dx,
	\end{align*}
    completing the proof.
\end{proof}

\begin{lemma} \label{leminf>0}
	Let $0 < \delta < \hat\delta$ where $\hat \delta$ is given in Lemma \ref{lemproG4}. Then $V_\tau(u) \neq 0$ for all $u \in \B_\rho^\mu \cap (D_\delta^*\backslash (-\p \cup \p))$. Moreover, we have
	\begin{align} \label{eqinf>0}
		\inf_{u \in \B_\rho^\mu \cap (\overline{D_\delta^*\backslash D_{\delta/2}^*})}\|V_\tau(u)\| > 0.
	\end{align}
\end{lemma}

\begin{proof}
	For $0 < \delta < \hat \delta$, by Lemma \ref{lemproG4} we have $u \neq G_\tau(u)$ for all $u \in \B_\rho^\mu \cap (D_\delta^*\backslash (-\p \cup \p))$, and so $V_\tau(u) \neq 0$. Then \eqref{eqinf>0} follows from Lemma \ref{lempscondition} since $\overline{D_\delta^*\backslash D_{\delta/2}^*}$ is a close set.
\end{proof}

Now we can state our main result concerning the descending flow and invariance set in this section. Since it will be used for $\tau = 1$ in the next section and Section \ref{secgenus}, we focus on this special case and the proof also works for other fixed $\tau \in [1/2,1]$. In the following we use the notation $V = V_1$.

\begin{proposition} \label{propdescendingflow}
	Given $\mu, \rho > 0$, $\delta \in (0,\hat\delta)$ where $\hat\delta$ is given in Lemma \ref{lemproG4}. Let $S \subset S_\mu(\Omega) \cap S^*(\delta)$, $c \in \R$, $\epsilon, \xi > 0$, such that
	\begin{align} \label{eqlowboundofV}
		\|V(u)\| \geq 2\epsilon /\xi, \quad \forall u \in E^{-1}([c-2\epsilon,c+2\epsilon],\Omega) \cap S(2\xi) \cap S^*(\delta),
	\end{align}
    where
    \begin{align}
    	& S(2\xi) := \bigl\{u \in S_\mu(\Omega): \text{dist}(u,S) \leq 2\xi \bigr\} \subset \B_\rho^\mu, \label{eqS2xi} \\
    	& E^{-1}([c-2\epsilon,c+2\epsilon],\Omega) := \bigl\{u \in S_\mu(\Omega): c-2\epsilon \leq E(u,\Omega) \leq c+2\epsilon \bigr\}.
    \end{align}
    Then there exists $\eta: C([0,\infty) \times S_\mu(\Omega), S_\mu(\Omega))$ such that $\eta(t,-u) = -\eta(t,u)$ and
	\begin{itemize}
		\item[(i)] $\eta(t,u) = u$ for all $t \geq 0$ if $u \notin E^{-1}([c-2\epsilon,c+2\epsilon],\Omega) \cap S(2\xi)$,
		\item[(ii)] $E(\eta(\cdot,u),\Omega)$ is nonincreasing, $\forall u \in S_\mu(\Omega)$,
		\item[(iii)] $\eta(t,u) \in D_\delta^*$ for all $t \geq 0$ if $u \in D_\delta^*$,
		\item[(iv)] $\forall u \in S$ with $E(u,\Omega) \leq c + \epsilon$, either $\eta(\xi,u) \in D_\delta^*$ or $E(\eta(\xi,u),\Omega) \leq c -\epsilon$.
	\end{itemize}
\end{proposition}

\begin{proof}
	First, we consider the case when $V(u) \neq 0$ for all $u \in E^{-1}([c-2\epsilon,c+2\epsilon],\Omega) \cap S(2\xi) \cap D^*_\delta$. Let us define
	\begin{align*}
		& U_1 := \bigl\{u \in S(2\xi): |E(u,\Omega)-c| \leq 2\epsilon \bigr\}, \\
		& U_2 := \bigl\{u \in S(\xi): |E(u,\Omega)-c| \leq \epsilon\bigr\}, \\
		& h(u):= \frac{\text{dist}(u,S_\mu(\Omega)\backslash U_1)}{\text{dist}(u,U_2)+\text{dist}(u,S_\mu(\Omega)\backslash U_1)},
	\end{align*}
	so that $h$ is locally Lipschitz continuous, $h = 1$ on $U_2$ and $h = 0$ on $S_\mu(\Omega)\backslash U_1$. Now we fix $\bar\la = \bar\la(\rho,\mu) \geq 0$ satisfying \eqref{eqcononbarla}. With $V$ given by \eqref{def:V}, we define the locally Lipschitz continuous vector field
	\begin{align*}
		W(u) :=
		\left\{
		\begin{aligned}
			& -h(u)\|V(u)\|^{-1}V(u), & u \in U_1, \\
			& 0, & u \notin U_1.
		\end{aligned}
		\right.
	\end{align*}
    With $W$ in hand we can construct a flow $\eta(t,\cdot) : S_\mu(\Omega) \to S_\mu(\Omega)$ as
    \begin{align} \label{eqflow}
    	\left\{
    	\begin{aligned}
    		& \frac{\partial}{\partial t}\eta(t,u) = W(\eta(t,u)), \\
    		& \eta(0,u) = u.
    	\end{aligned}
    	\right.
    \end{align}
    Note that $\|W\| \leq 1$ on $S_\mu(\Omega)$. For each $u \in S_\mu(\Omega)$, it is standard that the problem \eqref{eqflow} has a unique solution $\eta(t,u)$ for all $t \geq 0$ and $\eta \in C([0,\infty) \times S_\mu(\Omega), S_\mu(\Omega))$. Next we verify that $\eta$ satisfies the desired conditions. Observe that $V(-u) = -V(u)$, and $\eta(t,-u) = -\eta(t,u)$ follows. If $u \notin U_1$, then $W(u) = 0$ and $\eta(t,u) = u$ for all $t \geq 0$. Thus (i) holds. Moreover, we get
    \begin{align*}
    	\frac{d}{dt}E(\eta(t,u),\Omega) & = \langle \nabla E|_{S_\mu(\Omega)}(\eta(t,u),\Omega), \frac{\partial}{\partial t}\eta(t,u) \rangle \\
    	& = -h(\eta(t,u))\|V(\eta(t,u))\|^{-1}\langle \nabla E|_{S_\mu(\Omega)}(\eta(t,u),\Omega), V(\eta(t,u)) \rangle \\
    	& \leq 0,
    \end{align*}
	where we have used Lemma \ref{lempseudogradient}, and this indicates that $E(\eta(\cdot,u),\Omega)$ is nonincreasing for any $u \in S_\mu(\Omega)$.

\vskip0.1in

	Next, we check (iii). By (i), it suffices to consider $u \in U_1$. For such $u$, using (i) again we have $\eta(t,u) \in U_1$,  and so $\eta(t,u) \in \B_\rho^\mu$ for all $t \geq 0$. We will use the framework developed in \cite[Section 4]{JS} to complete the proof of (iii). For convenience of readers, we have reviewed the general setting and results of Section \cite[Section 4]{JS} at the end of Section \ref{secpre}. In view of Lemma \ref{lemproG4}, using Lemma \ref{lemconvex} with $E = H_0^1(\Omega), H = L^2(\Omega), S_\mu = S_\mu(\Omega), \tilde B = \overline{(\pm \p)_\delta}$ and $B_\mu = \overline{(\pm \p)_\delta} \cap S_\mu(\Omega)$, we obtain
	\begin{align*}
		\lim_{s \searrow 0}s^{-1}\text{dist}(u - sV(u), \overline{(\pm \p)_\delta} \cap S_\mu(\Omega)) = 0, \quad \forall u \in \overline{(\pm \p)_\delta} \cap \B_\rho^\mu,
	\end{align*}
    and so
    \begin{align*}
    	\lim_{s \searrow 0}s^{-1}\text{dist}(u + sW(u), \overline{(\pm \p)_\delta} \cap S_\mu(\Omega)) = 0, \quad \forall u \in \overline{(\pm \p)_\delta} \cap \B_\rho^\mu.
    \end{align*}
    Now, for $u \in U_1 \cap \overline{(\pm\p)_\delta}$, recalling that $\eta(t,u) \in \B_\rho^\mu$ for all $t\geq 0$, we apply Proposition \ref{propbmmanifold} and conclude that $\eta(t,u) \in \overline{(\pm\p)_\delta}$, which completes the proof of (iii).

    Finally, we prove (iv). If $u \in S$, we claim that $\eta(\xi,u) \in S(\xi)$. In fact,
    \begin{align*}
    	\|\eta(t,u) - u\| = \|\int_0^t W(\eta(\tau,u))d\tau\| \leq t,
    \end{align*}
    thus the claim holds true. For $u \in S$ with $E(u,\Omega) \leq c + \epsilon$, if there exists $t_0 \in [0,\xi]$ such that $\eta(t_0,u) \in D_\delta^*$ or $E(\eta(t_0,u),\Omega) \leq c - \epsilon$, by (ii) and (iii), we get that $\eta(\xi,u) \in D_\delta^*$ or $E(\eta(\xi,u),\Omega) \leq c - \epsilon$. For this reason, arguing by contradiction we can assume that $\eta(t,u) \in U_2$, $E(\eta(t,u)) > c - \epsilon$, and $\eta(t,u) \notin D_\delta^*$ for all $t \in [0,\xi]$. Then we get a self-contradictory inequality
    \begin{align*}
    	c - \epsilon < E(\eta(\xi,u),\Omega) & = E(u,\Omega) + \int_0^\xi \langle \nabla E|_{S_\mu(\Omega)}(\eta(t,u),\Omega), \frac{\partial}{\partial t}\eta(t,u) \rangle dt \\
    	& = E(u,\Omega) - \int_0^\xi \|V(\eta(t,u))\|^{-1}\langle \nabla E|_{S_\mu(\Omega)}(\eta(t,u),\Omega), V(\eta(t,u)) \rangle dt \\
    	& \leq c + \epsilon - \xi \frac{2\epsilon}{\xi} = c - \epsilon.
    \end{align*}

    When $\K := \{u \in U_1 \cap D_\delta^*: V(u) = 0\}$ is not empty, by \eqref{eqlowboundofV} and the continuity of $V$, there exist neighborhoods $\mathcal{N}_2 \supset \mathcal{N}_1 \supset \K$ in $S_\mu(\Omega)$, such that $\mathcal{N}_2 \cap U_1 \subset D_\delta^*$. In this case, the definition of $W$ is modified as
    \begin{align*}
    	W(u) :=
    	\left\{
    	\begin{aligned}
    		& -h(u)y(u)\|V(u)\|^{-1}V(u), & u \in U_1 \text{ and } u \notin \mathcal{N}_1, \\
    		& 0, & u \notin U_1 \ \text{ or } \ u \in \mathcal{N}_1,
    	\end{aligned}
    	\right.
    \end{align*}
    where $y : S_\mu(\Omega) \to [0,1]$ is a locally Lipschitz continuous map such that
    \begin{align*}
    	y(u) :=
    	\left\{
    	\begin{aligned}
    		& 1, & u \notin \mathcal{N}_2, \\
    		& 0, & u \in \mathcal{N}_1.
    	\end{aligned}
    	\right.
    \end{align*}
    Then following the proof of the case $\K = \emptyset$ step by step we can complete the proof.
\end{proof}

\section{A tool to localize bounded Palais-Smale sequences with Morse type information for constrained functionals} \label{sectool}

 The Theorem 2.1 of \cite{BCJN} and its symmetric version Theorem 4.2  of \cite{BCJN} provided a powerful tool to obtain bounded Palais-Smale sequences with Morse type information for some constrained functional having minimax geometry. However, one does not know where the Palais-Smale sequence is localized at. In this section, we will introduce certain conditions for some subsets, which can be checked by results in Section \ref{secinv} in our applications, and prove a theorem to localize the bounded Palais-Smale sequences with Morse type information. This result is an enhanced version of \cite[Theorem 4.2]{BCJN}, which can be seen as a special case with the subsets that we choose being empty. A non-symmetric version can be provided similarly, whose details are not presented since we address problems with a symmetric setting in this paper. As an application, this abstract tool will be used to search for sign-changing critical points with estimates of Morse index for $L^2$-constrained functionals.

\vskip0.1in

We recall the general setting used in \cite{BCJN}. Let $\left(E,\langle \cdot, \cdot \rangle\right)$ and $\left(H,(\cdot,\cdot)\right)$ be two \emph{infinite-dimensional} Hilbert spaces and assume that $E\hookrightarrow H \hookrightarrow E'$,
with continuous injections.  For simplicity, assume that the
continuous injection $E\hookrightarrow H$ has norm at most $1$ and
identify $E$ with its image in $H$. We also introduce
\begin{align*}
	\left\{
	\begin{aligned}
		\|u\|^2 & = \langle u,u \rangle, \\
		|u|^2 \ & = (u,u), \\
	\end{aligned}
	\right.
	\quad \quad u \in E,
\end{align*}
and define for $\mu>0$:
\begin{equation*}
	S_\mu = \bigl\{ u \in E: |u|^2=\mu \bigr\}.
\end{equation*}
Clearly, $S_{\mu}$ is a smooth submanifold of $E$ of codimension
$1$. Furthermore its tangent space at a given point $u \in S_{\mu}$
can be considered as the closed subspace of codimension $1$ of $E$
given by:
\begin{equation*}
	T_u S_{\mu} = \bigl\{v \in E : (u,v) =0 \bigr\}.
\end{equation*}
For $u \in E$ we denote with
\begin{align*}
	B(u,R) = \bigl\{v \in E: \|v - u\| < R\bigr\}, \quad B(S_\mu;u,R) = S_\mu \cap B(u,R).
\end{align*}
In the following definition, we denote  $\|\cdot\|_*$ and
$\|\cdot\|_{**}$, respectively, the operator norm of $\mathscr{L}(E,\R)$ and of
$\mathscr{L}(E,\mathscr{L}(E,\R))$.

\begin{definition}\label{Holder-continuous}
	Let $\phi : E \rightarrow \mathbb{R}$ be a $C^2$-functional on $E$
	and $\alpha \in (0,1]$. We say that $\phi'$ and $\phi''$
	are \emph{$\alpha$-H\"older continuous on bounded sets} if for any
	$R>0$, one can find $M=M(R)>0$ such that, for any
	$u_1,u_2\in B(0,R)$:
	\begin{equation} \label{eqalpha}
		\|\phi'(u_1)-\phi'(u_2)\|_* \leq M \|u_1-u_2\|^{\alpha}, \quad
		\|\phi''(u_1)-\phi''(u_2)\|_{**} \leq M\|u_1-u_2\|^\alpha.
	\end{equation}
\end{definition}

Let $P^* \subset P^{**}$ be open sets in $S_\mu$ and satisfy that $-P^{*} = P^*$ and $-P^{**} = P^{**}$. To state our location result, Theorem \ref{thmuepsym}, we need to introduce two assumptions.

\begin{itemize}
	\item[(A1)] For a $C^1$-functional $\phi$ satisfying that $\phi|_{S_\mu}$ is even, and $R < +\infty$, if there exists a pseudogradient $V$ for $\phi$ constrained to $S_\mu$ satisfying
	\begin{align} \label{eqC0}
		C_0\|\phi'|_{S_\mu}\| \leq \|V(u)\| \leq \|\phi'|_{S_\mu}\|, \forall u \in B(S_\mu;0,R), \quad \text{for some } C_0 \in (0,1],
	\end{align}
	taking $S \subset S_\mu \backslash \overline{P^{**}}$, $c \in \R$, $\ep, \xi > 0$, such that $S(2\xi) \subset B(S_\mu;0,R) \cap (S_\mu \backslash P^*)$ and
	\begin{align*}
		\|V(u)\| \geq 2\epsilon /\xi, \quad \forall u \in \phi|_{S_\mu}^{-1}([c-2\epsilon,c+2\epsilon]) \cap S(2\xi) \cap (S_\mu \backslash \overline{P^{**}}),
	\end{align*}
	where
	\begin{align*}
		& S(2\xi) := \bigl\{u \in S_\mu: \text{dist}(u,S) \leq 2\xi \bigr\}, \\
		& \phi|_{S_\mu}^{-1}([c-2\epsilon,c+2\epsilon]) := \bigl\{u \in S_\mu: c-2\epsilon \leq \phi(u) \leq c+2\epsilon \bigr\},
	\end{align*}
	then there exists $\eta: C([0,\xi] \times S_\mu, S_\mu)$ such that $\eta(t,-u) = -\eta(t,u)$ and
	\begin{itemize}
		\item[(i)] $\eta(t,u) = u$ for $t \in [0,\xi]$ if $u \notin \phi|_{S_\mu}^{-1}([c-2\epsilon,c+2\epsilon]) \cap S(2\xi)$,
		\item[(ii)] $\phi(\eta(\cdot,u))$ is nonincreasing, $\forall u \in S_\mu$,
		\item[(iii)] $\eta(t,u) \in \overline{P^{**}}$ for $t \in [0,\xi]$ if $u \in \overline{P^{**}}$,
		\item[(iv)] $\forall u \in S$ with $\phi(u) \leq c + \epsilon$, either $\eta(\xi,u) \in \overline{P^{**}}$ or $\phi(\eta(\xi,u)) \leq c -\epsilon$.
	\end{itemize}
	\item[(A2)] For a $C^1$-functional $\phi$ satisfying that $\phi|_{S_\mu}$ is even, and $R < +\infty$, assume that $V$, given in (A1), exists, then
	\begin{align*}
		\inf_{u \in B(S_\mu;0,R) \cap \overline{P^{**}} \backslash P^*}\|V(u)\| > 0.
	\end{align*}
\end{itemize}

\begin{definition}
	We consider the action of $\mathbb{Z}_2$ on $\R^d$ determined by an \emph{isometric} involution of $\R^d$ with its usual distance, that we denote by $\sigma$. For any subset $D \subset \R^d$ we denote by:
	\begin{align*}
		\sigma (D) = \bigl\{\sigma(x): x \in D\bigr\}.
	\end{align*}
	A subset $D$ is invariant or stable if $\sigma (D) = D$, in this case, a continuous map $f : D \to S_\mu$ is said to be equivariant if:
	\begin{align*}
		f \circ \sigma = -f.
	\end{align*}
\end{definition}

\begin{theorem} \label{thmuepsym}
	Let $\phi$ be a $C^2$-functional on $E$, such that $\phi'$ and $\phi''$ are $\alpha$-H\"older continuous on bounded sets for some $\alpha \in (0,1]$, and assume that $\phi|_{S_\mu}$ is even; let $R > 1$ and $0 < \alpha_1 \leq \frac{\alpha}{2(\alpha + 2)} < 1$. Moreover, suppose assumptions (A1) and (A2) hold true with respect to $\phi$ and $R$.
	
	Assume that $D$ is a compact subset of $\R^d$, and $c \in \R$ is a number satisfying
	\begin{align} \label{eqsup>c}
		c \le \sup\bigl\{\phi(u): u \in \eta(\xi,g(D)), u \notin \overline{P^{**}} \bigr\}, \quad \text{ for any equivariant } g \in C(D,S_\mu).
	\end{align}
	Then, for $\ep >0$ and for any equivariant $f \in C(D,S_\mu)$ with  $\sup_{u \in f(D) \backslash (f(D) \cap \overline{P^{**}})}\phi(u)\leq c + \ep$, satisfying
	\begin{align} \label{eqboundedballsym}
		K := \bigl\{u \in f(D) \backslash (f(D) \cap P^{**}): \phi(u) \geq c -\ep \bigr\} \subset B(0,R-1),
	\end{align}
	assuming that $\ep > 0$ is sufficiently small there exists $u_\ep \in S_\mu$ such that
	\begin{itemize}
		\item[(1)] $c - \ep \leq \phi(u_\ep) \leq c + \ep$;
		\item[(2)] $\|V(u_\ep)\| \leq 3\ep^{\alpha_1}$ where $V$ is the pseudogradient for $\phi$ constrained to $S_\mu$ given in (A1);
		\item[(3)] $u_\ep \in f(D) \backslash (f(D) \cap \overline{P^{**}})$;
		\item[(4)] If $D^2\phi(u_\ep)[w, w] < -\ep^{\alpha_1} \|w\|^2$ for all $w \neq 0$ in a subspace $W$ of $T_{u_\ep}S_\mu$, then $\dim W \leq d$.
	\end{itemize}
	By symmetry this also holds for $-u_\ep$.
\end{theorem}

Before proving Theorem \ref{thmuepsym}, we recall a result given in \cite{BCJN}.

\begin{lemma}[Lemma 4.5 in \cite{BCJN}] \label{lemflowdownsym}
	Let $\phi : E \to \R$ be a $C^2$-functional such that $\phi'$ and $\phi''$ are $\alpha$-H\"older continuous ($\alpha \in (0, 1]$) on bounded sets and $\phi|_{S_\mu}$ is even; let $f$ be a continuous equivariant map from a compact subset $D$ of $\R^d$ into $S_\mu$. Suppose $K_2$ is an invariant, compact subset of $D$ with the following properties:
	\begin{itemize}
		\item There exists $R > 1$ such that $f(K_2) \subset B(0, R - 1) \cap S_\mu$.
		\item There exists a constant $\beta > 0$ such that for all $y \in K_2$, there is a subspace $W_y$, of $T_{f(y)}S_\mu$ with $\dim W_y \geq d + 1$ so that
		\begin{align}
			D^2\phi(f(y))[w,w] < -\beta \|w\|^2, \quad \forall w \in W_y \backslash \{0\}.
		\end{align}
	\end{itemize}
	Then for any $0 < \delta \leq \delta_4$ where $\delta_4$ depends on $\mu, \beta, R, d$, ($\delta_4 = \min\{\delta_3,2\sqrt{2\mu}\}$ and $\delta_3$ is introduced in (2.18) of \cite{BCJN}) and $b > 0$ there is a continuous equivariant map $\hat f : D \to S_\mu$ such that:
	\begin{itemize}
		\item[(i)] $\hat f(x) = f(x)$ for $x \in D \backslash N_b(K_2)$ where $N_b(K_2)$ is the $b$-neighbourhood of $K_2$ in $\R^d$;
		\item[(ii)] $\phi(\hat f(x)) \leq \phi(f(x))$ for all $x \in D$;
		\item[(iii)] if $x \in K_2$, then $\hat f(x) < f(x) - \frac{\beta \delta^2}{864N_d^2}$, where $N_d$ is an integer depending on $d$;
		\item[(iv)] $\|\hat f(x) - f(x)\| \leq \frac{\delta}2$ for all $x \in D$.
	\end{itemize}
\end{lemma}

Now we are prepared to prove Theorem \ref{thmuepsym}.

\begin{proof}[Proof of Theorem \ref{thmuepsym}]
	Let $\delta_3 > 0$ satisfy the equality (2.18) in \cite{BCJN} and $\delta_4 = \min\{\delta_3,2\sqrt{2\mu}\}$. Observe that, we can assume that $\alpha \leq \frac12$. Then, similar to the proof of \cite[Theorem 2.1]{BCJN}, we take a constant $\ga > 0$ such that $\ga \beta^{\frac1\alpha} \leq \delta_4$ for any $\beta > 0$ small enough. Now let
	\begin{align} \label{eqdeltaep}
		\delta := \delta(\ep) = \frac12 \min\left\{ \left( \frac{\ep^{\alpha_1}}{M}\right) ^{\frac1\alpha}, \ga \ep^{\frac{\alpha_1}{\alpha}}\right\}
	\end{align}
	where $M = M(R)$ is given in Definition \ref{Holder-continuous}. Observe that $\delta \leq \delta_4$ when $\beta > 0$ is given by $\beta = \ep^{\alpha_1}$ with $\ep > 0$ small enough.
	
	Recall that $P^{**}$ is open and we consider the closed set
	\begin{align}
		K = \bigl\{x \in D: f(x) \notin P^{**}, \phi(f(x)) \geq c - \ep\bigr\}.
	\end{align}
	It is clear that $K$ is invariant since $\phi$ is even. Since $D$ is a compact subset of $\R^d$, $K$ is also compact.
	
	At this point, arguing by contradiction, we suppose that for all $x \in K$, either $\|V(f(x))\| > 3\ep^{\alpha_1}$ or there is a subspace $W_x$ of $T_{f(x)}S_\mu$ with $\dim W_x \geq d+1$ such that for all $w \in W_x \backslash \{0\}$, we have $D^2 \phi(f(x))[w,w] < -\ep^{\alpha_1}\|w\|^2$.
	
	In view of the assumption \eqref{eqalpha} on $\phi'$ and of the definition of $\delta > 0$ given in \eqref{eqdeltaep}, setting $\hat \delta = 2\delta$ we deduce that for all $x \in K$ such that $\|\phi'|_{S_\mu}(f(x))\| > 3\ep^{\alpha_1}$ we have $\|\phi'|_{S_\mu}(u)\| > \ep^{\alpha_1}$ for all $u \in B(S_\mu; f(x), \hat \delta)$. Recalling \eqref{eqC0} and \eqref{eqboundedballsym}, for all $x \in K$ such that $\|V(f(x))\| > 3\ep^{\alpha_1}$ we have $\|V(u)\| > C_0\ep^{\alpha_1}$ for all $u \in B(S_\mu; f(x), \hat \delta)$, where $C_0 \in (0,1]$ is given in \eqref{eqC0}. Let
	\begin{align*}
		T_1 = \bigl\{x \in K: \|V(u)\| > C_0\ep^{\alpha_1} \text{ for all } u \in B(S_\mu; f(x), \hat \delta)\bigr\}
	\end{align*}
	and $T_2 = K\backslash T_1$. Note that $\overline{T}_1, \overline{T}_2$ are compact and invariant, and $K = \overline{T}_1 \cup \overline{T}_2$. By (A2) we have
	\begin{align*}
		\inf_{u \in B(S_\mu;0,R) \cap \overline{P^{**}} \backslash P^*}\|V(u)\| > 0.
	\end{align*}
	Taking $\ep > 0$ small enough such that
	\begin{align*}
		3\ep^{\alpha_1} < \inf_{u \in B(S_\mu;0,R) \cap \overline{P^{**}} \backslash P^*}\|V(u)\|.
	\end{align*}
	Next, we claim that
	\begin{align} \label{eqT2faraway}
		\text{dist}(f(\overline{T}_2),B(S_\mu;0,R) \cap \overline{P^{**}} \backslash P^*) > 0.
	\end{align}
	By negation, if \eqref{eqT2faraway} does not hold true, there exist $\{x_n\} \subset \overline{T}_2$ and $\{v_n\} \subset B(S_\mu;0,R) \cap \overline{P^{**}} \backslash P^*)$ such that
	\begin{align*}
		\|f(x_n) - v_n\| \to 0 \quad \text{as } n \to \infty.
	\end{align*}
	Since $\overline{T}_2$ is compact, there exists $x \in \overline{T}_2$ such that $\lim_{n \to \infty}x_n = x$, and so $f(x_n) \to f(x)$. Moreover, $v_n \to f(x)$ and this shows that $f(x) \in \overline{P^{**}} \backslash P^*$. Observing that $f(\overline{T}_2) \subset B(S_\mu;0,R-1)$, we get $f(x) \in B(S_\mu;0,R) \cap \overline{P^{**}} \backslash P^*)$ and thus
	\begin{align*}
		\|V(f(x))\| \geq \inf_{u \in B(S_\mu;0,R) \cap \overline{P^{**}} \backslash P^*}\|V(u)\| > 3\ep^{\alpha_1}.
	\end{align*}
	Then we can take a neighborhood $U_x$ of $x$ in $\R^d$ such that
	\begin{align*}
		\|V(f(y))\| > 3\ep^{\alpha_1} \quad \text{for all } y \in U_x.
	\end{align*}
	This indicates that $K \cap U_x \subset T_1$. Hence, $\overline{T}_2 \cap U_x  = \emptyset$, contradicting the convergence of $x_n$ to $x$ in $\R^d$. We complete the proof of \eqref{eqT2faraway}. Now apply Lemma \ref{lemflowdownsym} with $K_2 = \overline T_2$, $\beta = \ep^{\alpha_1}$ and $b > 0$ satisfying
	\begin{align*}
		f(N_b(\overline T_2) \cap D) \subset B(0,R-\frac12) \quad \text{and} \quad f(N_b(\overline T_2) \cap D) \cap \overline{P^{**}} = \emptyset,
	\end{align*}
	to obtain a continuous equivariant map $g : D \to S_\mu$ such that
	\begin{align}
		g(x) = f(x) & \quad \text{for } x \in D \backslash N_b(\overline T_2) \quad \text{and} \quad \phi(g(x)) \leq \phi(f(x)) \quad \text{for } x \in D, \\
		& \phi(g(x)) \leq \phi(f(x)) - \frac{\ep^{\alpha_1}\delta^2}{864N_d^2} \quad \text{for } x \in \overline T_2, \label{eqg(x)down} \\
		& \|g(x) - f(x)\| \leq \frac\delta 2 \quad \text{for } x \in N_b(\overline T_2) \cap D. \label{eqgclosef}
	\end{align}
	Observe that inequality \eqref{eqgclosef} yields that for $x \in \overline T_1$, $B(S_\mu; g(x),\delta) \subset B(S_\mu;f(x),\hat\delta)$, and it is not restrictive to assume that
	\begin{align*}
		g(N_b(\overline T_2) \cap D) \subset B(0,R-\frac13) \quad \text{and} \quad g(N_b(\overline T_2) \cap D) \cap \overline{P^{**}} = \emptyset
	\end{align*}
	for sufficiently small $\ep > 0$. Let \begin{align*}
		\tilde K = \bigl\{x \in D: g(x) \notin P^{**}, c-\ep \leq \phi(g(x)) \leq c+\ep\bigr\}.
	\end{align*}
	If $x \in D \backslash K$, then $f(x) \in P^{**}$ or $\phi(f(x)) < c-\ep$. When the former holds, then $x \in D \backslash N_b(\overline T_2)$ and thus $g(x) = f(x) \in P^{**}$, implying $x \notin \tilde K$. When the latter holds, then $\phi(g(x)) \leq \phi(f(x)) < c -\ep$ and so $x \notin \tilde K$. Namely, $\tilde K \subset K$. Since $\alpha_1 < \frac{\alpha}{\alpha+2}$, in view of \eqref{eqdeltaep}, we have
	\begin{align} \label{eqdownmorethan2ep}
		\frac{\ep^{\alpha_1}\delta^2}{864N_d^2} > 2\ep
	\end{align}
	holds when $\ep > 0$ is small enough. Then, if $x \in \overline T_2$, \eqref{eqg(x)down} shows that $x \notin \tilde K$, and we conclude that $\tilde K \subset \overline T_1$. Now with $S = g(\tilde K)$, we take $\xi = \frac\delta2$, observing that since
	$\alpha_1 < \frac{\alpha}{\alpha +1}$ it holds that
	\begin{align*}
		C_0\ep^{\alpha_1} \geq \frac{4\ep}{\delta} \quad \text{if } \ep > 0 \text{ is sufficiently small},
	\end{align*}
	by using (A1) we deduce that there exists $\eta: C([0,\xi] \times S_\mu, S_\mu)$ such that $\eta(t,-u) = -\eta(t,u)$ and
	\begin{itemize}
		\item[(i)] $\eta(t,u) = u$ for $t \in [0,\xi]$ if $u \notin \phi|_{S_\mu}^{-1}([c-2\epsilon,c+2\epsilon]) \cap S(2\xi)$,
		\item[(ii)] $\phi(\eta(\cdot,u))$ is nonincreasing, $\forall u \in S_\mu$,
		\item[(iii)] $\eta(t,u) \in \overline{P^{**}}$ for $t \in [0,\xi]$ if $u \in \overline{P^{**}}$,
		\item[(iv)] $\forall u \in S$ with $\phi(u) \leq c + \epsilon$, either $\eta(\xi,u) \in \overline{P^{**}}$ or $\phi(\eta(\xi,u)) \leq c -\epsilon$.
	\end{itemize}
	If $x \in D$ with $g(x) \in P^{**}$, by using property (iii) of $\eta$ we conclude that $\eta(\xi,g(x)) \in \overline{P^{**}}$. If $x \in D$ with $\phi(g(x)) < c -\ep$, then by using property (ii) of $\eta$ we obtain that $\phi(\eta(\xi,g(x))) < c -\ep$. Otherwise, $x \in \tilde K$, by using property (iv) of $\eta$ we deduce that either $\eta(\xi,g(x)) \in \overline{P^{**}}$ or $\phi(\eta(\xi,g(x))) \leq c-\ep$. Thus we get that
	\begin{align*}
		\sup\bigl\{\phi(u): u \in \eta(\xi,g(D)), u \notin \overline{P^{**}} \bigr\} \le c-\ep
	\end{align*}
	contradicting \eqref{eqsup>c}, and the proof is completed.
\end{proof}

\section{Sign-changing solutions obtained by genus theory} \label{secgenus}

In this section, we are devoted to search for sign-changing solutions by using genus theory and provide the proof of Theorem \ref{thmsmall}. For this, we introduce the following notions.

\begin{definition} \label{defgenus}
	Let $A \subset H_0^1(\Omega)$ be a closed set, symmetric with respect to the origin (i.e., $-A = A$). We define the genus $\ga$ of $A$ as
	\begin{align*}
		\ga(A) := \sup\bigl\{ m : \exists h \in C(\s^{m-1},A), h(-s) = -h(s) \text{ for any } s \in \s^{m-1} \bigr\}.
	\end{align*}
	Furthermore, we define
	\begin{align*}
		& \Sigma_{\mu} = \bigl\{A \subset S_\mu(\Omega): A \text{ is closed and } -A = A\bigr\}, \quad \quad \Sigma_{\mu}^{(k)} = \bigl\{A \in \Sigma_{\mu}: \ga(A) \geq k\bigr\}; \\
		& \Sigma_{\rho,\mu} = \bigl\{A \subset \overline{\B_\rho^\mu}: A \text{ is closed and } -A = A\bigr\}, \quad \quad \Sigma_{\rho,\mu}^{(k)} = \bigl\{A \in \Sigma_{\rho,\mu}: \ga(A) \geq k\bigr\}.
	\end{align*}
\end{definition}

\begin{lemma} \label{lemsigmanotempty}
	If $\la_k(\Omega) \mu \leq \rho$, then $\Sigma_{\rho,\mu}^{(k)}$ is not empty. Particularly, $\Sigma_{\mu}^{(k)}$ is always non-empty.
\end{lemma}

\begin{proof}
	Let $\phi_1, \cdots, \phi_k$ be the eigenfunctions corresponding to $\la_1(\Omega), \cdots \la_k(\Omega)$ and define
	\begin{align} \label{eqdefMk}
		\mathbb{M}_k := \left\{\sum_{i=1}^ka_i\sqrt{\mu}\phi_i: \sum_{i=1}^ka_i^2 = 1\right\} \subset S_\mu(\Omega).
	\end{align}
    Clearly, $\mathbb{M}_k$ is closed and there exists a function $h \in C(\s^{k-1},\mathbb{M}_k)$ defined by
    \begin{align*}
    	h(a_1,\cdots,a_k) = \sum_{i=1}^ka_i\sqrt{\mu}\phi_i,
    \end{align*}
    and $h(-s) = -h(s)$ for any $s \in \s^{k-1}$. Thus $\ga(\mathbb{M}_k) \geq k$. In fact, $\mathbb{M}_k \cong \s^{k-1}$ and it is well known that $\ga(\mathbb{M}_k) = k$. Moreover, from $\mu \la_k(\Omega) \leq \rho$ we see that $\mathbb{M}_k \subset \overline{\B_\rho^\mu}$ and complete the proof.
\end{proof}

\begin{lemma} \label{lemAcapSnotemp}
	Let $\delta$ satisfy \eqref{eqconondelta0} and $k \geq 2$. Then, for any $A \in \Sigma_{\mu}^{(k)}$ or $A \in \Sigma_{\rho,\mu}^{(k)}$ we have $A \cap S^*(\delta) \neq \emptyset$.
\end{lemma}

\begin{proof}
	Let $A$ be any set belonging to $\Sigma_{\mu}^{(k)}$. There exists $h \in C(\s^{k-1},A)$, with $h(-s) = -h(s)$ for any $s \in \s^{k-1}$. By negation, we suppose $A \cap S^*(\delta) = \emptyset$. Thus we can take $s_0 \in \s^{k-1}$ such that $h(s_0) \in \p_\delta$ and $h(-s_0) \in -\p_\delta$. Since $k \ge 2$, the set $\s^{k-1}$ is path connected. Taking a path connecting $s_0$ and $-s_0$, we can find $s^*$ on this path such that $h(s^*) \in \overline{\p_\delta} \cap \overline{-\p_\delta}$. Observe that $h(s^*) \in A \subset S_\mu(\Omega)$, using Lemma \ref{lemcapnonemp} we get a contradiction. Since $\Sigma_{\rho,\mu}^{(k)} \subset \Sigma_{\mu}^{(k)}$, for any $A \in \Sigma_{\rho,\mu}^{(k)}$ we also have $A \cap S^*(\delta) \neq \emptyset$ and complete the proof.
\end{proof}

\begin{lemma} \label{lemnonempty}
	Let $k \geq 2$. Then, for any $A \in \Sigma_{\mu}^{(k)}$ or $A \in \Sigma_{\rho,\mu}^{(k)}$ we have $A \cap \text{span}\{\phi_1,\cdots,\phi_{k-1}\}^{\bot} \neq \emptyset$.
\end{lemma}

\begin{proof}
	Reasoning as mentioned at the end of the proof of Lemma \ref{lemAcapSnotemp}, we just prove the case of $A \in \Sigma_{\mu}^{(k)}$. Arguing by contradiction, we suppose there exists $A \in \Sigma_{\mu}^{(k)}$ such that $A \cap \text{span}\{\phi_1,\cdots,\phi_{k-1}\}^{\bot} = \emptyset$. Let us define a map by
	\begin{align*}
		P : S_\mu(\Omega) \to \s^{k-2}, \quad u \mapsto (l_u\int_\Omega \phi_1 udx, \cdots, l_u\int_\Omega \phi_{k-1} udx) \quad \text{where } l_u = \left( \sum_{i=1}^{k-1}\left( \int_\Omega \phi_i udx\right)^2 \right) ^{-\frac12}.
	\end{align*}
    Note that $P(u)$ is well-defined and continuous in a small neighborhood of $u \in S_\mu(\Omega) \backslash (\text{span}\{\phi_1,\cdots,\phi_{k-1}\}^{\bot}  \cap S_\mu(\Omega))$. Moreover, $P(-u) = -P(u)$. Since $A \in \Sigma_{\mu}^{(k)}$, there exists $h \in C(\s^{k-1},A)$, with $h(-s) = -h(s)$ for any $s \in \s^{k-1}$. Thus the map $P \circ h \in C(\s^{k-1},\s^{k-2})$ satisfies $p \circ h (-s) = - p \circ h (s)$ for any $s \in \s^{k-1}$. This is impossible by Borsuk-Ulam theorem.
\end{proof}

\subsection{$L^2$-subcritical case}

Let $2 < p < p_c$. By Lemma \ref{lemsigmanotempty} and Lemma \ref{lemAcapSnotemp} we can define the following minimax value:
\begin{align} \label{eqdefofnu}
	\nu_{\mu,k,\delta} := \inf_{A \in \Sigma_{\mu}^{(k)}}\sup_{u \in A \cap S^*(\delta)}E(u,\Omega).
\end{align}
We will prove $\nu_{\mu,k,\delta}$ is indeed a critical value with $\delta > 0$ sufficiently small. More precisely, we have the following theorem.

\begin{theorem} \label{thmgenussub}
	Let $2 < p < p_c$. For $k \geq 2$ and $\mu > 0$, there exists $\tilde \delta = \tilde \delta(\mu,k) > 0$ such that when $0 < \delta < \tilde \delta$, there exists a sign-changing constrained critical point $u_{\mu,k}$ of $E(\cdot,\Omega)$ on $S_\mu(\Omega)$ at the level $\nu_{\mu,k,\delta}$. The function $u_{\mu,k}$ solves \eqref{eqonbounddomain} with $\la = \la_{\mu,k}$ for some $\la_{\mu,k} \in \R$ and $m(u_{\mu,k}) \le k+1$. Moreover, the following hold:
	\begin{itemize}
		\item when $\mu$ is fixed, for any sequence $\{\delta_k\}_{k \ge 2}$ such that $0 < \delta_k < \tilde \delta$, we have $\nu_{\mu,k,\delta_k} \to +\infty$ as $k \to +\infty$;
		\item when $k$ is fixed, for any sequence $\{\delta_\mu\}$ such that $0 < \delta_\mu < \tilde \delta$, we have $\nu_{\mu,k,\delta_\mu} \to 0$ as $\mu \to 0^+$ and $\nu_{\mu,k,\delta_\mu} \to -\infty$ as $\mu \to +\infty$.
	\end{itemize}
\end{theorem}

\begin{proof}
	By Gagliardo-Nirenberg inequality \eqref{eqgndomain} we have
	\begin{align} \label{eq5.3}
		\frac12 \int_\Omega |\nabla u|^2dx \ge E(u,\Omega) \geq \frac12 \int_\Omega |\nabla u|^2dx - \frac1p C_{p,N}^p \mu^{\frac{p(1-\ga_p)}{2}}\left( \int_\Omega |\nabla u|^2dx\right) ^{\frac{p\ga_p}{2}}, \quad \forall u \in S_\mu(\Omega),
	\end{align}
    where $\ga_p = N(\frac12-\frac1p)$. Note that $\frac{p\ga_p}{2} \in (0,1)$ when $2 < p < p_c$. Thus $E(\cdot, \Omega)$ is bounded from below on $S_\mu(\Omega)$ and
    \begin{align} \label{eqcoer}
    	E(u,\Omega) \to +\infty \quad \Leftrightarrow \quad \|u\| \to +\infty.
    \end{align}
    Together with Lemma \ref{lemsigmanotempty} and Lemma \ref{lemAcapSnotemp}, the value $\nu_{\mu,k,\delta}$ is well defined and finite if we assume \eqref{eqconondelta0}. By definition of $\nu_{\mu,k,\delta}$, for any $\ep > 0$, there exists $A \in \Sigma_{\mu}^{(k)}$ such that $\sup_{u \in A \cap S^*(\delta)}E(u,\Omega) < \nu_{\mu,k,\delta} + \ep$. Taking $\ep \leq 1$, there exists $\rho > 0$ independent of $\ep$ such that $A \cap S^*(\delta) \subset \B_{\rho-1}^\mu$. In fact, by \eqref{eqcoer} we can take $\rho > 0$ such that
    \begin{align*}
    	\bigl\{u \in S_\mu(\Omega) : E(u,\Omega) \leq \nu_{\mu,k} + 1\bigr\} \subset \B_{\rho-1}^\mu,
    \end{align*}
    where
    \begin{align}
    	\nu_{\mu,k} := \inf_{A \in \Sigma_{\mu}^{(k)}}\sup_{u \in A}E(u,\Omega)
    \end{align}
    is well-defined, finite and independent of $\delta$. It is clear that $\nu_{\mu,k,\delta} \leq \nu_{\mu,k}$. Thus $\rho$, whose choice depends on $k$ and on $\mu$ but is independent of $\delta$, satisfies our requirement. Then, we aim to apply Proposition \ref{propdescendingflow} with $c  = \nu_{\mu,k,\delta}$ and $S = A \cap S^*(\delta)$ with sufficiently small $\delta > 0$ depending on $\rho$ and on $\mu$. Let us take $\xi > 0$ such that $S(2\xi) \subset \B_{\rho}^\mu$, where the definition of $S(2\xi)$ is given in \eqref{eqS2xi}. We claim that there exists $u_\ep \in E^{-1}([c-2\epsilon,c+2\epsilon],\Omega) \cap S(2\xi) \cap S^*(\delta)$ such that $\|V(u_\ep)\| < 2\epsilon /\xi$, where $V$ is the pseudogradient defined in \eqref{def:V} with $\tau = 1$. Otherwise, \eqref{eqlowboundofV} holds true and by Proposition \ref{propdescendingflow}, there exists $\hat \delta = \hat \delta (\rho,\mu)$ such that for $0 < \delta < \hat \delta$, there exists $\eta: C([0,\infty) \times S_\mu(\Omega), S_\mu(\Omega))$ such that $\eta(t,-u) = \eta(t,u)$ and
    \begin{itemize}
    	\item[(i)] $\eta(t,u) \in D_\delta^*$ for all $t \geq 0$ if $u \in D_\delta^*$,
    	\item[(ii)] $\forall u \in S$ with $E(u,\Omega) \leq \nu_{\mu,k,\delta} + \epsilon$, either $\eta(\xi,u) \in D_\delta^*$ or $E(\eta(\xi,u),\Omega) \leq \nu_{\mu,k,\delta} -\epsilon$.
    \end{itemize}
    It is clear that $\eta(\xi,A) \in \Sigma_{\mu}^{(k)}$ and so $\sup_{u \in \eta(\xi,A) \cap S^*(\delta)}E(u,\Omega) \geq \nu_{\mu,k,\delta}$. On the other hand, by properties (i) and (ii) we deduce that $\sup_{u \in \eta(\xi,A) \cap S^*(\delta)}E(u,\Omega) \leq \nu_{\mu,k,\delta} - \ep < \nu_{\mu,k,\delta}$. This is a contradiction and validates the claim. We take $\delta \in (0,\tilde \delta)$ where
    \begin{align*}
    	\tilde \delta = \tilde \delta(\mu,k) = \min\biggl\{\sqrt{\frac{\la_1(\Omega)}{2}\mu},\hat \delta(\rho(\mu,k),\mu)\biggr\}.
    \end{align*}
    At this point, with $\ep_n \to 0^+$ we obtain an $H_0^1$-bounded sequence $\{u_n\} = \{u_{\ep_n}\} \subset \B_{\rho}^\mu \cap S^*(\delta)$ such that
    \begin{align} \label{eqpssequencesub}
    	E(u_n,\Omega) \to \nu_{\mu,k,\delta}, \quad V(u_n) \to 0 \text{ strongly in } H_0^1(\Omega), \quad \text{ as } n \to \infty.
    \end{align}
    By Lemma \ref{lempscondition}, $\{u_n\}$ has a strongly convergent subsequence and we obtain a constrained critical point $u_{\mu,k}$ contained in $S_\mu(\Omega) \cap \overline{S^*(\delta)}$, at level $\nu_{\mu,k,\delta}$. By $u_{\mu,k} \in \overline{S^*(\delta)}$ it follows that $u$ is sign-changing. According to Lagrange multiplier principle, $u_{\mu,k}$ solves \eqref{eqonbounddomain} with $\la = \la_{\mu,k}$ for some $\la_{\mu,k} \in \R$.

    To obtain the Morse index information, without loss of generality, for the set $A \in \Sigma_{\mu}^{(k)}$ satisfying $\sup_{u \in A \cap S^*(\delta)}E(u,\Omega) < \nu_{\mu,k,\delta} + \ep$ we may assume $A = h(\s^{k-1})$ for some continuous map $h$ with $h(-s) = -h(s)$ for any $s \in \s^{k-1}$. In fact, setting $m = \ga(A)$ and by the definition of the genus, there exists $h \in C(\s^{m-1},A)$ with $h(-s) = -h(s)$ for any $s \in \s^{m-1}$. Then we can define a new set
    \begin{align*}
    	\tilde A := \bigl\{h(s_1,s_2,\cdots,s_k,0,\cdots,0): (s_1,\cdots,s_k) \in \s^{k-1}\bigr\},
    \end{align*}
    so that $\tilde A \in \Sigma_{\mu}^{(k)}$ and
    \begin{align*}
    	\sup_{u \in \tilde A \cap S^*(\delta)}E(u,\Omega) \leq \sup_{u \in A \cap S^*(\delta)}E(u,\Omega) < \nu_{\mu,k,\delta} + \ep.
    \end{align*}
    In view of Lemma \ref{leminf>0} and Proposition \ref{propdescendingflow}, using Theorem \ref{thmuepsym} with $P^* = \p_{\delta/2} \cup (-\p)_{\delta/2}$, $P^{**} = \p_\delta \cup (-\p_\delta)$, $R = \rho$ and $c = \nu_{\mu,k,\delta}$, we can further assume that $\tilde m_{\zeta_n} (u_n) \leq k$ for some $\{\zeta_n\}$ with $\zeta_n \to 0^+$, where $\{u_n\}\subset \B_{\rho}^\mu \cap S^*(\delta)$ is the Palais-Smale sequence satisfying \eqref{eqpssequencesub}. Then by Lemma \ref{lemmorse} we obtain a constrained critical point $u_{\mu,k} \in S_\mu(\Omega) \cap \overline{S^*(\delta)}$ such that $\tilde m_0(u_{\mu,k}) \leq k$ and $m(u_{\mu,k}) \leq k+1$.

    When $\mu$ is fixed, for any sequence $\{\delta_k\}_{k \ge 2}$ such that $0 < \delta_k < \tilde \delta$, by Lemma \ref{lemnonempty} and \eqref{eq5.3} we get
    \begin{align} \label{eq5.7}
    	\nu_{\mu,k,\delta_k} \geq \inf_{S_\mu(\Omega) \cap \text{span}\{\phi_1,\cdots,\phi_{k-1}\}^{\bot}}\left( \frac12 \int_\Omega |\nabla u|^2dx - \frac1p C_{p,N}^p \mu^{\frac{p(1-\ga_p)}{2}}\left( \int_\Omega |\nabla u|^2dx\right) ^{\frac{p\ga_p}{2}} \right).
    \end{align}
    Observing that $\int_\Omega |\nabla u|^2dx \ge \la_k(\Omega) \mu \to +\infty$ as $k \to +\infty$ for any $u \in S_\mu(\Omega) \cap \text{span}\{\phi_1,\cdots,\phi_{k-1}\}^{\bot}$, we conclude that $\nu_{\mu,k,\delta_k} \to +\infty$ as $k \to +\infty$.

    When $k$ is fixed, for any sequence $\{\delta_\mu\}$ such that $0 < \delta_\mu < \tilde \delta$, by \eqref{eq5.3} and the fact that $\mathbb{M}_k  \in \Sigma_{\mu}^{(k)}$, it follows that
    \begin{align} \label{eq7.8}
    	\nu_{\mu,k,\delta_\mu} \le \frac12 \sup_{u \in \mathbb{M}_k}\int_\Omega |\nabla u|^2dx \le \frac12\la_k(\Omega) \mu \to 0^+ \quad \text{as } \mu \to 0^+,
    \end{align}
    where $\mathbb{M}_k$ is defined in \eqref{eqdefMk}. Thus $\limsup_{\mu \to 0^+}\nu_{\mu,k,\delta_\mu} \le 0$. Next we prove that $\liminf_{\mu \to 0^+}\nu_{\mu,k,\delta_\mu} \ge 0$. By negation, we suppose that there exists $\{\mu_n\}$ such that $\mu_n \to 0$ and $\lim_{n \to \infty}\nu_{\mu_n,k,\delta_n} < 0$, where $\delta_n = \delta_{\mu_n}$. Let us take a sequence $\{u_n\} = \{u_{\mu_n,k}\} \subset S_{\mu_n}(\Omega)$ such that $E(u_n,\Omega) = \nu_{\mu_n,k,\delta_n}$, $u_n$ solves \eqref{eqonbounddomain} with $\la = \la_n$, $\mu = \mu_n \to 0$, and $m(u_n) \le k+1$. We claim that $\{u_n\}$ is not bounded in $L^\infty$. In fact, if $\{u_n\}$ is a $L^\infty$-bounded sequence, by $\int_\Omega |u_n|^2dx = \mu_n \to 0$ and by the boundedness of $\Omega$ we get that $\int_\Omega |u_n|^pdx \to 0$, thus
    \begin{align*}
    	\lim_{n\to\infty}E(u_n,\Omega) = \frac12\lim_{n\to\infty}\int_\Omega |\nabla u_n|^2dx \geq 0,
    \end{align*}
    contradicting $\lim_{n \to \infty}\nu_{\mu_n,k,\delta_n} < 0$. Since $\{u_n\}$ is not bounded in $L^\infty$, by Lemma \ref{lemlanbelow} and Lemma \ref{lemlanbounded} we get that $\la_n \to +\infty$. Then using \eqref{eqntoinftyun2} in Proposition \ref{propasm} we conclude that $\int_\Omega |u_n|^2dx \to +\infty$, contradicting $\mu_n \to 0$. This enables us to obtain that $\lim_{\mu \to 0^+}\nu_{\mu,k,\delta_\mu} = 0$.

    In the following we analyze the case with $\mu \to +\infty$. It is readily seen that
    \begin{align*}
    	\|u_{\mu,k}\|_{L^\infty}^2 \geq \frac{1}{|\Omega|}\int_\Omega |u_{\mu,k}|^2dx = \frac{\mu}{|\Omega|} \to +\infty, \quad \text{as } \mu \to +\infty.
    \end{align*}
    Let $\la_{\mu,k}$ be the Lagrange multiplier corresponding to $u_{\mu,k}$. Using Lemma \ref{lemlanbelow} and Lemma \ref{lemlanbounded} we know $\la_{\mu,k} \to +\infty$, and then, by Proposition \ref{proprelationship} we conclude that $E(u_{\mu,k},\Omega) \to -\infty$ as $\mu \to +\infty$. Combining above discussions we prove that when $k$ is fixed, for any sequence $\{\delta_\mu\}$ such that $0 < \delta_\mu < \tilde \delta$, $\nu_{\mu,k,\delta_\mu} \to 0$ as $\mu \to 0^+$ and $\nu_{\mu,k,\delta_\mu} \to -\infty$ as $\mu \to +\infty$. The proof is complete.
\end{proof}

\subsection{$L^2$-critical case}

Let $p = p_c$. Similar to the $L^2$-subcritical case, we can prove that $\nu_{\mu,k,\delta}$, defined in \eqref{eqdefofnu}, is a critical value if
\begin{align} \label{eqbarmu}
	0 < \mu < \bar\mu := \left( \frac{p_c}{2C_{p_c,N}^{p_c}}\right) ^{\frac{2}{p_c-2}}.
\end{align}

\begin{theorem} \label{thmgenuscri}
	Let $p = p_c$, $k \geq 2$ and let $\mu$ satisfy \eqref{eqbarmu}. Then,
	there exists $\tilde \delta = \tilde \delta(\mu,k) > 0$ such that when $0 < \delta < \tilde \delta$, there exists a sign-changing constrained critical point $u_{\mu,k}$ of $E(\cdot,\Omega)$ on $S_\mu(\Omega)$ at the level $\nu_{\mu,k,\delta}$. The function $u_{\mu,k}$ solves \eqref{eqonbounddomain} with $\la = \la_{\mu,k}$ for some $\la_{\mu,k} \in \R$ and $m(u_{\mu,k}) \le k+1$. Moreover, the following hold:
	\begin{itemize}
		\item when $\mu$ is fixed, for any sequence $\{\delta_k\}_{k \ge 2}$ such that $0 < \delta_k < \tilde \delta$, we have $\nu_{\mu,k,\delta_k} \to +\infty$ as $k \to +\infty$;
		\item when $k$ is fixed, for any sequence $\{\delta_\mu\}$ such that $0 < \delta_\mu < \tilde \delta$, we have $\nu_{\mu,k,\delta_\mu} \to 0$ as $\mu \to 0^+$.
	\end{itemize}
\end{theorem}

\begin{proof}

	When $p = p_c$, the value $p\ga_p$ is equal to $2$ and \eqref{eq5.3} becomes
	\begin{align} \label{eq5.3two}
		\frac12 \int_\Omega |\nabla u|^2dx \ge E(u,\Omega) \geq \left( \frac12 - \frac1p C_{p,N}^p \mu^{\frac{p-2}{2}}\right)  \int_\Omega |\nabla u|^2dx.
	\end{align}
    Using \eqref{eqbarmu} and along the line in proving Theorem \ref{thmgenussub}, we obtain the existence of $\tilde \delta = \tilde \delta(\mu,k) > 0$ such that when $0 < \delta < \tilde \delta$, there exists a sign-changing constrained critical point $u_{\mu,k}$ of $E(\cdot,\Omega)$ on $S_\mu(\Omega)$ at the level $\nu_{\mu,k,\delta}$. Moreover, $u_{\mu,k}$ solves \eqref{eqonbounddomain} with $\la = \la_{\mu,k}$ for some $\la_{\mu,k} \in \R$ and $m(u_{\mu,k}) \le k+1$.

    When $\mu$ is fixed, for any sequence $\{\delta_k\}_{k \ge 2}$ such that $0 < \delta_k < \tilde \delta$, by Lemma \ref{lemnonempty} and \eqref{eq5.3two} we get, as $k \to +\infty$,
    \begin{align}
    	\nu_{\mu,k,\delta_k} \geq \left( \frac12 - \frac1p C_{p,N}^p \mu^{\frac{p-2}{2}}\right)\inf_{S_\mu(\Omega) \cap \text{span}\{\phi_1,\cdots,\phi_{k-1}\}^{\bot}}\int_\Omega |\nabla u|^2dx \ge \left( \frac12 - \frac1p C_{p,N}^p \mu^{\frac{p-2}{2}}\right)\la_k(\Omega) \mu \to +\infty.
    \end{align}

    When $k$ is fixed, for any sequence $\{\delta_\mu\}$ such that $0 < \delta_\mu < \tilde \delta$, by \eqref{eq5.3two} and the fact that $\mathbb{M}_k  \in \Sigma_{\mu}^{(k)}$, it follows that
    \begin{align*}
    	0 \le \nu_{\mu,k,\delta_\mu} \le \frac12 \sup_{u \in \mathbb{M}_k}\int_\Omega |\nabla u|^2dx \le \frac12\la_k(\Omega) \mu \to 0^+ \quad \text{as } \mu \to 0^+,
    \end{align*}
    where $\mathbb{M}_k$ is defined in \eqref{eqdefMk}, which shows that $\nu_{\mu,k,\delta_\mu} \to 0$ as $\mu \to 0^+$. The proof is complete.
\end{proof}

\subsection{$L^2$-supercritical case}

In the $L^2$-supercritical case, $p\ga_p > 2$ and $E(\cdot,\Omega)$ is not bounded from below on $S_\mu(\Omega)$ and \eqref{eqcoer} fails. This brings difficulties in finding bounded Palais-Smale sequences. To overcome this issue, we introduce the following minimax value:
\begin{align}
	\nu_{\rho,\mu,k,\delta} := \inf_{A \in \Sigma_{\rho,\mu}^{(k)}}\sup_{u \in A \cap S^*(\delta)}E(u,\Omega).
\end{align}
We will prove $\nu_{\rho,\mu,k,\delta}$ is indeed a critical value under certain conditions. More precisely, we have the following theorem.

\begin{theorem} \label{thmgenus}
	Let $k \geq 2$, $\mu > 0$, $\rho \geq \la_k(\Omega)\mu$, $\tau > 0$ be fixed, and let us denote with $\K_c$ the set of critical points of $E$ constrained to $S_\mu(\Omega)$ at level $c$ contained in $\B_\rho^\mu$. If
	\begin{align} \label{eqconininte}
		\nu_{\rho,\mu,k,\delta} < \hat \nu_{\rho,\mu,k,\delta} := \inf_{A \in \Sigma_{\rho,\mu}^{(k)} \atop A \backslash \B_{\rho -\tau}^\mu \cap S^*(\delta) \neq \emptyset} \sup_{u \in A \backslash \B_{\rho -\tau}^\mu \cap S^*(\delta)}E(u,\Omega),
	\end{align}
    then, with sufficiently small $\delta > 0$ depending on $\rho$ and on $\mu$, the set $\K_{\nu_{\rho,\mu,k,\delta}} \cap \overline{S^*(\delta)} \neq \emptyset$ and it contains a critical point of Morse index less or equal to $k+1$. Moreover, in case $\rho \geq \la_{k+r}(\Omega)\mu$, \eqref{eqconininte} holds for $k, \cdots, k+r$, and $\nu = \nu_{\rho,\mu,k,\delta} = \cdots = \nu_{\rho,\mu,k+r,\delta}$ for some $r \ge 1$, then the genus of $\K_{\nu} \cap \overline{S^*(\delta)}$ is no less than $r+1$, so that $\K_{\nu} \cap \overline{S^*(\delta)}$ contains infinitely many critical points.
\end{theorem}

\begin{proof}
	Recall that
	\begin{align*}
		\frac12 \int_\Omega |\nabla u|^2dx \ge E(u,\Omega) \geq \frac12 \int_\Omega |\nabla u|^2dx - \frac1p C_{p,N}^p \mu^{\frac{p(1-\ga_p)}{2}}\left( \int_\Omega |\nabla u|^2dx\right) ^{\frac{p\ga_p}{2}}, \quad \forall u \in S_\mu(\Omega).
	\end{align*}
    Note that $p\ga_p > 2$ when $p_c < p < 2^*$. It is cleat that $|E(u,\Omega)|$ is uniformly bounded with respect to $u \in \B_\rho^\mu$. Together with Lemma \ref{lemsigmanotempty} and Lemma \ref{lemAcapSnotemp}, the value $\nu_{\rho,\mu,k,\delta}$ is well defined and finite. Firstly, we assume that there exists $A \in \Sigma_{\rho,\mu}^{(k)}$ such that $A \backslash \B_{\rho -\tau}^\mu \cap S^*(\delta) \neq \emptyset$, in this case, the value $\hat \nu_{\rho,\mu,k,\delta}$ is well defined and finite. By definition of $\nu_{\rho,\mu,k,\delta}$, for any $\ep > 0$, there exists $A \in \Sigma_{\rho,\mu}^{(k)}$ such that $\sup_{u \in A \cap S^*(\delta)}E(u,\Omega) < \nu_{\rho,\mu,k,\delta} + \ep$. With $\ep > 0$ small enough, it follows by assumption \eqref{eqconininte} that $A \cap S^*(\delta) \subset \B_{\rho-\tau}^\mu$. Then, we aim to apply Proposition \ref{propdescendingflow} with $c  = \nu_{\rho,\mu,k,\delta}$ and $S = A \cap S^*(\delta)$. Let us take $\xi > 0$, depending on $\tau$ but not on $A$ and not on $\ep$, such that $S(2\xi) \subset \B_{\rho-\tau/2}^\mu$, where the definition of $S(2\xi)$ is given in \eqref{eqS2xi}. We claim that there exists $u_\ep \in E^{-1}([c-2\epsilon,c+2\epsilon],\Omega) \cap S(2\xi) \cap S^*(\delta)$ such that $\|V(u_\ep)\| < 2\epsilon /\xi$, where $V$ is the pseudogradient defined in \eqref{def:V} with $\tau =1$. Otherwise, \eqref{eqlowboundofV} holds true and by Proposition \ref{propdescendingflow}, when $\delta > 0$ is sufficiently small, there exists $\eta: C([0,\infty) \times S_\mu(\Omega), S_\mu(\Omega))$ such that $\eta(t,-u) = \eta(t,u)$ and
    \begin{itemize}
    	\item[(i)] $\eta(t,u) \in D_\delta^*$ for all $t \geq 0$ if $u \in D_\delta^*$,
    	\item[(ii)] $\forall u \in S$ with $E(u,\Omega) \leq \nu_{\rho,\mu,k,\delta} + \epsilon$, either $\eta(\xi,u) \in D_\delta^*$ or $E(\eta(\xi,u),\Omega) \leq \nu_{\rho,\mu,k,\delta} -\epsilon$.
    \end{itemize}
    It is clear that $\eta(\xi,A) \in \Sigma_{\rho,\mu}^{(k)}$ and so $\sup_{u \in \eta(\xi,A) \cap S^*(\delta)}E(u,\Omega) \geq \nu_{\rho,\mu,k,\delta}$. On the other hand, by properties (i) and (ii) we deduce that $\sup_{u \in \eta(\xi,A) \cap S^*(\delta)}E(u,\Omega) \leq \nu_{\rho,\mu,k,\delta} - \ep < \nu_{\rho,\mu,k,\delta}$. This is a contradiction and validates the claim. At this point, with $\ep_n \to 0^+$ we obtain an $H_0^1$-bounded sequence $\{u_n\} = \{u_{\ep_n}\} \subset \B_{\rho-\tau/2}^\mu \cap S^*(\delta)$ such that
    \begin{align} \label{eqpssequence}
    	E(u_n,\Omega) \to \nu_{\rho,\mu,k,\delta}, \quad V(u_n) \to 0 \text{ strongly in } H_0^1(\Omega), \quad \text{ as } n \to \infty.
    \end{align}
    By Lemma \ref{lempscondition}, $\{u_n\}$ has a strongly convergent subsequence and we obtain a constrained critical point $u$ contained in $\B_\rho^\mu \cap \overline{S^*(\delta)}$, that is, $\K_{\nu_{\rho,\mu,k,\delta}} \cap \overline{S^*(\delta)} \neq \emptyset$.

    To obtain the Morse index information, similar to the proof of Theorem \ref{thmgenussub}, in view of Lemma \ref{leminf>0} and Proposition \ref{propdescendingflow}, using Theorem \ref{thmuepsym}, we can further assume that $\tilde m_{\zeta_n} (u_n) \leq k$ for some $\{\zeta_n\}$ with $\zeta_n \to 0^+$, where $\{u_n\}\subset \B_{\rho-\tau/2}^\mu \cap S^*(\delta)$ is the Palais-Smale sequence satisfying \eqref{eqpssequence}. Then by Lemma \ref{lemmorse} we obtain a constrained critical point $u \in \B_\rho^\mu \cap \overline{S^*(\delta)}$ such that $\tilde m_0(u) \leq k$ and $m(u) \leq k+1$.

    Secondly, we consider the case such that $A \backslash \B_{\rho -\tau}^\mu \cap S^*(\delta) = \emptyset$ for any $A \in \Sigma_{\rho,\mu}^{(k)}$. The value $\nu_{\rho,\mu,k,\delta}$ is no longer well defined, however, we have $A \cap S^*(\delta) \subset \B_{\rho -\tau}^\mu$ for $A \in \Sigma_{\rho,\mu}^{(k)}$. Then following previously arguments step by step we conclude that for sufficiently small $\delta > 0$, the set $\K_{\nu_{\rho,\mu,k,\delta}} \cap \overline{S^*(\delta)} \neq \emptyset$ and it contains a critical point of Morse index less or equal to $k+1$.

    Finally, we suppose that $\rho \geq \la_{k+r}(\Omega)\mu$, \eqref{eqconininte} holds for $k, \cdots, k+r$, and $\nu = \nu_{\rho,\mu,k,\delta} = \cdots = \nu_{\rho,\mu,k+r,\delta}$ for some $r \ge 1$. The fact $\ga(\K_{\nu} \cap S^*(\delta)) \geq r+1$ can be proved by combining the flow invariance arguments in Section \ref{secinv} and extension of standard arguments in such a question without sign-change requirement. By the definition of $\nu_{\rho,\mu,k+r,\delta}$, for $\ep > 0$ there exists $A \in \Sigma_{\rho,\mu}^{(k+r)}$ such that $\sup_{u \in A \cap S^*(\delta)}E(u,\Omega) < \nu + \ep$. With $\ep > 0$ small enough, it follows by assumption \eqref{eqconininte} for $k+r$ that $A \cap S^*(\delta) \subset \B_{\rho-\tau}^\mu$. Let $N$ be a small invariant open neighborhood of $\mathcal{K}_\nu \cap \overline{S^*(\delta)}$ in $S_\mu(\Omega)$ such that $\ga(\mathcal{K}_\nu \cap \overline{S^*(\delta)}) = \ga(\overline{N})$. Then, we aim to apply Proposition \ref{propdescendingflow} with $c  = \nu$ and $S = A \backslash (A \cap N)$. Let us take $\xi > 0$, depending on $\tau$ but not on $A$ and not on $\ep$, such that $S(2\xi) \subset \B_{\rho-\tau/2}^\mu$. Then, by Lemma \ref{lempscondition}, the assumption \eqref{eqlowboundofV} holds true if we assume that $\ep > 0$ is sufficiently small. Proposition \ref{propdescendingflow} yields that there exists $\eta: C([0,\infty) \times S_\mu(\Omega), S_\mu(\Omega))$ such that $\eta(t,-u) = \eta(t,u)$ and
    \begin{itemize}
    	\item[(i)] $\eta(t,u) \in D_\delta^*$ for all $t \geq 0$ if $u \in D_\delta^*$,
    	\item[(ii)] $\forall u \in A \backslash (A \cap N)$ with $E(u,\Omega) \leq \nu + \epsilon$, either $\eta(\xi,u) \in D_\delta^*$ or $E(\eta(\xi,u),\Omega) \leq \nu -\epsilon$.
    \end{itemize}
    We claim that $\ga(A \backslash (A \cap N)) \le k -1$. Otherwise, we have $\eta(\xi,A \backslash (A \cap N)) \in \Sigma_{\rho,\mu}^{(k)}$. However, it holds that
    \begin{align*}
    	\sup_{\eta(\xi,A \backslash (A \cap N)) \cap S^*(\delta)}E(\cdot,\Omega) \leq \nu - \ep < \nu,
    \end{align*}
    contradicting the definition of $\nu_{\rho,\mu,k,\delta}$. Then well known properties of genus yield that
    \begin{align*}
    	k + r & \le \ga(A) \le \ga(A \cap \overline{N}) + \ga(A \backslash (A \cap N)) \\
    	& \le \ga(\overline{N}) + k -1 \\
    	& = \ga(\mathcal{K}_\nu \cap \overline{S^*(\delta)}) + k -1,
    \end{align*}
    showing that $\ga(\mathcal{K}_\nu \cap \overline{S^*(\delta)}) \ge r +1$. The proof is complete.
\end{proof}

We now provide a sufficient condition to guarantee the validity of assumption \eqref{eqconininte}.

\begin{lemma} \label{lemsufficientcon}
	Let $k \geq 2$, $\mu > 0$, $\rho > 0$ satisfy
	\begin{align} \label{eqsufficientcon}
		\la_k(\Omega)\mu + \frac2p C_{p,N}^p \mu^{\frac{p(1-\ga_p)}{2}}\rho ^{\frac{p\ga_p}{2}} < \rho + \frac2p |\Omega|^{\frac{2-p}{2}}\mu^{\frac p2},
	\end{align}
	then \eqref{eqconininte} holds true for sufficiently small $\tau > 0$.
\end{lemma}

\begin{proof}
	On the one hand, observing $\mathbb{M}_k \subset \Sigma_{\rho,\mu}^{(k)}$ and $\int_\Omega|u|^pdx \geq |\Omega|^{\frac{2-p}{2}}\mu^{\frac p2}$, we have
	\begin{align} \label{eqnukleq}
		\nu_{\rho,\mu,k,\delta} \leq \frac12 \la_k(\Omega)\mu - \frac1p |\Omega|^{\frac{2-p}{2}}\mu^{\frac p2}.
	\end{align}
    On the other hand, by \eqref{eq5.3} it holds that
    \begin{align*}
    	\hat \nu_{\rho,\mu,k,\delta} \geq \inf_{\overline{\B_\rho^\mu} \backslash \B_{\rho-\tau}^\mu} E(\cdot,\Omega) \ge \frac12 \rho - \frac1p C_{p,N}^p \mu^{\frac{p(1-\ga_p)}{2}}\rho ^{\frac{p\ga_p}{2}} + o_\tau(1),
    \end{align*}
    where $o_\tau(1) \to 0$ as $\tau \to 0^+$. Thus, using \eqref{eqsufficientcon}, with $\tau > 0$ sufficiently small we have that \eqref{eqconininte} holds true.
\end{proof}

\subsection{Proof of Theorem \ref{thmsmall}}

At the end of this section we complete the proof of Theorem \ref{thmsmall}.

\begin{proof}[Proof of Theorem \ref{thmsmall}]
	We consider subcritical, critical, supercritical cases respectively.
	
	\vskip0.1in
	\emph{(1) Subcritical case with $2 < p < p_c$: }
	
	By Theorem \ref{thmgenussub}, for $k \geq 2$ and $\mu > 0$, there exists $\tilde \delta = \tilde \delta(\mu,k) > 0$ such that when $0 < \delta < \tilde \delta$, there exists a sign-changing constrained critical point $u_{\mu,k}$ of $E(\cdot,\Omega)$ on $S_\mu(\Omega)$ at the level $\nu_{\mu,k,\delta}$. Choosing $\delta_k \in (0,\tilde \delta)$, we have $\nu_{\mu,k,\delta_k} \to +\infty$ as $k \to +\infty$. Thus we can find a sequence $\{k_j\}$ with $k_1 = 2$ such that $\nu_{\mu,k_1,\delta_{k_1}} < \cdots < \nu_{\mu,k_j,\delta_{k_j}} < \cdots$. We redeonte $u_{\mu,j} = u_{\mu,k_j}$ and proved that \eqref{eqonbounddomain} has infinitely many sign-changing solutions $u_{\mu,1}, \cdots, u_{\mu,j}, \cdots$ . For fixed $\mu > 0$, we have
	\begin{align*}
		E(u_{\mu,j},\Omega) = \nu_{\mu,k_j,\delta_{k_j}} \to +\infty \quad \text{as } j \to +\infty.
	\end{align*}
	Moreover, for fixed $j$, the choice of $\delta_{k_j}$ depends on $\mu$ and we denote it by $\delta_\mu$. Then, by Theorem \ref{thmgenussub},
	\begin{align*}
		E(u_{\mu,j},\Omega) = \nu_{\mu,k_j,\delta_{\mu}} \to 0 \quad \text{as } \mu \to 0^+,
	\end{align*}
    and
    \begin{align*}
    	E(u_{\mu,j},\Omega) = \nu_{\mu,k_j,\delta_{\mu}} \to -\infty \quad \text{as } \mu \to +\infty.
    \end{align*}
    Let $\la_{\mu,j}$ be the Lagrange multiplier corresponding to $u_{\mu,j}$. As a direct consequence of Proposition \ref{proprelationship} we have $\la_{\mu,j} \to +\infty$ as $\mu \to +\infty$.

    In the following we study the limit behavior of $\la_{\mu,j}$ when $\mu \to 0^+$.

For fixed $j$, on the one hand, by \eqref{eq7.8} we have
\begin{align*}
	E(u_{\mu,j},\Omega) = \nu_{\mu,k_j,\delta_\mu} \le \frac12\la_{k_j}(\Omega) \mu.
\end{align*}
On the other hand, using \eqref{eq5.7} we get that
\begin{align*}
	E(u_{\mu,j},\Omega) = \nu_{\mu,k_j,\delta_\mu} \geq \inf_{S_\mu(\Omega) \cap \text{span}\{\phi_1,\cdots,\phi_{k_j-1}\}^{\bot}}\left( \frac12 \int_\Omega |\nabla u|^2dx - \frac1p C_{p,N}^p \mu^{\frac{p(1-\ga_p)}{2}}\left( \int_\Omega |\nabla u|^2dx\right) ^{\frac{p\ga_p}{2}} \right).
\end{align*}
Let
\begin{align*}
	f(t) = \frac12 t - \frac1p C_{p,N}^p \mu^{\frac{p(1-\ga_p)}{2}}t ^{\frac{p\ga_p}{2}}, \quad t > 0.
\end{align*}
Computing directly we get that
\begin{align*}
	f'(t) = \frac12 - \frac{\ga_p}2 C_{p,N}^p \mu^{\frac{p(1-\ga_p)}{2}}t ^{\frac{p\ga_p}{2}-1},
\end{align*}
and since $p\ga_p < 2$,
\begin{align*}
	\min_{t > 0}f(t) = f(t_\mu) = \left( \frac12 - \frac{1}{p\ga_p}\right) t_\mu \to 0 \quad \text{as } \mu \to 0^+,
\end{align*}
where
\begin{align*}
	t_\mu = \left( \frac{1}{\ga_pC_{p,N}^p}\right) ^{\frac{2}{p\ga_p-2}}\mu^{\frac{p(1-\ga_p)}{2-p\ga_p}} \to 0 \quad \text{as } \mu \to 0^+.
\end{align*}
Note that $\int_\Omega |\nabla u|^2dx \geq \la_{k_j}(\Omega)\mu$ for $u \in S_\mu(\Omega) \cap \text{span}\{\phi_1,\cdots,\phi_{k_j-1}\}^{\bot}$. By the fact
\begin{align*}
	\frac{p(1-\ga_p)}{2-p\ga_p} > 1,
\end{align*}
as $\mu \to 0^+$, it is not restrictive to assume that $\la_{k_j}(\Omega)\mu > t_\mu$. Thus
\begin{align*}
	E(u_{\mu,j},\Omega) \geq \inf_{u \in S_\mu(\Omega) \cap \text{span}\{\phi_1,\cdots,\phi_{k_j-1}\}^{\bot}}f(\int_\Omega |\nabla u|^2dx) \ge f(\la_{k_j}(\Omega)\mu).
\end{align*}
Recalling that $E(u_{\mu,j},\Omega) \le \frac12\la_{k_j}(\Omega) \mu$, we get
\begin{align*}
	\frac12\la_{k_j}(\Omega) - \frac1p C_{p,N}^p \la_{k_j}(\Omega)^{\frac{p\ga_p}{2}} \mu^{\frac{p}{2}-1} \le \mu^{-1}E(u_{\mu,j},\Omega) \le \frac12\la_{k_j}(\Omega),
\end{align*}
implying that
\begin{align*}
	\mu^{-1}E(u_{\mu,j},\Omega) \to \frac12\la_{k_j}(\Omega) \quad \text{as } \mu \to 0^+.
\end{align*}
By Lemma \ref{lemlanbounded} and Proposition \ref{proprelationship}, we have that $u_{\mu,j}$ is bounded in $L^\infty$ uniformly with respect to $\mu \to 0^+$, and so
\begin{align*}
	\int_\Omega |u_{\mu,j}|^pdx \le \|u_{\mu,j}\|_{L^\infty}^{p-2}\mu \to 0 \quad \text{as } \mu \to 0^+,
\end{align*}
thus
\begin{align*}
	\int_\Omega |\nabla u_{\mu,j}|^2dx = 2E(u_{\mu,j},\Omega) + \frac2p\int_\Omega |u_{\mu,j}|^pdx \to 0 \quad \text{as } \mu \to 0^+.
\end{align*}
By Sobolev inequality we see that
\begin{align*}
	\frac{\int_\Omega |u_{\mu,j}|^pdx}{\int_\Omega |\nabla u_{\mu,j}|^2dx} \to 0 \quad \text{as } \mu \to 0^+.
\end{align*}
Then we can conclude
\begin{align*}
	\mu^{-1}\int_\Omega|\nabla u_{\mu,j}|^2dx \to \la_{k_j}(\Omega) \quad \text{as } \mu \to 0^+.
\end{align*}
From
\begin{align*}
	\int_\Omega |\nabla u_{\mu,j}|^2 dx + \la_{\mu,j} \int_\Omega |u_{\mu,j}|^2dx = \int_\Omega |u_{\mu,j}|^p dx,
\end{align*}
it follows that
\begin{align*}
	\la_{\mu,j} = \mu^{-1}\int_\Omega |u_{\mu,j}|^p dx - \mu^{-1}\int_\Omega |\nabla u_{\mu,j}|^2 dx \to - \la_{k_j}(\Omega) \quad \text{as } \mu \to 0^+.
\end{align*}
Particularly, since $k_1 = 2$ we have $\la_{\mu,2} \to -\la_2(\Omega)$ as $\mu \to 0^+$. We complete the proof of this case.

\vskip0.1in
\emph{(2) Critical case with $p = p_c$: }

By Theorem \ref{thmgenuscri}, for $k \geq 2$ and $\mu \in (0,\bar\mu)$ where $\bar\mu$ is given in \eqref{eqbarmu}, there exists $\tilde \delta = \tilde \delta(\mu,k) > 0$ such that when $0 < \delta < \tilde \delta$, there exists a sign-changing constrained critical point $u_{\mu,k}$ of $E(\cdot,\Omega)$ on $S_\mu(\Omega)$ at the level $\nu_{\mu,k,\delta}$. Choosing $\delta_k \in (0,\tilde \delta)$, we have $\nu_{\mu,k,\delta_k} \to +\infty$ as $k \to +\infty$. Thus we can find a sequence $\{k_j\}$ with $k_1 = 2$ such that $\nu_{\mu,k_1,\delta_{k_1}} < \cdots < \nu_{\mu,k_j,\delta_{k_j}} < \cdots$. We redeonte $u_{\mu,j} = u_{\mu,k_j}$ and proved that \eqref{eqonbounddomain} has infinitely many sign-changing solutions $u_{\mu,1}, \cdots, u_{\mu,j}, \cdots$ . For fixed $\mu > 0$, we have
\begin{align*}
	E(u_{\mu,j},\Omega) = \nu_{\mu,k_j,\delta_{k_j}} \to +\infty \quad \text{as } j \to +\infty.
\end{align*}
Moreover, for fixed $j$, the choice of $\delta_{k_j}$ depends on $\mu$ and we denote it by $\delta_\mu$. Then, by Theorem \ref{thmgenuscri},
\begin{align*}
	E(u_{\mu,j},\Omega) = \nu_{\mu,k_j,\delta_{\mu}} \to 0 \quad \text{as } \mu \to 0^+.
\end{align*}

For fixed $j$, similar to the subcritical case, we have
\begin{align*}
	\left( \frac12 - \frac1p C_{p,N}^p\mu^{\frac{p}{2}-1}\right)  \la_{k_j}(\Omega) \le \mu^{-1}E(u_{\mu,j},\Omega)  \le \frac12\la_{k_j}(\Omega).
\end{align*}
Then, by \eqref{eq5.3two} it follows
\begin{align*}
	\mu^{-1}\int_\Omega |\nabla u_{\mu,j}|^2dx \to \la_{k_j}(\Omega) \quad \text{as } \mu \to 0^+.
\end{align*}
Using \eqref{eqgninequality} we obtain
\begin{align*}
	\mu^{-1}\int_\Omega|u_{\mu,j}|^pdx \le C_{p,N}^p \mu^{\frac{p}{2}-1} \mu^{-1}\int_\Omega |\nabla u_{\mu,j}|^2dx \to 0 \quad \text{as } \mu \to 0^+.
\end{align*}
Then, along the line in the subscritical case we can complete the proof.

\vskip0.1in
\emph{(3) Supercritical case with $p_c < p < 2^*$:}

Given a positive integer $j$, there exists $\mu^*_{p,j}$ such that for all $0 < \mu < \mu^*_{p,j}$, the set
\begin{align} \label{eqsetrho}
	\biggl\{\rho: \rho > \la_{j+1}(\Omega)\mu, \ \la_{j+1}(\Omega)\mu + \frac2p C_{p,N}^p \mu^{\frac{p(1-\ga_p)}{2}}\rho ^{\frac{p\ga_p}{2}} < \rho + \frac2p |\Omega|^{\frac{2-p}{2}}\mu^{\frac p2}\biggr\}
\end{align}
is not empty. In fact, there exists $\mu^*_{p,j}$ such that for any $0 < \mu < \mu^*_{p,j}$, we can take $s = s(\mu) > 0$ such that
\begin{align*}
	\frac2p C_{p,N}^p \mu^{\frac{p}{2}}(\la_{j+1}(\Omega) + s) ^{\frac{p\ga_p}{2}} < s\mu + \frac2p |\Omega|^{\frac{2-p}{2}}\mu^{\frac p2}
\end{align*}
and thus $\rho = (\la_{j+1}(\Omega) + s)\mu$ belongs to the set defined in \eqref{eqsetrho}, showing that this set is not empty. For $\mu \in (0,\mu^*_{p,j})$ we take $\rho = \rho_\mu$ in the set defined in \eqref{eqsetrho}. By Lemma \ref{lemsufficientcon}, there exists $\tau > 0$ such that \eqref{eqconininte} holds true for $k = 2,\cdots,j+1$. By Theorem \ref{thmgenus}, for any $k \in \{2,\cdots,j+1\}$, there exists a sign-changing constrained critical point $\tilde u_{\mu,k}$ of $E(\cdot,\Omega)$ on $S_\mu(\Omega)$ at the level $\nu_{\rho,\mu,k,\delta}$, whose Morse index is less or equal to $k+1$.

We firstly assume that $\nu_{\rho,\mu,2,\delta} < \nu_{\rho,\mu,3,\delta} < \cdots < \nu_{\rho,\mu,j+1,\delta}$. Denote $u_{\mu,1} = \tilde u_{\mu,2}, \cdots, u_{\mu,j} = \tilde u_{\mu,j+1}$, and these are $j$ different sign-changing solutions of \eqref{eqonbounddomain} with $\la = \la_{\mu,1}$, $\la_{\mu,2}$, $\cdots$, $\la_{\mu,j}$. For fixed $i \in \{1,\cdots,j\}$,  denoting $\rho = \rho_\mu$ and $\delta = \delta_\mu$, by the fact that $\mathbb{M}_{i+1}  \in \Sigma_{\rho_\mu,\mu}^{(i+1)}$, it follows that
\begin{align} \label{eq7.18}
	E(u_{\mu,i},\Omega) = \nu_{\rho_\mu,\mu,i+1,\delta_\mu} \le \frac12 \sup_{u \in \mathbb{M}_{i+1}}\int_\Omega |\nabla u|^2dx \le \frac12\la_{i+1}(\Omega) \mu \to 0^+ \quad \text{as } \mu \to 0^+,
\end{align}
where $\mathbb{M}_{i+1}$ is defined in \eqref{eqdefMk}. Thus we get $\limsup_{\mu \to 0^+}\nu_{\rho_\mu,\mu,i+1,\delta_\mu} \le 0$. In the following we prove that $\liminf_{\mu \to 0^+}\nu_{\rho_\mu,\mu,i+1,\delta_\mu} \ge 0$. By negation, we suppose that there exists $\{\mu_n\}$ such that $\mu_n \to 0$ and $\lim_{n \to \infty}\nu_{\rho_n,\mu_n,i+1,\delta_n} < 0$, where $\rho_n = \rho_{\mu_n}$ and $\delta_n = \delta_{\mu_n}$. Let us take a sequence $\{u_n\} = \{u_{\mu_n,i}\} \subset S_{\mu_n}(\Omega)$ such that $E(u_n,\Omega) = \nu_{\rho_n,\mu_n,i+1,\delta_n}$, $u_n$ solves \eqref{eqonbounddomain} with $\la = \la_n$, $\mu = \mu_n \to 0$, and $m(u_n) \le i+2$. Similar to the proof of Theorem \ref{thmgenussub}, by Lemma \ref{lemlanbelow} and Lemma \ref{lemlanbounded} we get that $\la_n \to +\infty$. Following the arguments in proving Proposition \ref{proprelationship}, by using Proposition \ref{propasm} we can prove that $E(u_n,\Omega) \to +\infty$, contradicting \eqref{eq7.18}. This shows that $E(u_{\mu,i},\Omega) \to 0$ as $\mu \to 0^+$. Moreover, by Lemma \ref{lemnonempty} and \eqref{eq5.3} it holds that
\begin{align} \label{eqnuineq}
	\nu_{\rho_\mu,\mu,i+1,\delta_\mu} \geq \inf_{\B_{\rho_\mu}^\mu \cap \text{span}\{\phi_1,\cdots,\phi_{i}\}^{\bot}}\left( \frac12 \int_\Omega |\nabla u|^2dx - \frac1p C_{p,N}^p \mu^{\frac{p(1-\ga_p)}{2}}\left( \int_\Omega |\nabla u|^2dx\right) ^{\frac{p\ga_p}{2}} \right).
\end{align}
Recall that $f$ is introduced in the subcritical case by
\begin{align*}
	f(t) = \frac12 t - \frac1p C_{p,N}^p \mu^{\frac{p(1-\ga_p)}{2}}t ^{\frac{p\ga_p}{2}}, \quad t > 0,
\end{align*}
and
\begin{align*}
	f'(t) = \frac12 - \frac{\ga_p}2 C_{p,N}^p \mu^{\frac{p(1-\ga_p)}{2}}t ^{\frac{p\ga_p}{2}-1}.
\end{align*}
In this case, $p\ga_p > 2$, thus $f(t)$ is strictly decreasing in $t \ge t_\mu$,
where
\begin{align*}
	t_\mu = \left( \frac{1}{\ga_pC_{p,N}^p}\right) ^{\frac{2}{p\ga_p-2}}\mu^{\frac{p(1-\ga_p)}{2-p\ga_p}} \to 0 \quad \text{as } \mu \to 0^+.
\end{align*}
As discussed at the start of this case, $\rho_\mu$ can be chosen as $\rho_\mu = (\la_{i+1}(\Omega)+s(\mu))\mu$ where
\begin{align*}
	s(\mu) \to 0 \quad \text{and} \quad s(\mu)\mu^{\frac{2-p}{2}} \to +\infty \quad \text{as } \mu \to 0^+.
\end{align*}
Note that $\la_{i+1}(\Omega)\mu \le \int_\Omega |\nabla u|^2dx \le (\la_{i+1}(\Omega)+s(\mu))\mu$ for $u \in \B_{\rho_\mu}^\mu \cap \text{span}\{\phi_1,\cdots,\phi_{i}\}^{\bot}$. By the fact
\begin{align*}
	\frac{p(1-\ga_p)}{2-p\ga_p} > 1,
\end{align*}
as $\mu \to 0^+$, it is not restrictive to assume that $\la_{i+1}(\Omega)\mu > t_\mu$. Thus \eqref{eqnuineq} yields that
\begin{align*}
	\nu_{\rho_\mu,\mu,i+1,\delta_\mu} \geq f((\la_{i+1}(\Omega)+s(\mu))\mu) = \frac12\la_{i+1}(\Omega)\mu + o(\mu),
\end{align*}
where $\mu^{-1}o(\mu) \to 0$ as $\mu \to 0^+$. Together with \eqref{eq7.18}, we conclude that
\begin{align*}
	\mu^{-1}E(u_{\mu,i},\Omega) \to \frac12\la_{i+1}(\Omega) \quad \text{as } \mu \to 0^+.
\end{align*}
Then, following the discussions in the proof of subcritical case, we can obtain that $\la_{\mu,i} \to -\la_{i+1}(\Omega)$, particularly, $\la_{\mu,1} \to -\la_2(\Omega)$ as $\mu \to 0^+$.

Next, we assume that $\nu_{\rho,\mu,k,\delta} = \nu_{\rho,\mu,k+1,\delta}$ for some $k \in \{2,\cdots,j\}$. By Theorem \ref{thmgenus}, \eqref{eqonbounddomain} has $j$ different sign-changing solutions $u_{\mu,1}$, $u_{\mu,2}$, $\cdots$, $u_{\mu,j}$, with $\la = \la_{\mu,1}$, $\la_{\mu,2}$, $\cdots$, $\la_{\mu,j}$, where $u_{\mu,1}$ is at level $\nu_{\rho,\mu,2,\delta}$, $u_{\mu,2}, \cdots, u_{\mu,j}$ are at level $\nu = \nu_{\rho,\mu,k,\delta} = \nu_{\rho,\mu,k+1,\delta}$, and $m(u_{\mu,1}) \le 3$, $m(u_{\mu,2}) \le k+1$. As proved above, we have
\begin{align*}
	\mu^{-1}\nu_{\rho_\mu,\mu,2,\delta_\mu} \to \frac12\la_2(\Omega) \quad \text{and} \quad \mu^{-1}\nu \to \frac12\la_{k}(\Omega) \quad \text{as } \mu \to 0^+,
\end{align*}
which implies that
\begin{align*}
	E(u_{\mu,i},\Omega) \to 0 \quad \text{as } \mu \to 0^+
\end{align*}
for any $i \in \{1,2,\cdots,j\}$. Moreover, we have that $\la_{\mu,1} \to -\la_2(\Omega)$ and $\la_{\mu,2} \to -\la_k(\Omega)$ as $\mu \to 0^+$. In the following we consider $i \in \{3,\cdots,j\}$. By the boundedness of $u_{\mu,i}$ in $H_0^1(\Omega)$ uniformly with respect to $\mu \to 0^+$ and by Gagliardo-Nirenberg inequality \eqref{eqgninequality} we get
\begin{align*}
	\int_\Omega|u_{\mu,i}|^pdx \to 0 \quad \text{as } \mu \to 0^+,
\end{align*}
and so
\begin{align*}
	\int_\Omega |\nabla u_{\mu,i}|^2dx = 2E(u_{\mu,i},\Omega) + \frac2p\int_\Omega |u_{\mu,i}|^pdx \to 0 \quad \text{as } \mu \to 0^+.
\end{align*}
By Sobolev inequality we see that
\begin{align*}
	\frac{\int_\Omega |u_{\mu,i}|^pdx}{\int_\Omega |\nabla u_{\mu,i}|^2dx} \to 0 \quad \text{as } \mu \to 0^+.
\end{align*}
Thus we can conclude
\begin{align*}
	\mu^{-1}\int_\Omega|\nabla u_{\mu,i}|^2dx \to \la_k(\Omega) \quad \text{as } \mu \to 0^+.
\end{align*}
From
\begin{align*}
	\int_\Omega |\nabla u_{\mu,i}|^2 dx + \la_{\mu,i} \int_\Omega |u_{\mu,i}|^2dx = \int_\Omega |u_{\mu,i}|^p dx,
\end{align*}
it follows that
\begin{align*}
	\la_{\mu,i} = \mu^{-1}\int_\Omega |u_{\mu,i}|^p dx - \mu^{-1}\int_\Omega |\nabla u_{\mu,i}|^2 dx \to -\la_k(\Omega) \quad \text{as } \mu \to 0^+.
\end{align*}
The proof is complete.
\end{proof}

\section{Review and development of a theorem in \cite{BCJN}} \label{secreview}

In this section, we use the general setting introduced in Section \ref{sectool}. Firstly, we review the results in \cite[Theorem 1.12]{BCJN} and the fact that $\tilde m_{\zeta_n}(u_n) \leq 1$ if $k =1$ in (iv) comes from \cite[Theorem 1.5]{BCJN}.

\begin{theorem}\label{Objective1}
    Let $I\subset (0,\infty)$ be an interval and consider a family of $C^2$ functionals $\Phi_\tau\colon E\to \R$ of the form
	\begin{equation*}
		\Phi_\tau(u)=A(u)-\tau B(u),\quad \tau\in I,
	\end{equation*}
	where $B(u)\geq 0$ for all $u\in E$ and
	\begin{equation}\label{hp coer}
		\text{ either } A(u) \to +\infty
		\text{ or } B(u) \to +\infty
		\quad \text{as } u \in E \text{ and }
		\|u\| \to +\infty.
	\end{equation}
	Suppose that, for every $\tau \in I$, $\Phi_\tau|_{S_\mu}$ is even, and moreover that $\Phi'_\tau$ and $\Phi''_\tau$ are
	$\alpha$-H\"older continuous on bounded sets in the sense of
	Definition~\ref{Holder-continuous} for some
	$\alpha \in (0,1]$. Finally, suppose that there exists an integer $k \geq 1$ and two odd functions
	$\gamma_{i,k} : \mathbb{S}^{k-1} \to S_{\mu}$ where $i =0,1$, such
	that the set
	\begin{align} \label{SMP1}
		\Gamma_k := \big\{\gamma \in C([0,1]\times \s^{k-1}, S_\mu):  \forall t \in [0,1], \gamma(t,\cdot) \text{ is odd}, \gamma(0,\cdot) = \gamma_{0,k}, \text{ and } \gamma(1,\cdot) = \gamma_{1,k}\big\}
	\end{align}
	is not empty and is independent of $\tau$, and
	\begin{align} \label{SMP2}
		c_{\tau}^k := \inf_{\gamma \in \Gamma_k}\sup_{(t,s) \in [0,1]\times\s^{k-1}}\Phi_\tau(\gamma(t,s)) > \max_{s\in \s^{k-1}}\max\bigl\{\Phi_\tau(\gamma_{0,k}(s)), \Phi_\tau(\gamma_{1,k}(s))\bigr\}.
	\end{align}
	Then, for almost every $\tau \in I$, there exist sequences
	$\{u_n\} \subset S_\mu$ and $\zeta_n \to 0^+$ such that, as
	$n \to + \infty$,
	\begin{itemize}
		\item[(i)] $\Phi_\tau(u_n) \to c_{\tau}^k$;
		\item[(ii)] $\bigl\| \nabla \Phi_\tau|_{S_\mu}(u_n)\| \to 0$;
		\item[(iii)] $\{u_n\}$ is bounded in $E$;
		\item[(iv)] $\tilde m_{\zeta_n}(u_n) \leq k+1$; moreover, $\tilde m_{\zeta_n}(u_n) \leq 1$ if $k =1$.
	\end{itemize}
\end{theorem}

\begin{remark} \label{rmkalmost}
	For $\tau \in I$ such that $(c_\tau^k)'$, the derivative of $c_\tau^k$ with respect to $\tau$, exists, the results in Theorem \ref{Objective1} hold. Since $c_\tau^k$ is non-increasing, such $\tau$ is almost everywhere. See details in \cite{BCJN}.
\end{remark}

In the following, we study a special situation when $c_\tau^1 = c_\tau^2$. In this case, with an additional condition \eqref{eq6.5}, we can obtain the location information of the bounded Palais-Smale sequence that is established in Theorem \ref{Objective1}. This observation is useful to search for sign-changing critical points.

\begin{proposition} \label{propc1=c2}
	Assume that the assumptions in Theorem \ref{Objective1} hold with both $k =1$ and $k =2$. Let $F^* \subset S_\mu$ be a closed set satisfying
	\begin{align} \label{eq6.5}
		F^* \cap (-F^*) \cap S_\mu = \emptyset, \quad \text{where } -F^* = \bigl\{-u: u \in F^*\bigr\}.
	\end{align}
	For almost every $\tau \in I$, if $c = c_\tau^1 = c_\tau^2$, there exist sequences
	$\{u_n\} \subset S_\mu \backslash (-F^* \cup F^*)$ and $\zeta_n \to 0^+$ such that, as
	$n \to + \infty$,
	\begin{itemize}
		\item[(i)] $\Phi_\tau(u_n) \to c$;
		\item[(ii)] $\bigl\| \nabla \Phi_\tau|_{S_\mu}(u_n)\| \to 0$;
		\item[(iii)] $\{u_n\}$ is bounded in $E$;
		\item[(iv)] $\tilde m_{\zeta_n}(u_n) \leq 1$.
	\end{itemize}
\end{proposition}

We present two topological lemmas preparing the proof of Proposition \ref{propc1=c2}.

\begin{lemma} \label{lemtopo1}
	Let $A \subset (0,1) \times \s^1$ be a closed set and let the map $P$ be defined as
	\begin{align*}
		P: [0,1] \times \s^1 \to \s^1, \quad (t,s) \mapsto s.
	\end{align*}
	Then $[0,1] \times \s^1 \backslash A$ is not path-connected is equivalent to that $A$ contains a continuous closed curve $\mathcal{C}$ such that the degree $\deg(P|_{\mathcal{C}},\mathcal{C},\s^1) \neq 0$.
\end{lemma}

\begin{proof}
	\emph{Sufficiency:} Let $\mathcal{C}$ be a continuous closed curve contained in $(0,1) \times \s^1$ such that $\deg(P|_{\mathcal{C}},\mathcal{C},\s^1) \neq 0$. Without loss of generality, we assume $\deg(P|_{\mathcal{C}},\mathcal{C},\s^1) =1$. After a continuous deformation, we assume that $\mathcal{C} = \{\frac12\} \times \s^1$ and $[0,1] \times \s^1 \backslash \mathcal{C}$ is not path-connected. Since $\mathcal{C} \subset A$, we get that $[0,1] \times \s^1 \backslash A$ is not path-connected.
	
	\vskip0.1in
	\emph{Necessity:}
	Let $A$ be an union of single points, then $[0,1] \times \s^1 \backslash A$ is path-connected. Moreover, if any path-connected branch of $A$ can continuously contract into a single point, then $[0,1] \times \s^1 \backslash A$ is path-connected. Let $A_1$, not being a single point, be a path-connected branch of $A$. Next we prove that $A_1$ can continuously contract into a single point provided that any continuous closed curve $\mathcal{C} \subset A_1$ satisfies $\deg(P|_{\mathcal{C}},\mathcal{C},\s^1) = 0$. Let us take $p \in A_1$. For any other point $\bar p \in A_1$, since $A_1$ is path-connected, we can find a continuous closed curve $\widetilde{\mathcal{C}} \subset A_1$ such that $p, \bar p \in \widetilde{\mathcal{C}}$. By $\deg(P|_{\widetilde{\mathcal{C}}},\widetilde{\mathcal{C}},\s^1) = 0$ it follows that $\widetilde{\mathcal{C}}$ can continuously contract into the single point $p_1$. This shows that $A_1$ can continuously contract into a single point, completing the proof.
\end{proof}

\begin{lemma} \label{lemtopo2}
	Let $\mathcal{C}$ be a continuous closed curve contained in $(0,1) \times \s^1$ satisfying $\deg(P|_{\mathcal{C}},\mathcal{C},\s^1) \neq 0$. There exist $t \in (0,1)$ and $s \in \s^1$ such that $(t,s) \in \mathcal{C}$ and $(t,-s) \in \mathcal{C}$.
\end{lemma}

\begin{proof}
	Let $\mathcal{C}$ be parameterized by $\ep \in [-1,1]$, that is $\mathcal{C} = \{(t(\ep),s(\ep)): \ep \in [-1,1]\}$ for some $(t(\ep),s(\ep)) \in C([-1,1],(0,1) \times \s^1)$ satisfying $(t(1),s(1)) = (t(-1),s(-1))$. Allowing $(t(\ep_1),s(\ep_1)) = (t(\ep_2),s(\ep_2))$ for $\ep_1 \neq \ep_2$, we can assume that $s(\ep-1) = -s(\ep)$ for $\ep \in [0,1]$ since $\deg(P|_{\mathcal{C}},\mathcal{C},\s^1) \neq 0$. We consider the function
	\begin{align*}
		g(\ep) = t(\ep) - t(\ep-1), \quad \ep \in [0,1].
	\end{align*}
	Without loss of generality, assuming $g(0) \le 0$, we have
	\begin{align*}
		g(1) = t(1) - t(0) = -(t(0) - t(-1)) = -g(0) \ge 0.
	\end{align*}
	By continuity, there exists $\ep_0 \in [0,1]$ such that $g(\ep_0) = 0$, that is $t(\ep_0) = t(\ep_0-1)$. Recalling $s(\ep_0-1) = -s(\ep_0)$, the pair of points $(t(\ep_0),s(\ep_0))$ and $(t(\ep_0-1),s(\ep_0-1))$ satisfies our requirement. The proof is complete.
\end{proof}

\begin{proof}[Proof of Proposition \ref{propc1=c2}]
	By Remark \ref{rmkalmost} we assume that $(c_\tau^2)'$ exists. Let $\{\tau_n\}\subset I$ be a monotone increasing sequence converging $\tau$ and $\varepsilon_n = (2-(c_\tau^2)')(\tau - \tau_n) \to 0^+$. Then, using the so-called "monotonicity trick", see details in \cite{Louis} or in the proof of \cite[Theorem 1.10]{BCJN}, we obtain a fixed $R = R((c_\tau^2)') > 0$ and $\{\gamma_n\} \subset \Gamma_2$ such that
	\begin{itemize}
		\item[(i)] $\displaystyle\max_{(t,s)\in [0,1]\times\s^1}\Phi_\tau(\gamma_n(t,s)) \leq c + \varepsilon_n$,
		\item[(ii)] $\gamma_n(t,s) \in B(0,R-1)$ whenever
		\begin{align*}
			\Phi_\tau(\gamma_n(t,s)) \geq c - \varepsilon_n, \quad (t,s) \in [0,1] \times \s^{1}.
		\end{align*}
	\end{itemize}
    In view of \eqref{eq6.5}, Lemma \ref{lemtopo1} and Lemma \ref{lemtopo2}, the set
    \begin{align*}
    	\bigl\{(t,s) \in [0,1]\times\s^1: \ga_n(t,s) \notin F^* \cup (-F^*)\bigr\}
    \end{align*}
     is path-connected. Thus, using \eqref{SMP2}, with $n$ sufficiently large we can find $\hat \gamma_n \in \Ga_1$ such that
     \begin{itemize}
     	\item[(a)] $\displaystyle\max_{(t,s)\in [0,1]\times\s^0}\Phi_\tau(\hat\gamma_n(t,s)) \leq c + \varepsilon_n$,
     	\item[(b)] $\hat\gamma_n(t,s) \in B(0,R-1)$ whenever
     	\begin{align*}
     		\Phi_\tau(\hat\gamma_n(t,s)) \geq c - \varepsilon_n, \quad (t,s) \in [0,1] \times \s^{0},
     	\end{align*}
        \item[(c)] $\hat \ga_n(t,s) \notin F^* \cup (-F^*)$ whenever
        \begin{align*}
        	\Phi_\tau(\hat\gamma_n(t,s)) \geq c - \varepsilon_n, \quad (t,s) \in [0,1] \times \s^{0}.
        \end{align*}
     \end{itemize}
     Then using \cite[Theorem 4.2]{BCJN}, we obtain $u_{n} \in S_\mu$ such that
     \begin{itemize}
     	\item[(1)]  $c - \varepsilon_n \le \Phi_\tau(u_{n}) \le c + \varepsilon_n$;
     	\item[(2)] $\|\nabla \Phi_\tau|_{S_\mu}(u_n)\| \leq 3\varepsilon_n^{\alpha_1}$ where $ 0 < \alpha_1 \le \frac{\alpha}{2(\alpha+2)} < 1$;
     	\item[(3)] $u_n \in \bigl\{\hat \gamma_n(t,s): (t,s) \in [0,1] \times \s^0\bigr\}$;
     	\item[(4)] If $D^2\Phi_\tau(u_n)[w, w] < -\varepsilon_n^{\alpha_1}\|w\|^2$ for all $w \neq 0$ in a subspace $W$ of $T_{u_n}S_\mu$, then
     	$\dim W \le 1$.
     \end{itemize}
     In view of (c) and (3) we conclude that $u_n \in S_\mu \backslash (-F^* \cup F^*)$. Recalling that $\varepsilon_n = (2-(c_\tau^2)')(\tau - \tau_n) \to 0^+$ we complete the proof.
\end{proof}

\section{Saddle-point type sign-changing solutions} \label{secappropro}

\subsection{Approximating problems with a parameter in the nonlinear term}

As already anticipated in the introduction, we can construct minimax structures by generalizing the mountain pass structure and obtain Palais-Smale sequences. However, in the present setting the existence of a bounded Palais-Smale sequence at the minimax level, is not straightforward. To overcome this issue, we follow the framework in \cite{BCJN} and introduce the family of functionals
\begin{align*}
	E_\tau(u,\Omega) = \frac12\int_\Omega|\nabla u|^2dx - \frac \tau p\int_\Omega|u|^pdx,
\end{align*}
depending on the parameter $\tau \in [1/2,1]$. The main idea is to apply Theorem \ref{Objective1}. In the following lemma we show that there exists the uniformly minimax geometry.

\begin{lemma} \label{lemsadpoi}
	Suppose $p_c < p < 2^*$. Let $k$ be a positive integer and let
	\begin{align} \label{eqassumonmu}
		\mu \in \bigl\{\mu > 0: \exists \rho > 0 \quad \text{such that} \quad \la_k(\Omega)\mu + \frac2p C_{p,N}^p \mu^{\frac{p(1-\ga_p)}{2}}\rho ^{\frac{p\ga_p}{2}} < \rho\bigr\}.
	\end{align}
	Then there exist two odd functions $\gamma_{i,k}: \s^{k-1} \to S_\mu(\Omega)$ where $i = 0,1$, such that the set
	\begin{align}
		\Gamma_k := \big\{\gamma \in C([0,1]\times \s^{k-1}, S_\mu(\Omega)): \forall t \in [0,1], \gamma(t,\cdot) \text{ is odd}, \gamma(0,\cdot) = \gamma_{0,k}, \text{ and } \gamma(1,\cdot) = \gamma_{1,k}\big\}
	\end{align}
    is not empty and is independent of $\tau$, and
    \begin{align} \label{eqctauN}
    	c_\tau^k := \inf_{\gamma \in \Gamma_k}\sup_{(t,s) \in [0,1]\times\s^{k-1}}E_\tau(\gamma(t,s),\Omega) > \max_{s\in \s^{k-1}}\max\big\{ E_\tau(\gamma_{0,k}(s),\Omega), E_\tau(\gamma_{1,k}(s),\Omega)\big\}.
    \end{align}
    Further, we have $c^k_\tau \to c^k_1$ as $\tau \to 1^-$.
\end{lemma}

\begin{proof}
	Let us define
	\begin{align*}
		\gamma_{0,k}: \s^{k-1} \to S_\mu(\Omega), \quad  (a_1,a_2,\cdots,a_N) \mapsto \sum_{i=1}^ka_i\sqrt{\mu}\phi_i.
	\end{align*}
	Note that $\gamma_{0,k}(\s^{k-1})$ is exactly the set $\mathbb{M}_k$ given in \eqref{eqdefMk}. If $\rho > \la_k(\Omega)\mu$ we have $\gamma_{0,k}(\s^{k-1}) \subset \B_\rho^\mu$. Let $\rho$ satisfy the inequality in \eqref{eqassumonmu}. By H\"older inequality, the assumption \eqref{eqassumonmu} and Gagliardo-Nirenberg inequality \eqref{eqgndomain}, for all $\tau \in [1/2,1]$, we get
	\begin{align} \label{eqmaxsn-1}
		\max_{s \in \s^{k-1}}E_\tau(\gamma_{0,k}(s),\Omega)  \leq \frac12 \la_k(\Omega)\mu < \frac12 \rho - \frac1p C_{p,N}^p \mu^{\frac{p(1-\ga_p)}{2}}\rho ^{\frac{p\ga_p}{2}} \leq \inf_{u \in \U_\rho^\mu}E_1(u,\Omega) \le \inf_{u \in \U_\rho^\mu}E_\tau(u,\Omega).
	\end{align}
	
	Let us take $k$ non-zero $H_0^1$ functions $v_1, v_2, \cdots, v_k$ such that $\text{supp} \ v_i \cap \text{supp} \ v_j = \emptyset$ and so $v_i$ is orthogonal to $v_j$ both in $H_0^1$ and $L^2$ for $i \neq k$. Denoting $v_{i,t}(x) := t^{N/2}v_i(tx)$ with $t > 1$. For any $i \in \{1,2,\cdots,k\}$, we can check that $v_{i,t} \in S_\mu(\Omega)$, and that
	\begin{align*}
		\int_\Omega|\nabla v_{i,t}|^2dx = t^2\int_\Omega|\nabla v_i|^2dx.
	\end{align*}
	So $\int_\Omega|\nabla v_{i,t}|^2dx > \rho$ and $v_{i,t} \notin \overline{\B_\rho^\mu}$ for $t$ sufficiently large (independent of $\tau$). Moreover, for $u = \sum_{i=1}^k a_iv_i$ with $(a_1,\cdots,a_N) \in \s^{k-1}$, we have $u_t = \sum_{i=1}^k a_iv_{i,t}$ where $u_t = t^{N/2}u(tx)$ and it holds that $u_t \in S_\mu(\Omega)$ and
	\begin{align*}
		E_\tau(u_t,\Omega) = \frac {t^2}2\int_\Omega|\nabla u|^2dx - \frac{\tau t^{\frac{(p-2)N}2}}p\int_\Omega|u|^pdx.
	\end{align*}
	Note that
	\begin{align*}
		& \int_\Omega|\nabla u|^2dx = \sum_{i=1}^k a_i^2 \int_\Omega|\nabla v_i|^2dx \leq \max_i \int_\Omega|\nabla v_i|^2dx, \\
		& \int_\Omega|u|^pdx = \sum_{i=1}^k a_i^p \int_\Omega|v_i|^pdx \geq k^{\frac{2-p}{p}}\min_i \int_\Omega|v_i|^pdx.
	\end{align*}
	Then we get
	\begin{align*}
		E_\tau(u_t,\Omega) \le \frac {t^2}2\max_i \int_\Omega|\nabla v_i|^2dx - \frac{ t^{\frac{(p-2)N}2}}{2p}k^{\frac{2-p}{p}}\min_i \int_\Omega|v_i|^pdx,
	\end{align*}
    and since $p > p_c$,
	\begin{align*}
		E_\tau(u_t,\Omega) < \max_{s \in \s^{k-1}}E_1(\gamma_{0,k}(s),\Omega) \leq \max_{s \in \s^{k-1}}E_\tau(\gamma_{0,k}(s),\Omega), \quad \tau \in [\frac12,1],
	\end{align*}
	for $t$ sufficiently large independent of $\tau$ and of $(a_1,\cdots,a_N) \in \s^{k-1}$. Now we take $w_i = v_{i,t}$ with $t$ sufficiently large and independent of $\tau$. Without loss of generality, we also assume that $w_i$ is non-negative in $\Omega$. Define
	\begin{align*}
		\gamma_{1,k}: \s^{k-1} \to S_\mu(\Omega), \quad (a_1,a_2,\cdots,a_N) \mapsto \sum_{i=1}^ka_iw_i.
	\end{align*}
	Then any $u \in \gamma_{1,k}(\s^{k-1})$ does not belong to $\B_{\rho}^\mu$ and
	\begin{align} \label{eqmaxsn-1two}
		\max_{s \in \s^{k-1}}E_\tau(\gamma_{1,k}(s),\Omega) < \max_{s \in \s^{k-1}}E_\tau(\gamma_{0,k}(s),\Omega).
	\end{align}
	
	We define the subspace $V_k=\text{span}\{\sqrt{\mu}\phi_1,\sqrt{\mu}\phi_2,\cdots, \sqrt{\mu}\phi_k;w_1, w_2,\cdots,w_k\}$. Note that $d = \dim(V_k) \leq 2k$. Let $R$ be an operator (in $L^2(\Omega)$) such that $R=I$ on $V_k^{\perp}$, $R(\sqrt{\mu}\phi_i)= w_i$, $i=1,2,\cdots ,k$. Possibly after permutations, we can choose $R$ such that $R\big |_{V_k}\in SO(d)$ (actually, there are infinitely many different choices of $R$). Now, since $SO(d)$ is connected, there is a continuous path $\tilde{\gamma}:\,[0,1]\rightarrow SO(d)$ such that $\tilde \gamma(0)=I$, $\tilde \gamma(1)=R\big |_{V_k}$. Let us define the map
	\begin{align*}
		\gamma: [0,1]\times \s^{k-1} \to S_\mu(\Omega), \quad (t,a_1,a_2,\cdots,a_k) \mapsto \sum_{i=1}^ka_i\tilde \gamma(t).
	\end{align*}
	It is clear that $\gamma$ is continuous, $\gamma(t,\cdot)$ is odd for all $t \in [0,1]$, and $\gamma(0,\cdot) = \gamma_{0,k}(\cdot)$, $\gamma(1,\cdot) = \gamma_{1,k}(\cdot)$. Hence, $\Gamma_k$ is not empty.
	
	For any $s \in \s^{k-1}$, we have $\gamma_{0,k}(s) \in \B^\mu_{\rho}$ and $\gamma_{1,k}(s) \notin \B^\mu_{\rho}$. By continuity, for any $\gamma \in \Gamma_k$, there exists $(t',s') \in [0,1]\times\s^{k-1}$ such that $\gamma(t',s') \in \U_{\rho}^\mu$. Hence,
	\begin{align*}
		\max_{(t,s) \in [0,1]\times\s^{N-1}}E_\tau(\gamma(t,s),\Omega) \geq E_\tau(\gamma(t',s'),\Omega) \geq \inf_{u \in \U_{\rho}^\mu}E_\tau(u,\Omega),
	\end{align*}
	and so, combining \eqref{eqmaxsn-1} and \eqref{eqmaxsn-1two} we conclude
	\begin{align*}
		c_\tau^k \geq \inf_{u \in \U_{\rho}^\mu}E_\tau(u,\Omega) > \max_{s \in \s^{N-1}}E_\tau(\gamma_{0,k}(s),\Omega) > \max_{s \in \s^{N-1}}E_\tau(\gamma_{1,k}(s),\Omega).
	\end{align*}
	
	Finally, we prove $c^k_\tau \to c^k_1$ as $\tau \to 1^-$. On the one hand, for any $u \in S_\mu(\Omega)$ and $\tau \in [1/2,1]$, we have $E_\tau(u,\Omega) \geq E_1(u,\Omega)$. This shows that $\liminf_{\tau \to 1^-}c^k_\tau \geq c^k_1$. On the other hand, it is well known that $\limsup_{\tau \to 1^-}c^k_\tau \leq c^k_1$ see, e.g. \cite{Louis}. Hence, we have $c^k_\tau \to c^k_1$ as $\tau \to 1^-$
	and complete the proof.
\end{proof}

\begin{remark} \label{rmklowerbound}
	It seems that the choice of $\ga_{1,k}$ in the proof of Lemma \ref{lemsadpoi}, and so the set $\Ga_k$, depends on $\rho$. However, for $\mu$ satisfying \eqref{eqassumonmu}, we have
	\begin{align*}
		\la_k(\Omega) \mu < \rho < \left(\frac p {2C_{p,N}^p}  \right)^{\frac{2}{p\ga_p-2}}\mu^{-\frac{p(1-\ga_p)}{p\ga_p-2}}
	\end{align*}
    since $p\ga_p > 2$. Thus, we can take $\ga_{1,k}(\s^{k-1})$ outside $\B_{g(\mu)}^\mu$ where
    $$
    g(\mu) = \left(\frac p {2C_{p,N}^p}  \right)^{\frac{2}{p\ga_p-2}}\mu^{-\frac{p(1-\ga_p)}{p\ga_p-2}},
    $$
    independent of the choice of $\rho$. Moreover, we provide a lower bound of $c_\tau^k$ independent of $\rho$. We define
    \begin{align*}
    	\Theta(\mu) := \bigl\{\rho > 0: \la_k(\Omega)\mu + \frac2p C_{p,N}^p \mu^{\frac{p(1-\ga_p)}{2}}\rho ^{\frac{p\ga_p}{2}} < \rho\bigr\}.
    \end{align*}
    We observe that the assumption \eqref{eqassumonmu} holds if and only if $ \Theta(\mu) \neq \emptyset$. Using \eqref{eqmaxsn-1} in the proof of Lemma \ref{lemsadpoi}, for all $\tau \in [1/2,1]$ we have
    \begin{align} \label{eqctooo}
    	c_\tau^k \geq \sup_{\rho \in \Theta(\mu)}\left( \frac12 \rho - \frac1p C_{p,N}^p \mu^{\frac{p(1-\ga_p)}{2}}\rho ^{\frac{p\ga_p}{2}}\right).
    \end{align}
     This estimate is useful to show that $c_\tau^k \to +\infty$ as $\mu \to 0^+$.
\end{remark}

Now we can apply Theorem \ref{Objective1} and Proposition \ref{propc1=c2} on the approximating problems.

\begin{proposition} \label{propaetau}
	Under the assumptions of Lemma \ref{lemsadpoi}, for almost every $\tau \in [1/2,1]$, there exists a critical point $u_\tau^k$ of $E_\tau(\cdot,\Omega)$ constrained to $S_\mu(\Omega)$ at level $c_\tau^k$, which solves
	\begin{align} \label{eqequoftau}
		\begin{cases}
			-\Delta u_\tau^k + \la_\tau^k u_\tau^k = \tau |u_\tau^k|^{p-2}u_\tau^k & \text{in } \Omega, \\[1.5\jot]
			\displaystyle
			u_\tau^k(x) = 0 & \text{ on } \partial \Omega,
		\end{cases}
	\end{align}
	for some $\la_\tau^k \geq -C_k$ where $C_k > 0$ is a constant depending on $k$ and independent of $\tau$. The Morse index $m(u_\tau^k) \leq k+2$. Moreover, $u_\tau^1$ is positive in $\Omega$ and $m(u_\tau^1) \leq 2$; and if $c_\tau^1 = c_\tau^2$, $u_\tau^2$ can be chosen as a sign-changing function with $m(u_\tau^2) \leq 2$. As $\mu \to 0^+$, we have $c_\tau^k \to +\infty$ uniformly with respect to $\tau$.
\end{proposition}

\begin{proof}
    We apply Theorem \ref{Objective1} to the family of functionals $E_\tau(\cdot,\Omega)$, with $E = H_0^1(\Omega)$, $H = L^2(\Omega)$, and $\Gamma_k$ defined in Lemma \ref{lemsadpoi}. Setting
	\begin{align*}
		A(u) = \frac12 \int_\Omega|\nabla u|^2dx \quad \text{and} \quad B(u) = \frac1p\int_\Omega|u|^pdx,
	\end{align*}
    assumption \eqref{hp coer} holds, since we have
    \begin{align*}
    	u \in S_\mu(\Omega), \|u\| \to +\infty \quad \Rightarrow \quad A(u) \to +\infty.
    \end{align*}
    Moreover, the unconstrained first and second derivatives of $E_\tau$ are H\"older continuous on bounded sets of $S_\mu(\Omega)$. In this way, together with Lemma \ref{lemsadpoi}, for almost every $\tau \in [1/2,1]$ there exists a bounded Palais-Smale sequence $\{u_n\} \subset S_\mu(\Omega)$ for the constrained functional $E_\tau(\cdot,\Omega)|_{S_\mu(\Omega)}$ at level $c_\tau^k$, and $\zeta_n \to 0^+$, such that $\tilde m_{\zeta_n}(u_n) \leq k+1$. By Lemma \ref{lempscondition}, we have $u_n \to u_\tau^k$ strongly in $H_0^1$, and $u_\tau^k$ is a constrained critical point, thus a solution to \eqref{eqequoftau} for some Lagrange multiplier $\la_\tau^k$. By Lemma \ref{lemmorse} we have $m(u_\tau^k) \leq k+2$. The existence of $C_k > 0$ depending on $k$ and independent of $\tau \in [1/2,1]$ such that $\la_\tau^k \geq -C_k$ can be deduced immediately from the uniform boundedness of Morse index of $u_\tau^k$, by using arguments in the proof of Lemma \ref{lemlanbelow}. Particularly, when $k = 1$, the geometry presented in Lemma \ref{lemsadpoi} degenerates to the mountain pass geometry and the positivity of $u_\tau^1$ and $m(u_\tau^1) \leq 2$ can be proved using arguments in \cite{CJS}. If $c_\tau^1 = c_\tau^2$, we use Proposition \ref{propc1=c2} with $F^* = \overline{\p_\delta}$, where $\delta > 0$ satisfies \eqref{eqconondelta0}. In view of Lemma \ref{lemcapnonemp}, we know \eqref{eq6.5} holds true. Thus we obtian a bounded Palais-Smale sequence $\{u_n\} \subset S_\mu(\Omega) \backslash (\overline{\p_\delta} \cup (\overline{-\p_\delta}))$ for the constrained functional $E_\tau(\cdot,\Omega)|_{S_\mu(\Omega)}$ at level $c_\tau^2$, and $\zeta_n \to 0^+$, such that $\tilde m_{\zeta_n}(u_n) \leq 1$. Then we can get a constrained critical point $u_\tau^2 \in S_\mu(\Omega) \backslash (\p_\delta \cup (-\p_\delta))$, yielding that $u_\tau^2$ is sign-changing, with $m(u_\tau^2) \leq 2$.

    Finally, we study the asymptotic behavior of $c_\tau^k$ as $\mu \to 0^+$. Let
    \begin{align*}
    	f(\rho) = \frac12 \rho - \frac1p C_{p,N}^p \mu^{\frac{p(1-\ga_p)}{2}}\rho ^{\frac{p\ga_p}{2}}, \quad \rho > 0.
    \end{align*}
    Computing directly we get that
    \begin{align*}
    	f'(\rho) = \frac12 - \frac{\ga_p}2 C_{p,N}^p \mu^{\frac{p(1-\ga_p)}{2}}\rho ^{\frac{p\ga_p}{2}-1},
    \end{align*}
    and
    \begin{align*}
    	f(\rho_\mu) = \left( \frac12 - \frac{1}{p\ga_p}\right) \rho_\mu = \max_{\rho > 0}f(\rho),
    \end{align*}
    where
    \begin{align*}
    	\rho_\mu = \left( \frac{1}{\ga_pC_{p,N}^p}\right) ^{\frac{2}{p\ga_p-2}}\mu^{-\frac{p(1-\ga_p)}{p\ga_p-2}}.
    \end{align*}
    Since $\ga_p < 1$ and $p\ga_p > 2$, for $\mu > 0$ sufficiently small we have
    \begin{align*}
    	\la_k(\Omega)\mu < \left( 1 - \frac{2}{p\ga_p}\right) \rho_\mu,
    \end{align*}
    showing that $\Theta(\mu) \neq \emptyset$, where $\Theta(\mu)$ is introduced in Remark \ref{rmklowerbound}. By \eqref{eqctooo}, we obtain that, as $\mu \to 0^+$,
    \begin{align*}
        c_\tau^k \geq \sup_{\rho \in \Theta(\mu)}f(\rho) = f(\rho_\mu) = \left( \frac12 - \frac{1}{p\ga_p}\right)\left( \frac{1}{\ga_pC_{p,N}^p}\right) ^{\frac{2}{p\ga_p-2}}\mu^{-\frac{p(1-\ga_p)}{p\ga_p-2}} \to +\infty.
    \end{align*}
    The proof is complete.
\end{proof}

\subsection{Proof of Theorem \ref{thmlarge}}

As a consequence of Proposition \ref{proprelationship}, we have the following result which is prepared to prove Theorem \ref{thmlarge}.

\begin{proposition} \label{proplabounded}
	Let $p_c < p < 2^*$ and $\{u_n\} \subset H_0^1(\Omega)$ be a sequence of solutions to
	\begin{align} \label{eq9.8}
		\begin{cases}
			-\Delta u_n + \la_n u_n = \tau_n u_n^{p-1} & \text{in } \Omega, \\[1.5\jot]
			u(x) = 0 & \text{ on } \partial \Omega,
		\end{cases}
	\end{align}
	where $\tau_n \to 1$ and $\la_n \in \R$. Suppose that
	\begin{align*}
		\int_\Omega|u_n|^2dx = \mu, \quad m(u_n) \leq \bar{k}, \quad \forall n \in \N,
	\end{align*}
	for some $\mu > 0$ and $\bar{k} \in \N$, and that
	\begin{align}
		\text{the sequence of the energy levels } \{E_{\tau_n}(u_n,\Omega)\} \text{ is bounded}.
	\end{align}
	Then the sequences $\{\la_n\} \subset \R$ and $\{u_n\} \subset H_0^1(\Omega)$ must be bounded. In addition, $\{u_n\}$ is a bounded Palais-Smale sequence for $E_1(\cdot,\Omega)$ constrained on $S_\mu(\Omega)$.
\end{proposition}

\begin{proof}
	By Proposition \ref{proprelationship}, the sequence  $\{\la_n\}$ is bounded. Then using Lemma \ref{lemlanbounded} we know $\{u_n\}$ is bounded in $L^\infty$ and so in $L^p$. By \eqref{eq9.8} we get that $\{u_n\}$ is bounded in $H_0^1$ immediately and complete the proof.
\end{proof}

Finally we complete the proof of Theorem \ref{thmlarge}.

\begin{proof}[Proof of Theorem \ref{thmlarge}]
	Recall that
	\begin{align*}
		\Theta(\mu) := \bigl\{\rho > 0: \la_2(\Omega)\mu + \frac2p C_{p,N}^p \mu^{\frac{p(1-\ga_p)}{2}}\rho ^{\frac{p\ga_p}{2}} < \rho\bigr\}
	\end{align*}
	is given in Remark \ref{rmklowerbound} and we take $k =2$ here. Firstly, we prove that $\Theta(\mu) \neq \emptyset$ is equivalent to $0 < \mu < \mu_2$ where
	\begin{align} \label{eqmu2}
		\mu_2 := \sup \bigl\{\mu > 0: \exists \rho > 0 \quad \text{such that} \quad \la_2(\Omega)\mu + \frac2p C_{p,N}^p \mu^{\frac{p(1-\ga_p)}{2}}\rho ^{\frac{p\ga_p}{2}} < \rho\bigr\}.
	\end{align}
	From $p\ga_p > 2$ we see that $\mu_2 < +\infty$. Moreover, for any $\rho >0$ we have
	\begin{align*}
		\la_2(\Omega)\mu_2 + \frac2p C_{p,N}^p \mu_2^{\frac{p(1-\ga_p)}{2}}\rho ^{\frac{p\ga_p}{2}} \geq \rho
	\end{align*}
	by using the definition of $\mu_2$. Thus $\Theta(\mu_2) = \emptyset$. Next it suffices to show that for $\mu \in (0,\mu_2)$ it holds that $\Theta(\mu) \neq \emptyset$. By the definition of $\mu_2$, there exist $\omega_n < \mu_2$ and $\rho_n \in \Theta(\omega_n)$ such that $\omega_n \to \mu_2$. Since $p\ga_p > 2$ we get that $\rho_n$ is bounded, and up to a subsequence, we assume that $\rho_n \to \rho_0$ for some $\rho_0 > 0$. Therefore,
	\begin{align*}
		\la_2(\Omega)\mu_2 + \frac2p C_{p,N}^p \mu_2^{\frac{p(1-\ga_p)}{2}}\rho_0 ^{\frac{p\ga_p}{2}} = \rho_0.
	\end{align*}
	Then, for $\mu < \mu_2$ we have
	\begin{align*}
		\la_2(\Omega)\mu + \frac2p C_{p,N}^p \mu^{\frac{p(1-\ga_p)}{2}}\rho_0 ^{\frac{p\ga_p}{2}} < \rho_0,
	\end{align*}
	yielding that $\Theta(\mu) \neq \emptyset$.

	Let $0 < \mu < \mu_2$. By Proposition \ref{propaetau}, for almost every $\tau \in [1/2,1]$, there exist critical points $u_\tau^1$ and $u_\tau^2$ of $E_\tau(\cdot,\Omega)$ constrained to $S_\mu$ at levels $c_\tau^1$ and $c_\tau^2$ respectively, which solves \eqref{eqequationtau} for some $\la = \la_\tau^k \geq -C_k$, $k =1,2$. The Morse index $m(u_\tau^1) \leq 2$ and $m(u_\tau^2) \leq 4$. Moreover, $u_\tau^1$ is positive; and if $c_\tau^1 = c_\tau^2$, $u_\tau^2$ can be chosen as a sign-changing function with $m(u_\tau^2) \leq 2$. From now on we emphasize that $\Omega = B_1$ in this proof. As $\mu \to 0^+$, we have $c_\tau^1, c_\tau^2 \to +\infty$ uniformly with respect to $\tau$. If $c_\tau^1 < c_\tau^2$, we have $u_\tau^1 \neq u_\tau^2$. By Theorem \ref{thmunique}, there exists $\mu^{**}_p \in (0,\min\{\mu_2,\mu^*_{p,0}\})$ such that for $0 < \mu < \mu^{**}_p$, the function $u_\tau^2$ is sign-changing. From above discussion we always assume that $0 < \mu < \mu^{**}_p$ and $u_\tau^2$ is sign-changing with $m(u_\tau^2) \le 4$. Observer that $c_1^2 \le E_\tau(u_\tau^2) = c_\tau^2 \le c_{1/2}^2$. Hence, we can take a sequence $\tau_n \to 1^-$, with a corresponding sequence $\{u_n\} = \{u_{\tau_n}\}$ whose energy is uniformly bounded with respect to $n$. By Proposition \ref{proplabounded}, $\{u_n\}$ is a bounded Palais-Smale sequence of $E_1$ constrained to $S_\mu$. Up to a subsequence, there exists $u_{\text{sc},\mu} \in S_\mu$ such that $u_n \to u_{\text{sc},\mu}$ strongly in $H_0^1(B_1)$, and thus $u_{\text{sc},\mu}$ is a constrained critical point of $E_1$ constrained to $S_\mu$. By Lagrange multiplier principle, $u_{\text{sc},\mu}$ solves \eqref{eqonbounddomain} for a Lagrange multiplier $\la=\la_{\text{sc},\mu}$. It is relatively standard to see that $m(u_{\text{sc},\mu}) \leq 4$.
	
	Now we prove that $u_{\text{sc},\mu}$ is indeed sign-changing. Take $\rho > 0$ such that $\{u_n\} = \{u_{\tau_n}\} \subset \B_\rho^\mu$. Let $V_\tau$ be defined in \eqref{def:V}. It is clear that $V_{\tau_n}(u_n) = 0$. By Lemma \ref{lemproG4}, there exists $\delta > 0$ depending on $\rho$ and on $\mu$ but independent of $n$ such that $u_n \notin D^*_\delta$, that is $u_n \in S^*(\delta)$. Thus $u_{\text{sc},\mu} \in \overline{S^*(\delta)}$, showing that $u_{\text{sc},\mu}$ is sign-changing.
	
	Recall that
	\begin{align} \label{eqelarge}
		E(u_{\text{sc},\mu}) = c_1^2 \to +\infty \quad \text{as } \mu \to 0^+.
	\end{align}
	To complete the proof, it needs to prove that $\la_{\text{sc},\mu} \to +\infty$ as $\mu \to 0^+$. Observe that $m(u_{\text{sc},\mu}) \leq 4$, and, by Lemma \ref{lemlanbelow} and by Proposition \ref{proprelationship} we complete the proof.
\end{proof}

\medskip

\vskip0.3in

{\small \noindent \textit{\bf Statements and Declarations:} The authors have no relevant financial or non-financial interests to disclose.}

{\small \noindent \textit{\bf Data availability:} Data sharing is not applicable to this article as no datasets were generated or analysed during the current study.}

{\small \noindent \textit{\bf Acknowledgements:} This work was completed when the first author was visiting Universit\'e Marie et Louis Pasteur, LmB (UMR 6623). Song is funded by China Postdoctoral Science Foundation (2024T170452) and Zou is funded by National Key R\&D Program of China (Grant 2023YFA1010001) and NSFC (12571123). The authors thank Louis Jeanjean for reading carefully and providing useful suggestions to help improve the article.}

\end{document}